\documentclass[reqno]{amsart}

\usepackage[utf8]{inputenc}
\usepackage[T1]{fontenc}
\usepackage{amsmath, amssymb, amsthm}
\usepackage[english]{babel}
\usepackage{csquotes}
\usepackage{hyperref}
\usepackage{enumitem}

\numberwithin{equation}{section}

\newtheorem{thm}{Theorem}[section]
\newtheorem{pro}[thm]{Proposition}
\newtheorem{lem}[thm]{Lemma}
\newtheorem{cor}[thm]{Corollary}

\theoremstyle{definition}

\newtheorem{rem}[thm]{Remark}
\newcommand{\supp}{\mathord{\mathrm{supp}}\,}
\newcommand{\spx}{\textbf{X}}

\newcommand{\spec}{\mathfrak{E}}
\newcommand{\spfa}{\mathbf{F}}
\newcommand{\spfb}{\mathbf{F}}
\newcommand{\spfc}{\mathfrak{F}}
\newcommand{\spna}{\mathbf{N}}
\newcommand{\spnb}{\mathbf{N}}
\newcommand{\spnc}{\mathfrak{N}}
\newcommand{\spbc}{\mathfrak{B}}

\begin{document}
	
	\title{Local well-posedness of the Benjamin-Ono equation for a class of bounded initial data}
	\author{Niklas J\"ockel}
	\address{Fakult\"at f\"ur Mathematik, Universit\"at Bielefeld, Postfach 100131, 33501 Bielefeld, Germany}
	\email{njoeckel@math.uni-bielefeld.de}

	\begin{abstract}
		We prove local well-posedness of the Benjamin-Ono equation for a class of bounded initial data including periodic and bore-like functions. As a consequence, we obtain local well-posedness in $H^s(\mathbb{R})+H^\sigma(\mathbb{T})$ for $s>\frac{1}{2}$ and $\sigma>\frac{7}{2}$. These results follow by studying a generalized forced Benjamin-Ono equation. 
	\end{abstract}
	\maketitle
	\section{Introduction}
	We consider the initial-value problem of the real-valued Benjamin-Ono equation
	\begin{equation}\label{eq_bo}
		\begin{cases}
			\phi_t
			+
			\mathcal{H}\phi_{xx}
			+
			(\phi\phi)_x
			&=
			0
			\\
			\phi(0,\cdot)
			&=
			\phi_0
		\end{cases}
	\end{equation}
	posed on $\mathbb{R}_t\times\mathbb{R}_x$. Above, $\mathcal{H}$ denotes the Hilbert transform and $\phi_0$ is the initial datum. This equation has been derived by Benjamin \cite{Ben1967} and Acrivos and Davis \cite{DA1967} to describe long internal waves in stratified fluids; see also Ono \cite{Ono1975}.
	
	In the last two decades a lot of attention has been given to the question in which Sobolev spaces $H^s(\mathbb{R})$ and $H^\sigma(\mathbb{T})$ the initial value problem \eqref{eq_bo} is (locally) well-posed. This question is now essentially answered. In \cite{GKT2023}, Gérard, Kappeler, and Topalov proved the optimal well-posedness result for periodic initial data -- one has global well-posedness for $\sigma>-\frac{1}{2}$ and ill-posedness for $\sigma\leq -\frac{1}{2}$; see also \cite{GT2023}. Moreover, recently, Killip, Laurens, and Vi\c{s}an \cite{KLV2023} established global well-posedness for spatially decaying (as well as periodic) initial data in the range $s>-\frac{1}{2}$ (and $\sigma>-\frac{1}{2}$). For an ill-posedness result for $s<-\frac{1}{2}$, we refer to \cite{BL2001}.
	
	In this work, we will consider the Benjamin-Ono equation with initial datum $\phi_0$ which is bounded, but neither spatially decaying nor periodic. This includes initial data in $H^s(\mathbb{R})+H^\sigma(\mathbb{T})$ as well as bore-like initial data such as $\phi_0=\tanh$. Equation \eqref{eq_bo} has already been studied for bore-like initial data by Iorio, Linares and Scialom \cite{ILS1998} and for Zhidkov functions as initial data by Gallo \cite{Gal2005}; we recall their results in Section \ref{ss_rol}. For similar statements concerning the Schrödinger, the Korteweg-de Vries, and the Zakharov-Kuznetsov equations, we refer to the recent works \cite{CHKP2019,Lau2022,Lau2023,Pal2022,Pal2023} and the references mentioned therein.
	
	In order to obtain local well-posedness results for \eqref{eq_bo}, we first study the following initial-value problem
	\begin{equation}\label{eq_bo_split}
		\begin{cases}
			u_t
			+
			\mathcal{H}u_{xx}
			+
			(uu)_x
			+
			(ub)_x
			+
			f
			&=
			0
			\\
			u(0)
			&=
			u_0
		\end{cases}
	\end{equation}
	with initial datum $u_0\in H^s(\mathbb{R})$. This system can be derived from \eqref{eq_bo} by splitting both the solution and the initial datum into a sum of a spatially decaying function and a spatially bounded function, i.e.\@ we assume that $\phi= u+b$ and $\phi_0=u_0+b_0$ hold. Then, \eqref{eq_bo_split} follows from \eqref{eq_bo} by choosing $f=f_b:=b_t+\mathcal{H}b_{xx} + (bb)_x$.
	
	Let us state the main theorem of this work establishing local well-posedness for \eqref{eq_bo_split}. For this, we need the following assumptions on $b$ and $f$ for some arbitrary small $\varepsilon>0$:
	\begin{align}\label{as_b_f}
		\begin{cases}
			b\in L^\infty([0,1];B^{3+\varepsilon}_{\infty,\infty}(\mathbb{R}))\cap C^0([0,1];C^2(\mathbb{R}))
			&=:
			\mathcal{X}
			\\
			f\in L^\infty([0,1];H^{3+\varepsilon}(\mathbb{R}))\cap C^0([0,1];H^0(\mathbb{R}))
			&=:
			\mathcal{Y}
		\end{cases}
	\end{align}
	\begin{thm}\label{t_main}
		$(a)$ There exists a positive non-decreasing function $T=T(R)$ and a unique continuous map
		\begin{align*}
			\mathcal{L}^2_{T}:B_R(0)\subset H^2(\mathbb{R})\times \mathcal{X}\times \mathcal{Y}
			\rightarrow
			C^0([0,T];H^2(\mathbb{R}))\cap C^1([0,T];H^0(\mathbb{R}))
		\end{align*}
		mapping initial data $(u_0,b,f)$ to strong solutions of \eqref{eq_bo_split}.

		$(b)$ Let $s\in(\frac{1}{2},2)$. There exists a positive non-decreasing function $T=T(R)$ such that $\mathcal{L}^2_T$ has a unique continuous extension
		\begin{align*}
			\mathcal{L}^s_{T}:
			B_R(0)\subset H^s(\mathbb{R})\times \mathcal{X}\times \mathcal{Y}
			\rightarrow
			C^0([0,T];H^s(\mathbb{R})).
		\end{align*}
	\end{thm}
	
	As already indicated above, Theorem \ref{t_main} implies local well-posedness results for the initial-value problem \eqref{eq_bo}. Fix a function $b$ such that $b$ and $f=f_b$ satisfy \eqref{as_b_f}. Then, for any $s>\frac{1}{2}$, the map
	\begin{align*}
		b_0+B_R(0)\subset b_0+H^s(\mathbb{R})
		&\rightarrow
		b+C^0([0,T];H^s(\mathbb{R})),
		\\
		b_0+u_0
		&\mapsto
		b+\mathcal{L}^s_T(u_0,b,f_b)
	\end{align*}
	is continuous and sends the dense subset of smooth initial data to classical solutions of \eqref{eq_bo}. Thus, we obtain the following result:
	\begin{cor}\label{c_bore}
		Let $s>\frac{1}{2}$. Assume that $b$ and $f=f_b$ satisfy \eqref{as_b_f}. Then, the Benjamin-Ono equation \eqref{eq_bo} is locally well-posed in $b+H^s(\mathbb{R})$.
	\end{cor}
	
	Similarly, we can derive local well-posedness of \eqref{eq_bo} in $H^s(\mathbb{R})+H^\sigma(\mathbb{T})$. For this, we denote the data-to-solution map (for periodic initial data) by
	\begin{align*}
		\mathcal{S}^\sigma_{\mathbb{T},T}:B_R(0)\subset 	H^\sigma(\mathbb{T})\rightarrow C^0([0,T];H^\sigma(\mathbb{T}))
	\end{align*}
	and recall its existence for $\sigma>-\frac{1}{2}$. In particular, smooth initial data are mapped to classical solutions. Moreover, by choosing $b=\mathcal{S}^\sigma_{\mathbb{T},T}(b_0)$, it follows $f=f_b=0$. Then, for $s>\frac{1}{2}$ and $\sigma>\frac{7}{2}$, the map
	\begin{align*}
		B_R(0)\subset H^s(\mathbb{R})+H^\sigma(\mathbb{T})
		&\rightarrow
		C^0([0,T];H^s(\mathbb{R})+H^\sigma(\mathbb{T}))
		\\
		u_0+b_0
		&\mapsto
		S^\sigma_{\mathbb{T},T}(b_0)
		+
		\mathcal{L}^s_T(u_0,S^\sigma_{\mathbb{T},T}(b_0),0)
	\end{align*}
	is continuous and sends the dense subset of initial data in $H^\infty(\mathbb{R})+H^\infty(\mathbb{T})$ to classical solutions of \eqref{eq_bo}. We conclude:
	\begin{cor}\label{c_main}
		Let $s>\frac{1}{2}$ and $\sigma>\frac{7}{2}$. Then, the Benjamin-Ono equation \eqref{eq_bo} is locally well-posed in $H^s(\mathbb{R})+H^\sigma(\mathbb{T})$.
	\end{cor}
	
	Lastly, we emphasize that \eqref{eq_bo_split} is more than just an auxiliary system. In \cite{Mat1995}, Matsuno derived this equation (for $b=0$ and $f$ compactly supported) to describe long internal waves in a stratified fluid with bottom topography given by the uneven profile $f$ (respectively a scaled and translated version of this). Thus, Theorem \ref{t_main} yields local well-posedness for this model and proves continuous dependence on the profile of the bottom topography.
	
	\subsection{Review of the literature}\label{ss_rol}
	Above, we already mentioned the most recent results concerning well-posedness of \eqref{eq_bo} for initial data in $H^s(\mathbb{R})$ or $H^\sigma(\mathbb{T})$. In the following, we recall a few additional results on the real line. For an extensive overview, we refer to \cite{KLV2023, KS2021a}.
	
	First statements concerning well-posedness of the Benjamin-Ono equation go back to Kato \cite{Kat1975}, Saut \cite{Sau1979a}, I{\'o}rio \cite{Ior1986}, and Abdelouhab, Bona, Felland, and Saut \cite{ABFS1989}. Each of these results provides local well-posedness for Sobolev-regularity $s$, respectively $\sigma$, being strictly larger than $\frac{3}{2}$. This lower bound stems from the derivative nonlinearity $(\phi\phi)_x$.
	
	Molinet, Saut, and Tzvetkov \cite{MST2001} showed that the dispersion is too weak to deal with the nonlinearity in a perturbative way. In particular, it is impossible to obtain local well-posedness results through direct applications of a fixpoint-argument.
	
	The breakthrough was found by Tao \cite{Tao2004}. He constructed a gauge transform that effectively cancels the worst behaving part of the nonlinearity, thus overcoming the above-mentioned restriction. This leads to global well-posedness in $H^1(\mathbb{R})$ by a modification of the Bona-Smith argument and conservation of energy.
	Many subsequent improvements rely on the application of some gauge transform -- including the almost optimal result by Killip, Laurens, and Vi\c{s}an \cite{KLV2023} for global well-posedness in $H^s(\mathbb{R})$ for $s>-\frac{1}{2}$.
	We remark that there are also results not using a gauge transform. Here, we refer to the work by Molinet, Pilod, and Vento \cite{MPV2018}, who studied a dispersion-generalized variant of \eqref{eq_bo} and obtained local well-posedness of the Benjamin-Ono equation in $H^s(\mathbb{R})$ for $s>\frac{1}{4}$.
	
	For comparison, we observe that Theorem \ref{t_main} applied to $b=0=f$ yields local well-posedness of the Benjamin-Ono equation \eqref{eq_bo} in $H^s(\mathbb{R})$ for $s>\frac{1}{2}$. This result is neither new nor optimal, but our analysis does not need a gauge transform and, in particular, it is stable under the perturbations generated by $b\neq 0\neq f$. 
	
	Next, we recall well-posedness results for other classes of initial data. To the best of our knowledge, there exist only three results in this direction. The first result is due to Fonseca and Linares and states that local well-posedness holds for sufficiently smooth initial data which are possibly unbounded, but of sublinear growth; see \cite{FL2000}.
	
	The other two results \cite{ILS1998,Gal2005} consider bore-like initial data and initial data from the larger class of Zhidkov functions.
	Here, a function $\phi_0$ is called bore-like if it satisfies the following three conditions:
	\begin{align*}
		\begin{cases}
			\phi_0(x)\rightarrow C_\pm \text{ as }x\rightarrow\infty,
			\\
			(\phi_0)_x\in H^{s-1}(\mathbb{R}) \text{ for some }s\geq1,
			\\
			(\phi_0-C_+)\in L^2([0,\infty)) \text{ and }(\phi_0+C_-)\in L^2((-\infty,0]).
		\end{cases}
	\end{align*}
	Note that this includes the example $\phi_0=\tanh$ for $C_\pm=\pm1$.
	The Zhidkov space of order $s$ is given by
	\begin{align*}
		X^s
		:=
		\{\phi_0\in\mathcal{D}'(\mathbb{R}), \phi_0\in L^\infty (\mathbb{R}), (\phi_0)_x\in H^{s-1}(\mathbb{R})\}.
	\end{align*}
	In both articles \cite{ILS1998,Gal2005} the following approach is used: First, the initial datum $\phi_0$ is written as $b_0+u_0$, where $b_0$ is a smooth function and $u_0$ is a small perturbation in $H^{s}(\mathbb{R})$. Setting $b(t,x)=b_0(x)$, this leads to the initial value problem \eqref{eq_bo_split} with $b$ as above and $f=\mathcal{H}b_{xx}+(bb)_x$. Solutions of this equation are then constructed via the method of parabolic regularization, whereas the continuous dependence on the initial datum follows by the Bona-Smith argument.
	Using this approach, I{\'o}rio, Linares, and Scialom \cite{ILS1998} proved local well-posedness of the Benjamin-Ono equation for bore-like initial data for $s>\frac{3}{2}$. Moreover, their result is global for $s\geq 2$. The authors already pointed out that the local result also follows from Kato \cite{Kat1975, Kat1979}. Interestingly, they were able to globalize their result despite the conservation laws being not (directly) applicable.
	Gallo \cite{Gal2005} established local well-posedness of the Benjamin-Ono equation in the Zhidkov space $X^s$ for $s>\frac{5}{4}$, which includes bore-like initial data. Again, this result is global for $s\geq\frac{3}{2}$.
	
	In comparison to Corollary \ref{c_bore}, the results in \cite{ILS1998, Gal2005} are formulated under weaker assumptions on the background function $b$ than \eqref{as_b_f}, but they require a lot more regularity for the $H^s(\mathbb{R})$-perturbation. The different assumptions on $b$ rely partly on the use of different commutator estimates; see Remark \ref{r_as_bf}. Lastly, we note that our result allows for background functions changing in time and that it applies directly to periodic background functions.
	
\subsection{Sketch of the proof}
	Part $(a)$ of Theorem \ref{t_main} follows by classical results given in Kato \cite{Kat1975}. In the following, we provide an overview of the proof of part $(b)$.
	
	We closely follow the approach developed by Ionsecu, Kenig, and Tataru \cite{IKT2008} leading to low-regularity well-posedness results for quasilinear equations. Moreover, we include some necessary adaptions to deal with the case $b\neq 0\neq f$; compare to Palacios \cite{Pal2022}.
	
	Let $u=\mathcal{L}^2_T(u_0,b,f)$ be a classical solution to \eqref{eq_bo_split}. We derive estimates of the form:
	\begin{align*}
		\begin{cases}
			\|u\|_{\spfc^s_T}
			&\lesssim
			\|(uu)_x\|_{\spnc^s_T}
			+
			\|(ub)_x\|_{\spnc^s_T}
			+
			T\|f\|_{L^\infty H^{3+\varepsilon}}
			+
			\|u\|_{\spec^s_T}
			\\
			\|(uu)_x\|_{\spnc^s_T}
			&\lesssim
			T\|u\|_{\spfc^s_T}^2
			\\
			\|(ub)_x\|_{\spnc^s_T}
			&\lesssim
			T\|u\|_{\spfc^s_T}\|b\|_{\spbc^{3+\varepsilon}_T}
			\\
			\|u\|_{\spec^s_T}^2
			&\lesssim
			\|u_0\|_{H^s}^2
			+
			T\|u\|_{\spfc^s_T}^2
			+
			T\|f\|_{L^\infty H^{3+\varepsilon}}
		\end{cases}
	\end{align*}
	Before we comment on these estimates and the involved spaces, let us complete the overview of the proof. After performing a bootstrap argument, we obtain the à-priori estimate
	\begin{align*}
		\|u\|_{\spec^s_T}
		\!\lesssim\!
		\|u_0\|_{H^s}
		\!+\!
		\|f\|_{L^\infty H^{3+\varepsilon}}
	\end{align*}
	for sufficiently small time $T\ll 1$ and $s\in(\frac{1}{4},2)$. Next, we argue similarly for the difference $v=u_1-u_2$ of two solutions $u_1=\mathcal{L}^2(u_{1,0},b,f)$ and $u_2=\mathcal{L}^2(u_{2,0},b,f)$. However, since the equation satisfied by $v$ is less symmetric, we have to lower the regularity at which we estimate the difference or assume more regularity on $u_2$. Here, for $z=-\frac{1}{2}$, we get the difference estimate
	\begin{align*}
		\|u_1\!-\!u_2\|_{\spec^z_T}
		\!\lesssim\!
		\|u_{1,0}\!-\!u_{2,0}\|_{H^z}
	\end{align*}
	and for $s>\frac{1}{2}$, we obtain
	\begin{align*}
		\|u_1\!-\!u_2\|_{\spec^s_T}
		\!\lesssim\!
		\|u_{1,0}\!-\!u_{2,0}\|_{H^s}
		\!+\!
		(\|u_1\!-\!u_2\|_{\spec^z_T}
		\|u_1\!-\!u_2\|_{\spec^s_T}
		\|u_2\|_{\spec^{s+1}_T}
		)^\frac{1}{2}.
	\end{align*}
	Then, the last three estimates together with the embedding $\spec^s_T\hookrightarrow C([0,T];H^s(\mathbb{R}))$ allow us to perform the classical Bona-Smith argument \cite{BS1975} proving that $\mathcal{L}^2_T$ extends continuously to $\mathcal{L}^s_T$. This finishes the proof of part $(b)$ of Theorem \ref{t_main}. 
	
	Now, let us go back to the set of four estimates stated above. The first estimate follows rather directly by applying Duhamel's formula on the level of Littlewood-Paley projections $P_Ku$ in combination with the definitions of the spaces $\spfc^s_T$, $\spnc^s_T$, $\spec^s_T$, and $\spbc^s_T$; see Section \ref{ss_fs}. The first two spaces involve frequency-dependent time-localizations meaning that each Littlewood-Paley piece $P_Ku$ is effectively only considered on a time interval of size $K^{-\theta}$, $\theta>0$. This feature is crucial for the proof of the second and third estimate.
	The fourth estimate is essentially an energy estimate. However, proving the actual estimate takes a lot of effort due to the time-localizations. Unlike to the procedures in \cite{IKT2008} and \cite{Pal2022}, we have to insert \eqref{eq_bo_split} twice, which leads to multiple terms of fourfold products of frequency-localized solutions. To bound these terms, we rely on the precise control of the resonance function, $L^2$-estimates for fourfold convolutions of localized functions, as well as on some symmetrizations to benefit from cancellations.
	
	Interestingly, the required strength of the time-localizations mentioned above is independent of whether we have $b=0$ or $b\neq 0$ -- in both cases we need $\theta>1$. This is contrary to \cite{Pal2022} dealing with the Zakharov-Kuznetsov equation. There, the value $\theta=1$ is needed in the case $b\neq 0$, but $\theta=0$ is sufficient for $b=0$.

\subsection{Structure of the paper}
	We fix some notation in the next section, followed by a short introduction of the relevant function spaces and their properties. In Section \ref{s_dce} we prove $L^2$-estimates for convolutions of dyadically localized functions. These estimates are used to bound $(uu)_x$ and $(ub)_x$ in Section \ref{s_stbe} as well as for the energy estimates in Section \ref{s_ee}. In Section \ref{s_p} we perform the aforementioned bootstrap argument and end the proof of Theorem \ref{t_main}.
\section{Preliminaries}
	In the following, we write $A\lesssim B$ if $A\leq cB$ holds for some constant $c>0$, $A\sim B$ if $A\lesssim B$ and $B\lesssim A$, and $A\gg B$ if $A\not\lesssim B$. The implicit constant $c$ may depend on further parameters and can change from line to line.
	
	We denote by $t$ and $x$ temporal and spatial variables, whereas $\tau$ and $\xi$ denote Fourier-time and frequency variables. To shorten notation summation of variables is abbreviated by listing their indices, e.g. $\xi_{123}:=\xi_1+\xi_2+\xi_3$.
	Variables in capital letters are reserved for dyadic numbers in $\mathbb{D}:=\{2^n:n\in\mathbb{N}\}$. As an exception to the previously introduced notation, we understand the expression $|\xi|\sim K$ as $|\xi|\lesssim 1$ if $K=1$.
	Moreover, we write $K_1^*,\dots,K_n^*$ to denote a reordering of $K_i$, $i\in[n]:=\{1,\dots,n\}$, satisfying $K_i^*\geq K_{i+1}^*$. 

\subsection{Littlewood-Paley projectors}
	Given a Schwartz function $f\in\mathcal{S}(\mathbb{R})$, the Fourier transform is given by
	\begin{align*}
		\mathcal{F}\left[f\right](\xi)
		:=
		(2\pi)^{-\frac{1}{2}}\int_{\mathbb{R}}e^{-i\xi x}f(x)dx,
	\end{align*}
	whereas the inverse Fourier transform is defined by $\mathcal{F}^{-1}\left[f\right](\xi):=\mathcal{F}\left[f\right](-\xi)$. 
	We also use the notations $\hat{f}$ and $f^\wedge$ for $\mathcal{F}[f]$ as well as $\check{f}$ and $f^\vee$ for $\mathcal{F}^{-1}[f]$.
	
	We define the Hilbert transform as a Fourier multiplier operator with symbol $-i\text{sgn}$, i.e.
	\begin{align*}
		\mathcal{H}:L^2(\mathbb{R})
		\rightarrow
		L^2(\mathbb{R}),
		f
		\mapsto
		\mathcal{F}^{-1}[-i\text{sgn}\mathcal{F}[f]].
	\end{align*}
	In particular, we have $(\mathcal{H}f_{xx})^\wedge(\xi)=i\omega(\xi)\hat{f}(\xi)$, where $\omega(\xi):=\xi|\xi|$.
	
	We choose a smooth even function $\chi_1\in C^\infty(\mathbb{R};[0,1])$ satisfying $\chi_1(x)=1$ for all $|x|\leq \frac{5}{4}$ and $\chi_1(x)=0$ for all $|x|\geq \frac{8}{5}$. For every dyadic integer $K>1$, we set $\chi_K(\xi):=\chi_1(\xi/(2K))-\chi_1(\xi/K)$ and note that $\supp \chi_K\subset\left[\frac{5}{8}K,\frac{8}{5}K\right]$ holds. The Littlewood-Paley projector $P_K$ for $K\in\mathbb{D}$ is defined as the Fourier multiplier operator with multiplier $\chi_K$, i.e.
	\begin{align*}
		P_K:L^2(\mathbb{R})
		\rightarrow
		L^2(\mathbb{R}),
		f
		\mapsto
		\mathcal{F}^{-1}[\chi_K\mathcal{F}[f]].
	\end{align*}
	Set $P_{\leq K}:=\sum_{1\leq L\leq K}P_L$. The terms $P_{<K}$, $P_{\geq K}$, and $P_{>K}$ are defined analogously.
	
	Moreover, we fix another smooth function $\eta_1\in C^\infty(\mathbb{R};[0,1])$ satisfying $\eta_1(x)=1$ for all $|x|\leq \frac{5}{4}$ and $\eta_1(x)=0$ for all $|x|\geq \frac{8}{5}$. We set $\eta_L(\xi):=\eta_1(\xi/(2L))-\eta_1(\xi/L)$ for all dyadic integers $L>1$ as well as $\eta_{\leq L}=\sum_{1\leq N\leq L}\eta_N$.
	
	The functions $(\chi_K)_{K\in\mathbb{D}}$ will be used to localize frequency variables, whereas $(\eta_L)_{L\in\mathbb{D}}$ will be used to localize modulation and time variables. For $K,L\in\mathbb{D}$, the support of the product of two such functions is denoted by
	\begin{align*}
		D_{L,K}
		:=
		\{(\tau,\xi)\in\mathbb{R}^2:
		\tau-\omega(\xi)\in\supp \eta_L, \xi\in\supp\chi_K\}.
	\end{align*}
\subsection{Function spaces}\label{ss_fs}
	Given $s\in\mathbb{R}$ and $p,q\in [1,\infty]$, we define the $B^s_{p,q}$-norm of a tempered distribution $f\in\mathcal{S}'(\mathbb{R})$ by
	\begin{align*}
		\|f\|_{B^s_{p,q}}
		:=
		\|(K^{s}\|P_K f\|_{L^p})_{K\in\mathbb{D}}\|_{l^q}.
	\end{align*}
	Then, the inhomogeneous Besov space $B^s_{p,q}(\mathbb{R})$ is given by the space of all tempered distributions with finite $B^s_{p,q}$-norm, i.e.\@ we have
	\begin{align*}
		B^s_{p,q}(\mathbb{R})
		:=
		\{f\in\mathcal{S}'(\mathbb{R}):\|f\|_{B^s_{p,q}}<\infty\}.
	\end{align*}
	Throughout this work, we only consider Besov spaces for $(p,q)=(2,2)$ and for $(p,q)=(\infty,\infty)$. In the first case, we recover the Sobolev space $B^s_{2,2}(\mathbb{R})= H^s(\mathbb{R})$ with an equivalent norm. In the second case, $B^s_{\infty,\infty}(\mathbb{R})$ corresponds to the space of $s$-Lipschitz functions $\Lambda^s(\mathbb{R})$, again, with an equivalent norm.
	We abbreviate
		\begin{align*}
		\spbc^s_T
		&:=
		L^\infty([-T,T];B^s_{\infty,\infty}(\mathbb{R})).
	\end{align*}
	
	Next, let us introduce the function spaces $\spfc^s_T$, $\spnc^s_T$, and $\spec^s_T$ as in \cite{IKT2008} with a few notational changes. We begin by giving the definitions and provide some comments afterwards. For the rest of this section, let $T\in(0,1)$, $K\in\mathbb{D}$, $s\in\mathbb{R}$, and $\theta>0$.
	First, we define a $l^1$-variant of the $X^{s,b}$-norm for $b=\frac{1}{2}$. For $f\in L^2(\mathbb{R}^2)$ define
	\begin{align*}
		\|f\|_{\spx^K}:=\sum_{L\in\mathbb{D}}L^{\frac{1}{2}}\|\eta_L(\tau-\omega(\xi))f(\tau,\xi)\|_{L^2}.
	\end{align*}
	This norm is independent of $K$, but the corresponding space $\spx^K$ contains only frequency-localized functions. More precisely, we set
	\begin{align*}
		\spx^K
		:=
		\{f\in L^2(\mathbb{R}\times \supp\chi_K):
		\|f\|_{\spx^K}<\infty\}.
	\end{align*}
	Based on this space, we define the following two norms involving frequency-dependent time-localizations:
	\begin{align*}
		\|u\|_{\spfa^K}
		&:=
		\sup_{t_K}\|\mathcal{F}\left[u\eta_0(K^{\theta }(t-t_K))\right]\|_{\spx^K},
		\\
		\|u\|_{\spna^K}
		&:=
		\sup_{t_K}\|(\tau-\omega(\xi)+iK^\theta)^{-1}\mathcal{F}\left[u\eta_0(K^{\theta}(t-t_K))\right]\|_{\spx^K}.
	\end{align*}
	Observe that the time localization restricts to a time interval of size $K^{-\theta}$. Hence, the larger we choose $\theta$, the faster the size of the intervals goes to zero as $K$ tends to infinity.
	The corresponding spaces are defined as
	\begin{align*}
		\spfa^K
		&:=
		\{u_K\in C(\mathbb{R};P_KL^2(\mathbb{R})):
		\|u_K\|_{\spfa^K}<\infty\},
		\\
		\spna^K
		&:=
		\{u_K\in C(\mathbb{R};P_KL^2(\mathbb{R})):
		\|u_K\|_{\spna^K}<\infty\},
	\end{align*}
	where $P_KL^2(\mathbb{R})=\{P_K f:f\in L^2(\mathbb{R})\}$.
	Next, we define the corresponding restriction norms by
	\begin{align*}
		\|u_k\|_{\spfb^K_T}
		&:=
		\inf\{\|\tilde{u}_K\|_{\spfa^K}:
		\tilde{u}_K\in\spfa^K,
		\tilde{u}_K=u_K\text{ on }[-T,T]\},
		\\
		\|u_K\|_{\spnb^K_T}
		&:=
		\inf\{\|\tilde{u}_K\|_{\spna^K}:
		\tilde{u}_K\in\spna^K,
		\tilde{u}_K=u_K\text{ on }[-T,T]\},
	\end{align*}
	and the spaces by
	\begin{align*}
		\spfb^K_T
		&:=
		\{u_K\in C(\mathbb{R};P_KL^2(\mathbb{R})):
		\|u_K\|_{\spfb^K_T}<\infty\},
		\\
		\spnb^K_T
		&:=
		\{u_K\in C(\mathbb{R};P_KL^2(\mathbb{R})):
		\|u_K\|_{\spnb^K_T}<\infty\}.
	\end{align*}
	Based on the previously defined norms, we set
	\begin{align*}
		\|u\|_{\spfc^s_T}
		&:=
		\|(K^{s}\|P_Ku\|_{\spfb^K_T})_{K\in\mathbb{D}}\|_{l^2},
		\\
		\|u\|_{\spnc^s_T}
		&:=
		\|(K^{s}\|P_Ku\|_{\spnb^K_T})_{K\in\mathbb{D}}\|_{l^2},
	\end{align*}
	and define
	\begin{align*}
		\spfc^s_T
		&:=
		\{u\in C([-T,T];H^2(\mathbb{R})):
		\|u\|_{\spfc^s_T}<\infty\},
		\\
		\spnc^s_T
		&:=
		\{u\in C([-T,T];H^2(\mathbb{R})):
		\|u\|_{\spnc^s_T}<\infty\}.
	\end{align*}
	It remains to define the space $\spec^s_T$. For this, let
	\begin{align*}
		\|u\|_{\spec^s_T}
		:=
		\|(K^{s}\|P_Ku\|_{L^\infty_t L^2_x})_{K\in\mathbb{D}}\|_{l^2}
	\end{align*}
	and define the corresponding space
	\begin{align*}
		\spec^s_T
		&:=
		\{u\in C([-T,T];H^2(\mathbb{R})):
		\|u\|_{\spec^s_T}<\infty\}.
	\end{align*}
	
	We emphasize that functions in $\spfc^s_T$, $\spnc^s_T$, and $\spec^s_T$ are elements of $C([0,T];H^2 (\mathbb{R}))$. Hence, the superscript $s$ does not directly correspond to the spatial regularity of the functions.
	We recall that our goal is to approximate initial data in $H^s(\mathbb{R})$, $s<2$, by initial data in $H^2(\mathbb{R})$. Then, we want to show that the corresponding sequence of solutions (which are elements in $C([0,T];H^2(\mathbb{R}))$) converges. In particular, the approximate solutions will be elements of the spaces $\spfc^s_T$ and $\spec^s_T$, whereas the space $\spnc^s_T$ contains the appearing nonlinearities.
	
	Now, we briefly comment on the value of $\theta$. As mentioned in the introduction, the parameter $\theta$ controls the size of the time interval to which we localize functions; see the definitions of the spaces $\spfa^K$ and $\spna^K$. More precisely, functions that are frequency-localized to $\supp \chi_K$ will be localized in time to intervals of size $K^{-\theta}$. These time-localizations are crucial for the estimates of the nonlinearity; see Section \ref{s_stbe}. In particular, we obtain the necessary condition $\theta>1$ (independent of $s$).
	However, while the time-localizations are beneficial for the estimate of the nonlinearity, they are very disadvantageous for the energy estimates causing rather tedious proofs; see Section \ref{s_ee}. For the energy estimates, we want to choose $\theta$ as small as possible. More precisely, if $s$ converges to $\frac{1}{2}$, then the corresponding $\theta=\theta(s)>1$ must converge to $1$. Hence, if not mentioned otherwise, we will always assume that $\theta$ is a number slightly larger than $1$ satisfying all necessary upper bounds appearing in our proofs.
\subsection{Properties of the function spaces}
	
	In this section, we recall a few properties of the previously introduced function spaces. For proofs and further properties, we refer to \cite{IKT2008}. Moreover, we present a generalization from \cite{Pal2022}, see Proposition \ref{p_cnv} below.
	
	We begin by stating three lemmata. The first two describe how the spaces $\spx^K$ behave with respect to time-localization with smooth respectively sharp cut-off functions. The third lemma guarantees that functions in $\spfb^K_T$, respectively in $\spnb^K_T$, possess well-behaved extensions in $\spfa^K$, respectively $\spna^K$.
	\begin{lem}\label{l_locstab}
		Let $\gamma\in\mathcal{S}(\mathbb{R})$. There exists a constant (only depending on $\gamma$) such that for all $K, L\in\mathbb{D}$, $t_0\in\mathbb{R}$, and $f\in\spx^K$ we have
		\begin{align*}
			\|\mathcal{F}\left[\gamma(L(t-t_0))\mathcal{F}^{-1}\left[f\right]\right]\|_{\spx^K}
			\lesssim
			\|f\|_{\spx^K}.
		\end{align*}
	\end{lem}
	\begin{lem}\label{l_regtime}
		Let $\delta\in[0,\frac{1}{2})$, $T\in(0,1)$, and let $I$ be an interval around $0$ of length $T$. For all $K, L_0\in\mathbb{D}$ and any $u$ we have
		\begin{align*}
			\sup_{L\geq L_0}L^{\frac{1}{2}-\delta}\|\eta_L(\tau-\omega(\xi))\chi_{K}(\xi)\mathcal{F}\left[1_I(t)u\right]\|_{L^2}
			\lesssim
			T^{\delta}\|\chi_{K}(\xi)\mathcal{F}\left[u\right]\|_{\spx^K}.
		\end{align*}
		Here, $\eta_{L}$ is interpreted as $\eta_{\leq L_0}$ if $L=L_0$.
	\end{lem}
	The following corollary will be used for $\delta=\delta'=\delta''$.
	\begin{cor}\label{c_regtime}
		Let $\zeta\in[0,\frac{1}{2}]$ and $\delta', \delta''>0$ with $\zeta+\delta'+\delta''<\frac{1}{2}$. Let $T\in(0,1)$ and let $I$ be an interval around $0$ of length $T$. For all $K, L_0\in\mathbb{D}$ and any $u$ we have
		\begin{align*}
			\sum_{L\geq L_0}L^{\zeta}\|\eta_L(\tau-\omega(\xi))\chi_{K}(\xi)\mathcal{F}\left[1_I(t)u\right]\|_{L^2}
			\lesssim_{\delta''}
			T^{\delta'}L_0^{\zeta-\frac{1}{2}+\delta'}\|\chi_{K}(\xi)\mathcal{F}\left[u\right]\|_{\spx^K}.
		\end{align*}
		Here, the implicit constant depends continuously on $\delta''$ and tends to $+\infty$ as $\delta''\rightarrow0$.
	\end{cor}
	\begin{proof}
		We have
		\begin{align*}
			&\sum_{L\geq L_0}L^{\zeta}\|\eta_L(\tau-\omega(\xi))\chi_{K}(\xi)\mathcal{F}\left[1_I(t)u\right]\|_{L^2}
			\\ \lesssim
			&\sum_{L\geq L_0} L^{-\delta''}\times
			\sup_{L\geq L_0}L^{\zeta+\delta''}\|\eta_L(\tau-\omega(\xi))\chi_{K}(\xi)\mathcal{F}\left[1_I(t)u\right]\|_{L^2}.
		\end{align*}
		Above, the first factor can be bounded by $C_{\delta''}L_0^{-\delta''}$ with $C_{\delta''}=2^{\delta''}(\delta''\log(2))^{-1}$. Multiplying the second factor by $(L/L_0)^{\frac{1}{2}-\zeta-\delta'-\delta''}\geq1$ and applying Lemma \ref{l_regtime}, it follows
		\begin{align*}
			&L_0^{\zeta-\frac{1}{2}+\delta'+\delta''}\sup_{L\geq L_0}L^{\frac{1}{2}-\delta'}\|\eta_L(\tau-\omega(\xi))\chi_{K}(\xi)\mathcal{F}\left[1_I(t)u\right]\|_{L^2}
			\\ \lesssim\,
			&T^{\delta'}L_0^{\zeta-\frac{1}{2}+\delta'+\delta''}
			\|\chi_{K}(\xi)\mathcal{F}\left[u\right]\|_{\spx^K}.
		\end{align*}
	\end{proof}
	
	\begin{lem}\label{l_ext}
		Let $K\in\mathbb{D}$ and $T\in(0,1)$. For every $u\in \spfb^K_T$ and every $v\in \spnb^K_T$ there exist extensions $\tilde{u}\in\spfa^K$ and $\tilde{v}\in\spna^K$ supported in $[-2T,2T]\times\mathbb{R}$ satisfying
		\begin{align*}
			\|\tilde{u}\|_{\spfa^K}
			\lesssim
			\|u\|_{\spfb^K_T}
			\qquad\text{and}\qquad
			\|\tilde{v}\|_{\spna^K}
			\lesssim
			\|v\|_{\spnb^K_T}.
		\end{align*}
	\end{lem}
	
	We continue by formulating a crucial fact for the bootstrap argument concerning the behaviour of some terms as $T$ converges to $0$. It is easy to observe that $\lim_{T\rightarrow0}\|u\|_{\spec^s_T}=\|u_0\|_{H^s}$ holds. More generally, the following proposition follows from \cite[Lemma 6.3]{Pal2022}, which itself is an extension of a result from \cite{IKT2008}.
	\begin{pro}\label{p_cnv}
		Let $0\leq s<2$, $T\in(0,1)$, and $u\in C([0,T];H^2(\mathbb{R}))$. Then, we have
		\begin{align*}
			\lim_{T\rightarrow 0}
			\|u\|_{\spec^s_T}
			+
			\|(uu)_x\|_{\spnc^s_T}
			+
			\|(ub)_x\|_{\spnc^s_T}
			=
			\|u_0\|_{H^s}.
		\end{align*}
	\end{pro}
	Lastly, we formulate a linear estimate. This estimate follows essentially by applying Duhamel's formula to frequency-localized functions.
	\begin{lem}
		Let $s\in\mathbb{R}$, $\theta> 0$, and $T\in(0,1)$. Let $u$ be a strong solution of $u_t+\mathcal{H}u_{xx}=n$. Then, we have
		\begin{align*}
			\|u\|_{\spfc^s_T}
			\lesssim
			\|u\|_{\spec^s_T}
			+
			\|n\|_{\spnc^s_T}.
		\end{align*}
	\end{lem}
	For the bootstrap argument we use the following version of the previous estimate.
	\begin{cor}
		Let $s\in\mathbb{R}$, $\theta> 0$, and $T\in(0,1)$. Let $u$ be a strong solution of \eqref{eq_bo_split}. Then, we have
		\begin{equation}\label{es_lin}
			\|u\|_{\spfc^s_T}
			\lesssim
			\|u\|_{\spec^s_T}
			+
			\|(uu)_x\|_{\spnc^s_T}
			+
			\|(ub)_x\|_{\spnc^s_T}
			+
			T^{\frac{1}{2}}\|f\|_{L^\infty H^s}.
		\end{equation}
	\end{cor}
\subsection{Resonance function}\label{ss_res}
	
	In this section we provide bounds on resonance functions. These functions appear naturally in the energy estimates when considering products of solutions of \eqref{eq_bo_split}.
	
	We define the hypersurface $\Gamma_n:=\{\xi=(\xi_1,\dots,\xi_n)\in\mathbb{R}^n|\xi_{1\cdots n}=0\}$, which -- for later use -- we equip with a measure $\gamma_n$ characterized by
	\begin{align*}
		\int_{\Gamma_n} \varphi(\xi)d\gamma_n(\xi)
		=
		\int_{\mathbb{R}^{n-1}} \varphi(\xi_1,\dots,\xi_{n-1},-\xi_{1\cdots(n-1)})d\xi_1\cdots d\xi_{n-1}.
	\end{align*}
	The resonance functions $\Omega_n$ are defined on $\Gamma_n$ as follows:
	\begin{equation}\label{d_res}
		\Omega_n:\Gamma_n
		\rightarrow
		\mathbb{R},
		\qquad
		(\xi_1,\dots,\xi_n)
		\mapsto
		\omega(\xi_1)
		+
		\cdots
		+
		\omega(\xi_n).
	\end{equation}
	For our purposes, the resonance functions $\Omega_3$ and $\Omega_4$ will be sufficient.
	\begin{lem}\label{l_res}
		Let $(\xi_1,\xi_2,\xi_3)\in\Gamma_3$ and $K_i\in\mathbb{D}$ satisfying $|\xi_i|\sim K_i$ for each $i\in[3]$. If $K_3^*>1$, then we have
		\begin{equation}\label{e_res_asy}
			|\Omega_3(\xi_1,\xi_2,\xi_3)|
			\sim
			K_1^*K_3^*.
		\end{equation}
		Let $(\xi_1,\xi_2,\xi_3,\xi_4)\in\Gamma_4$ and $K_i$ satisfying $|\xi_i|\sim K_i$ for each $i\in[4]$. If $K_3^*>1$, then we have
		\begin{equation}\label{e_res4_asy}
			|\Omega_4(\xi_1,\xi_2,\xi_3,\xi_4)|
			\lesssim
			K_1^*K_3^*.
		\end{equation}
	\end{lem}
	\begin{lem}\label{l_res_dif}
		Let $(\xi_a,\xi_b,\xi_2,\xi_3)\in\Gamma_4$ and $K_i\in\mathbb{D}$ satisfying $|\xi_i|\sim K_i$ for each $i\in\{a,b,2,3\}$. If $K_a\sim K_2\gtrsim K_b$ and $K_2\gg K_3> 1$, then we have
		\begin{align*}
			|\Omega_3^{-1}(\xi_a,\xi_{2b},\xi_3)-\Omega_3^{-1}(\xi_{ab},\xi_2,\xi_3)|
			\lesssim
			K_bK_3^{-1}K_2^{-2}.
		\end{align*}
	\end{lem}
	\begin{proof}[Proof of Lemma \ref{l_res}]
		We observe that $\Omega_3$ is invariant under permutations and satisfies $\Omega_3(-\xi_1,-\xi_2,-\xi_3)=-\Omega_3(\xi_1,\xi_2,\xi_3)$. Hence, after relabelling, we can assume that $\xi_1>-\xi_2>-\xi_3>0$ holds. Then, the first claim follows from
		\begin{align*}
			\Omega_3(\xi_1,\xi_2,\xi_3)
			=
			\xi_1^2-\xi_2^2-\xi_3^2
			=
			2\xi_2\xi_3.
		\end{align*}
		
		The second claim follows by similar considerations. First, we note that it suffices to consider the two cases $\xi_1>-\xi_2>-\xi_3>\xi_4>0$ and $\xi_1>-\xi_2>-\xi_3>-\xi_4>0$. In the first case, we apply the double mean value theorem and conclude
		\begin{align*}
			\Omega_4(\xi_1,\xi_2,\xi_3,\xi_4)
			&=
			\omega(\xi_1)-\omega(\xi_1-\xi_{12})-\omega(\xi_1-\xi_{13})+\omega(\xi_1-\xi_{12}-\xi_{13})
			\\
			&=
			(-\xi_{12})(-\xi_{13})\int_0^1\int_0^1\omega''(\xi_1-s\xi_{12}-t\xi_{13})dsdt
			=
			-2\xi_{34}\xi_{13}.
		\end{align*}
		In the second case, a direct calculation yields
		\begin{align*}
			\Omega_4(\xi_1,\xi_2,\xi_3,\xi_4)
			=
			\xi_1^2-\xi_2^2-\xi_3^2-\xi_4^2
			=
			-2(\xi_2\xi_3-\xi_{14}\xi_4).
		\end{align*}
		The claim follows since both $\xi_2\xi_3$ and $-\xi_{14}\xi_4$ have the same sign.
	\end{proof}
	\begin{proof}[Proof of Lemma \ref{l_res_dif}]
		It is easy to see that
		\begin{align*}
			|\Omega_3^{-1}(\xi_a,\xi_{2b},\xi_3)-\Omega_3^{-1}(\xi_{ab},\xi_2,\xi_3)|
			=
			\frac{|\Omega_4(\xi_a,-\xi_{a3},-\xi_{ab},\xi_{ab3})|}{|\Omega_3(\xi_a,\xi_{2b},\xi_3)\Omega_3(\xi_{ab},\xi_2,\xi_3)|}
		\end{align*}
		holds. According to \eqref{e_res_asy}, the denominator is comparable to $K_2^{-2}K_3^{-2}$. Thus, it suffices to bound the nominator by $K_3K_b$. We consider the cases $K_2\gg K_b$ and $K_2\sim K_b$ separately.
		In the first case, the frequencies $\xi_a$, $\xi_{a3}$, $\xi_{ab}$, and $\xi_{ab3}$ have modulus of size $K_a$. An application of the double mean value theorem implies that the nominator is bounded by $K_3K_b$ as requested.
		In the second case, the fundamental theorem of calculus yields $|\omega(\xi_{ab3})-\omega(\xi_{ab})|=|\omega(\xi_2)-\omega(\xi_{23})|\lesssim K_2K_3$ as well as $|\omega(\xi_a)-\omega(\xi_{a3})|\lesssim K_2K_3$. Hence, we conclude that the nominator is bounded by $K_2K_3\sim K_bK_3$.
	\end{proof}
\section{Dyadic convolution estimates}\label{s_dce}

	The goal of this section is to prove $L^2$-bounds for convolutions of functions which are both localized in frequency and modulation. Such functions appear naturally in the definition of the $\spx^K$-spaces. Consequently, the obtained estimates will be useful for the calculations in Sections \ref{s_stbe} and \ref{s_ee}.
	
	We begin by recalling a bound for the convolution of two functions that are localized in the modulation variable. A proof of this result for periodic functions can be found in \cite[Lemma 7.1]{Mol2007}; the arguments directly transfer to non-periodic functions.
	\begin{pro}\label{p_cnv_l2}
		Let $\phi_1$ and $\phi_2 \in L^2(\mathbb{R}^2;\mathbb{R}_+)$ with $\supp \phi_i\subset D_{L_i,K_i}$ for $i\in[2]$.
		Then, we have
		\begin{equation}\label{e_cnv_l2}
			\|\phi_1\ast \phi_2\|_{L^2}
			\lesssim
			(L_1^*)^\frac{1}{4}(L_2^*)^\frac{1}{2}\|\phi_1\|_{L^2}\|\phi_2\|_{L^2}.
		\end{equation}
	\end{pro}
	We also need the following refinement.
	\begin{lem}\label{l_cnv_l2}
		Let $\phi_1$ and $\phi_2 \in L^2(\mathbb{R}^2;\mathbb{R}_+)$ with $\supp\phi_i\subset D_{L_i,K_i}$ for $i\in[2]$.
		Then, we have
		\begin{equation}\label{e_cnv_l2_gen}
			\|1_{D_{L_3,K_3}}(\phi_1\ast \phi_2)\|_{L^2}
			\lesssim
			(L_3^*)^\frac{1}{2}(K_3^*)^\frac{1}{2}\|\phi_1\|_{L^2}\|\phi_2\|_{L^2}
		\end{equation}
		and, if additionally $K_3^*>1$, it follows
		\begin{equation}\label{e_cnv_l2_imp}
			\|1_{D_{L_3,K_3}}(\phi_1\ast \phi_2)\|_{L^2}
			\lesssim
			(L_1^*L_3^*)^\frac{1}{2}(K_1^*)^{-\frac{1}{2}}\|\phi_1\|_{L^2}\|\phi_2\|_{L^2}.
		\end{equation}
		Moreover, if $K_3^*>1$ and $L_1^*\ll K_1^*K_3^*$ hold, then we have $\|1_{D_{L_3,K_3}}(\phi_1\ast \phi_2)\|_{L^2}=0$.
	\end{lem}
	
	Instead of proving Lemma \ref{l_cnv_l2} directly, we will prove the dual formulation given in the following lemma.
	\begin{lem}\label{l_cnv}
		Let $\phi_1$, $\phi_2$, and $\phi_3\in L^2(\mathbb{R}^2;\mathbb{R}_+)$ with $\supp \phi_i\subset D_{L_i,K_i}$ for $i\in[3]$. Then, we have
		\begin{equation}\label{e_cnv_tri_gen}
			(\phi_1\ast\phi_2\ast\phi_3)(0,0)
			\lesssim
			(L_3^*)^\frac{1}{2}(K_3^*)^\frac{1}{2}\|\phi_1\|_{L^2}\|\phi_2\|_{L^2}\|\phi_3\|_{L^2}
		\end{equation}
		and, if additionally $K_3^*> 1$, it follows
		\begin{equation}\label{e_cnv_tri_imp}
			(\phi_1\ast\phi_2\ast\phi_3)(0,0)
			\lesssim
			(L_1^*L_3^*)^\frac{1}{2}(K_1^*)^{-\frac{1}{2}}\|\phi_1\|_{L^2}\|\phi_2\|_{L^2}\|\phi_3\|_{L^2}.
		\end{equation}
		Moreover, if $K_3^*>1$ and $L_1^*\ll K_1^*K_3^*$ hold, then we have $(\phi_1\ast\phi_2\ast\phi_3)(0,0)=0$.
	\end{lem}
	\begin{proof}
		Note that \eqref{e_cnv_tri_imp} follows directly from \eqref{e_cnv_tri_gen} and the last claim given above.
		
		Now, we prove estimate \eqref{e_cnv_tri_gen}. Without loss of generality, assume $L_2=L_3^*$ and $K_1=K_3^*$. Then, applying Plancherel's theorem, Hölder's inequality, Minkowski's inequality, and Bernstein's inequality, we obtain
		\begin{align*}
			(\phi_1\ast\phi_2\ast\phi_3)(0,0)
			\!\leq
			\|\phi_1^\vee\|_{L^\infty_t L^2_x}\!\|\phi_2^\vee\|_{L^2_tL^\infty_x}\!\|\phi_3\|_{L^2}
			\lesssim
			L_2^\frac{1}{2}K_1^\frac{1}{2}\|\phi_1\|_{L^2}\!\|\phi_2\|_{L^2}\!\|\phi_3\|_{L^2},
		\end{align*}
		proving the desired estimate.
		
		It remains to show that the convolution vanishes for $K_3^*>1$ and $L_1^*\ll K_1^*K_3^*$.
		We define $\phi_i^-(\tau,\xi):=\phi_i(\tau,-\xi)$, $\phi_i^\#(\tau,\xi):=\phi_i(\tau-\omega(\xi),\xi)$, and $\phi_i^\times=(\phi_i^-)^\#$. The functions $\phi_i^\#$ and $\phi_i^\times$ have supports contained in $\{|\tau|\sim L_i,|\xi|\sim K_i\}$. Then, a change of variables implies
		\begin{align*}
			(\phi_1\ast\phi_2\ast\phi_3)(0,0)
			\!=\!
			\int_{\Gamma_3}\!\int_{\Gamma_3}\!
			\phi_1^\#(\tau_1,\xi_1)\phi_2^\#(\tau_2,\xi_2)\phi_3^\times(\tau_3+\Omega_3(\xi),\xi_3)d\gamma_3(\tau)d\gamma_3(\xi),
		\end{align*}
		where we write $\tau=(\tau_1,\tau_2,\tau_3)$ and $\xi=(\xi_1,\xi_2,\xi_3)$.
		Due to $K_3^*> 1$, we have $K_2=K_3^*> 1$ and $L_3\neq L_3^*$ after relabelling. It follows $L_3\sim L_1^*$. Now, we argue by contraposition. Assume that the convolution is non-zero. Then, it is necessarily positive and the integrand must be positive at least for some values of $\tau_1$, $\tau_2$, $\xi_1$, and $\xi_2$. In particular, this implies $|\tau_1|\lesssim L_1$, $|\tau_2|\lesssim L_2$, and $|\tau_{3}+\Omega_3(\xi)|\lesssim L_3$. Together with $-\tau_3=\tau_1+\tau_2$ and $|\tau_1+\tau_2|\lesssim L_1+L_2\sim L_3$, we conclude $|\tau_3|\lesssim L_3$ leading to $|\Omega_3(\xi)|\lesssim L_3$. From \eqref{e_res_asy}, we deduce $K_1^*K_3^*\lesssim L_3\sim L_1^*$. Thus, the proof is complete.
	\end{proof}
	\begin{lem}\label{l_cnv_q}
		Let $\phi_1$, $\phi_2$, $\phi_3$, and $\phi_4\in L^2(\mathbb{R}^2;\mathbb{R}_+)$ with $\supp \phi_i\subset D_{L_i,K_i}$ for $i\in[4]$. Then, we have
		\begin{equation}\label{e_cnv_qua_gen}
			(\phi_1\ast\phi_2\ast\phi_3\ast\phi_4)(0,0)
			\!\lesssim\!
			(L_3^*L_4^*)^\frac{1}{2}(K_3^*K_4^*)^\frac{1}{2}
			\|\phi_1\|_{L^2}\!\|\phi_2\|_{L^2}\!\|\phi_3\|_{L^2}\!\|\phi_4\|_{L^2}
		\end{equation}
		and, if additionally $K_3^*>1$, it follows
		\begin{equation}\label{e_cnv_qua_imp}
			(\phi_1\ast\phi_2\ast\phi_3\ast\phi_4)(0,0)
			\!\lesssim\!
			(L_1^*L_2^*L_4^*)^\frac{1}{2}(K_1^*)^{-\frac{1}{2}}(K_4^*)^{\frac{1}{2}}\|\phi_1\|_{L^2}\!\|\phi_2\|_{L^2}\!\|\phi_3\|_{L^2}\!\|\phi_4\|_{L^2}.
		\end{equation}
	\end{lem}
	\begin{lem}\label{l_cnv_qb}
		Let $\phi_1$, $\phi_2$, and $\phi_3\in L^2(\mathbb{R}^2;\mathbb{R}_+)$ with $\supp \phi_i\subset D_{L_i,K_i}$ for $i\in[3]$. Let $\phi_\infty\in L^\infty(\mathbb{R}^2;\mathbb{R}_+)$. Then, we have
		\begin{equation}\label{e_cnv_qua_bnd_one}
			(\phi_1\ast\phi_2\ast\phi_3\ast\phi_\infty)(0,0)
			\!\lesssim\!
			(L_3^*)^\frac{1}{2}(K_3^*)^\frac{1}{2}
			\|\phi_1\|_{L^2}\!\|\phi_2\|_{L^2}\!\|\phi_3\|_{L^2}\!\|\phi_\infty^\vee\|_{L^\infty}
		\end{equation}
		and
		\begin{equation}\label{e_cnv_qua_bnd_two}
			(\phi_1\ast\phi_2\ast\phi_3\ast\phi_\infty)(0,0)
			\!\lesssim\!
			(L_2^*)^\frac{1}{4}(L_3^*)^\frac{1}{2}
			\|\phi_1\|_{L^2}\!\|\phi_2\|_{L^2}\!\|\phi_3\|_{L^2}\!\|\phi_\infty^\vee\|_{L^\infty}.
		\end{equation}
	\end{lem}
	
	The estimates in Lemma \ref{l_cnv_qb} follow easily by using Hölder's inequality and Proposition \ref{p_cnv_l2}, which is why we omit their proofs. 
	
	Recall that we used the lower bound on $\Omega_3$ in the proof of Lemma \ref{l_cnv}. Since such a lower bound is wrong for $\Omega_4$, we need a different argument for the proof estimate \eqref{e_cnv_qua_imp} in Lemma \ref{l_cnv_q}. 
	\begin{proof}[Proof of Lemma \ref{l_cnv_q}]
		As in the proof of Lemma \ref{l_cnv}, we derive estimate \eqref{e_cnv_qua_gen} by applying Hölder's and Bernstein's inequality.
		
		Next, we prove estimate \eqref{e_cnv_qua_imp}. If $K_1^*\sim K_4^*$ holds, then the estimate follows immediately from Hölder's inequality and two applications of \eqref{e_cnv_l2}. Thus, it suffices to consider $K_1^*\gg K_4^*$. After relabelling, we have $K_4\sim K_3\gtrsim K_2\gtrsim K_1$ and $K_4\gg K_1$. Using the notation of the previous proof and letting $\phi_4^\ast :=(\phi_4^\times)^-$, we have
		\begin{align*}
			&(\phi_1\ast\phi_2\ast\phi_3\ast\phi_4)(0,0)
			\\ =
			&\int_{\mathbb{R}^3}\! d\tau_1d\tau_2d\tau_3\int_{\mathbb{R}^3}\! d\xi_1d\xi_2d\xi_3 \phi_1^\#(\tau_1,\xi_1)\phi_2^\#(\tau_2,\xi_2)\phi_3^\#(\tau_3,\xi_3)\phi_4^\ast(\Omega_4(\xi)\!-\!\tau_{123},\xi_{123}),
		\end{align*}
		where $\xi=(\xi_1,\xi_2,\xi_3,-\xi_{123})$. We define $M(\tau_{123},\xi_1,\xi_2):=\{\xi_3:|\tau-\Omega_4(\xi)|\lesssim L_4\}$ and remark that $\sup_{\tau_{123},\xi_1,\xi_2}|M|\lesssim L_4K_3^{-1}$ holds. This estimate follows essentially from the inequality $|\partial_{\xi_1}\Omega_4(\xi)|\gtrsim K_3$, which makes use of the assumption $K_3\sim K_4\gg K_1$. Then, multiple applications of Hölder's and Bernstein's inequality lead to
		\begin{align*}
			&(\phi_1\ast\phi_2\ast\phi_3\ast\phi_4)(0,0)
			\\ =
			&\int_{\mathbb{R}^3} d\tau_1d\tau_2d\tau_3\int_{\mathbb{R}^3} d\xi_1d\xi_2d\xi_3
			\phi_1^\#\phi_2^\#\phi_3^\#\phi_4^\ast
			\\ \lesssim
			&\int_{\mathbb{R}^3} d\tau_1d\tau_2d\tau_3\int_{\mathbb{R}^2} d\xi_1d\xi_2
			\phi_2^\#\phi_3^\#\left[\int_{\mathbb{R}} d\xi_3 1_{M(\tau_{123},\xi_2,\xi_3)}\right]^\frac{1}{2}\left[\int_{\mathbb{R}} d\xi_3 \left[\phi_3^\#\phi_4^\ast\right]^2\right]^\frac{1}{2}
			\\ \lesssim\,
			&(L_3L_4)^\frac{1}{2}K_3^{-\frac{1}{2}}\int_{\mathbb{R}^2} d\tau_1d\tau_2\int_{\mathbb{R}^2} d\xi_1d\xi_2
			\phi_2^\#\phi_3^\#\left[\int_{\mathbb{R}^2} d\tau_3 d\xi_3 \left[\phi_3^\#\phi_4^\ast\right]^2\right]^\frac{1}{2}.
		\end{align*}
		The next steps depend on the relative size of $K_1$ and $K_2$, respectively $L_1$ and $L_2$. By assumption, we have $K_1\leq K_2$ and without loss of generality we may assume $L_2\leq L_1$. Then, we proceed by
		\begin{align*}
			&(\phi_1\ast\phi_2\ast\phi_3\ast\phi_4)(0,0)
			\\ \lesssim\,
			&(L_3L_4)^\frac{1}{2}K_3^{-\frac{1}{2}}\|\phi_3^\#\|\|\phi_4^\ast\|\int_{\mathbb{R}} d\tau_1\int_{\mathbb{R}} d\xi_2
			\left[\int_{\mathbb{R}} d\tau_2\left[\phi_2^\#\right]^2\right]^\frac{1}{2}\left[\int_{\mathbb{R}} d\xi_1\left[\phi_1^\#\right]^2\right]^\frac{1}{2}
			\\ \lesssim\,
			&(L_2L_3L_4)^\frac{1}{2}K_3^{-\frac{1}{2}}K_1^\frac{1}{2}\|\phi_1^\#\|\|\phi_2^\#\|\|\phi_3^\#\|\|\phi_4^\ast\|.
		\end{align*}
		It remains to note that the $L^2$-norms of $\phi_i$, $\phi_i^\ast$, and $\phi_i^\times$ coincide.
	\end{proof}
\section{Short-time bilinear estimates}\label{s_stbe}
	The aim of this section is to prove estimates for the nonlinearities appearing in \eqref{eq_bo_split}. In the following three lemmata we state these estimates in the form that we use later. Afterwards we provide the proof of slightly more general estimates.  	
	\begin{lem}\label{l_nl}
		Let $s\geq s_0>\frac{1}{4}$ and $\sigma>\max\{s,\frac{1}{2}\}$. Then, for all sufficiently small $\theta=\theta(s_0,\sigma)>1$ and $\delta=\delta(\theta,s_0,\sigma)>0$, every $T\in(0,1)$, all $u\in\spfc^s_T$, and all $b\in\spbc^\sigma_T$, we have
		\begin{align}
			\|(uu)_x\|_{\spnc^s_T}
			&\lesssim_\delta
			T^\delta
			\|u\|_{\spfc^s_T}\|u\|_{\spfc^{s_0}_T},\label{e_nl_uu}
			\\
			\|(ub)_x\|_{\spnc^s_T}
			&\lesssim_\delta
			T^\delta
			\|u\|_{\spfc^s_T}
			\|b\|_{\spbc^\sigma_T}.\label{e_nl_ub}
		\end{align}
	\end{lem}
	\begin{lem}\label{l_nl2}
		Let $s>\frac{1}{2}$, $z=-\frac{1}{2}$, and $\sigma>\max\{s,1\}$. Then, for all sufficiently small $\theta=\theta(s,\sigma)>1$ and $\delta=\delta(\theta,s,\sigma)>0$, every $T\in(0,1)$, all $u_1$, $u_2\in\spfc^s_T$, and all $b\in\spbc^\sigma_T$, we have
		\begin{align}
			\|((u_1-u_2)(u_1+u_2))_x\|_{\spnc^z_T}
			&\lesssim_\delta
			T^\delta
			\|u_1-u_2\|_{\spfc^z_T}
			(\|u_1\|_{\spfc^s_T}+\|u_2\|_{\spfc^s_T}),\label{e_nl_vw}
			\\
			\|((u_1-u_2)b)_x\|_{\spnc^z_T}
			&\lesssim_\delta
			T^\delta
			\|u_1-u_2\|_{\spfc^z_T}
			\|b\|_{\spbc^\sigma_T}.\label{e_nl_vb}
		\end{align}
	\end{lem}
	\begin{lem}\label{l_nl3}
		Let $s>\frac{1}{2}$ and $\sigma>\frac{1}{2}$. Then, for all sufficiently small $\theta=\theta(s,\sigma)>1$ and $\delta=\delta(\theta,s,\sigma)>0$, every $T\in(0,1)$, all $u_1$, $u_2\in\spfc^s_T$, and all $b\in\spbc^\sigma_T$, we have
		\begin{align}
			\|((u_1-u_2)(u_1+u_2))_x\|_{\spnc^s_T}
			&\lesssim_\delta
			T^\delta
			\|u_1-u_2\|_{\spfc^s_T}
			(\|u_1\|_{\spfc^s_T}+\|u_2\|_{\spfc^s_T}),\label{e_nl_vw1}
			\\
			\|((u_1-u_2)b)_x\|_{\spnc^s_T}
			&\lesssim_\delta
			T^\delta
			\|u_1-u_2\|_{\spfc^s_T}
			\|b\|_{\spbc^\sigma_T}.\label{e_nl_vb1}
		\end{align}
	\end{lem}

	In essence, each estimate stated above is proved in the same way. We write down the left-hand side explicitly and add further localizations. Then, we apply one of the estimates from Section \ref{s_dce} depending on the frequency interactions. Lastly, we use the definitions of the spaces $\spfc^s_T$ and $\spbc^s_T$ to obtain the desired upper bounds. In particular, here, we do not use that $u$ is a solution of \eqref{eq_bo_split}.

	In Lemma \ref{l_nl_loc} given below, we prove an estimate that implies \eqref{e_nl_uu}, \eqref{e_nl_vw}, and \eqref{e_nl_vw1}. The estimates in \eqref{e_nl_ub}, \eqref{e_nl_vb}, and \eqref{e_nl_vb1} will be a consequence of Lemma \ref{l_nl_bnd}.
	
	We begin with an auxiliary statement.
	\begin{lem}\label{l_partgamma}
		Let $\gamma\in C^\infty(\mathbb{R};[0,1])$ be a smooth function satisfying $\gamma\equiv 1$ on $[-\frac{1}{4},\frac{1}{4}]$ and $\gamma\equiv 0$ on the complement of $[-1,1]$. Let $c,\theta\geq 0$, and pick $M_1,M_2\in\mathbb{D}$ satisfying $2^cM_1\geq M_2$. Then, we have
		\begin{align*}
			\gamma^2(2^{c\theta+1} M_1^\theta t+n)
			=
			\gamma^2(2^{c\theta+1} M_1^\theta t+n)
			\eta_0(M_2^\theta t+n)
		\end{align*}
		 for all $t\in\mathbb{R}$ and $n\in\mathbb{Z}$.
		Moreover, let $M_3, M_4\in\mathbb{D}$ with $M_3\geq M_4$. Then, we have
		\begin{align*}
			\#\{
			n\in\mathbb{Z}:
			\gamma(2^{c\theta+1} M_3^\theta t+n)\eta_0(M_4^\theta t)
			\neq
			0\}
			\leq
			2^{c\theta+4}M_3^\theta M_4^{-\theta}
		\end{align*}
		 for all $t\in\mathbb{R}$.
	\end{lem}
	We will use Lemma \ref{l_partgamma} as follows: We understand $M_1\gtrsim M_2$ strictly as $8M_1\geq M_2$. Hence, the equation above holds with $c=3$. If $M_3\geq M_4$ or $M_3\gg M_4$ (i.e.\@ $M_4\not\gtrsim M_3$), then the inequality above is used for $c=0$.
	\begin{lem}\label{l_nl_loc}
		Let $s>\frac{1}{4}$. Then, for all sufficiently small $\theta=\theta(s)>1$ and $\delta=\delta(\theta)>0$, every $T\in(0,1)$, and all $f_1, f_2\in\spfc^s_T$, we have
		\begin{align*}
			\|(f_1f_2)_x\|_{\spnc^s_T}
			\lesssim_{\delta}
			T^\delta
			\|f_1\|_{\spfc^s_T}
			\|f_2\|_{\spfc^0_T}
			+
			\|f_1\|_{\spfc^0_T}
			\|f_2\|_{\spfc^s_T}.
		\end{align*}
		Let $s>\frac{1}{2}$ and $z=-\frac{1}{2}$. Then, for all sufficiently small $\theta=\theta(s,\epsilon)>1$ and $\delta=\delta>0$, all $f_1\in\spfc^z_T$, and all $f_2\in\spfc^s$, we have
		\begin{align*}
			\|(f_1f_2)_x\|_{\spnc^z_T}
			\lesssim_{\delta}
			T^\delta
			\|f_1\|_{\spfc^z_T}
			\|f_2\|_{\spfc^s_T}.
		\end{align*}
		Here, the implicit constants depend continuously on $\delta$ and tends to $+\infty$ as $\delta\rightarrow0$.
	\end{lem}
	Note that the first estimate implies both \eqref{e_nl_uu} and \eqref{e_nl_vw1}, whereas the second estimate implies \eqref{e_nl_vw}.
	\begin{proof}[Proof of Lemma \ref{l_nl_loc}]
		We begin by choosing extensions $f_{K_1}$ of $P_{K_1}f_1$ and $f_{K_2}$ of $P_{K_2}f_2$ for each $K_1$, $K_2\in\mathbb{D}$. According to Lemma \ref{l_ext}, we can assume that $f_{K_1}$ and $f_{K_2}$ have temporal support in $[-2T,2T]$ as well as that the estimates
		\begin{align}\label{e_nl_loc_ext}
			\|f_{K_1}\|_{\spna^{K_1}}
			\lesssim
			\|P_{K_1} f_1\|_{\spnb^{K_1}_T}
			\qquad
			\text{and}
			\qquad
			\|f_{K_2}\|_{\spna^{K_2}}
			\lesssim
			\|P_{K_2} f_2\|_{\spnb^{K_2}_T}
		\end{align}
		hold with implicit constants independent of $K_1$ and $K_2$. At first, we estimate the $\spnc^r_T$-norm for arbitrary $r\in\mathbb{R}$; later, we will specify $r=s$ or $r=z$.
		Using the definitions of the $\spnc^r_T$- and $\spnb^{K_3}_T$-norms, inserting Littlewood-Paley projectors to $f_1$ and $f_2$, and making use of the fact that $f_{K_1}f_{K_2}$ is an extension of $P_{K_1}f_1P_{K_2}f_2$, we obtain
		\begin{align}\label{e_nl_loc_red}
			\|(f_1f_2)_x\|_{\spnc^r_T}^2
			\leq
			\sum_{K_3}
			\left[
			K_3^r
			\sum_{K_1, K_2}
			\|P_{K_3} (f_{K_1}f_{K_2})_x\|_{\spna^{K_3}}
			\right]^2.
		\end{align}
		
		In the following, we will bound the right-hand side of \eqref{e_nl_loc_red} by considering different frequency interactions separately. For this, we split the domain of summation into the following four sets:
		\begin{align*}
			\Sigma_1
			&:=
			\{K_2\sim K_3\gg K_1\},
			&&\Sigma_2
			:=
			\{K_1\sim K_3\gg K_2\},
			\\
			\Sigma_3
			&:=
			\{K_1, K_2, K_3 \leq 64\},
			&&\Sigma_4
			:=
			\{K_1\sim K_2\gtrsim K_3\}.
		\end{align*}
		
		\textit{Summation over $\Sigma_1$.}
		We define $\mathfrak{t}=t-t_{K_3}$ and $\mathfrak{m}=\tau-\omega(\xi)$. Using the definition of the $\spna^{K_3}$-norm, we get
		\begin{align*}
			\|P_{K_3}(f_{K_1}f_{K_2})_x\|_{\spna^{K_3}}
			\!\leq\!
			K_3\!\sup_{t_{K_3}}\!\sum_{n}\!\|(\mathfrak{m}+iK_3^{\theta})^{-1}\chi_{K_3}(\xi)\mathcal{F}_{t,x}\!\left[\eta_0(K_3^{\theta}\mathfrak{t})f_{K_1}f_{K_2}\right]\!\|_{\spx^{K_3}}.
		\end{align*}
		Next, fix a smooth function $\gamma$ (as in Lemma \ref{l_partgamma}) such that translations of $\gamma^2$ constitute a partition of unity, i.e.\@ $\sum_{n\in\mathbb{Z}}\gamma^2(\cdot+n)=1$. In particular, we have
		\begin{align}\label{e_nl_loc_gamc1}
			\gamma^2(2^{3\theta+1} K_3^{\theta}\mathfrak{t}+n)
			=
			\gamma^2(2^{3\theta+1} K_3^{\theta}\mathfrak{t}+n)
			\eta_0(K_1^{\theta}\mathfrak{t}+n)
			\eta_0(K_2^{\theta}\mathfrak{t}+n).
		\end{align}
		Define
		\begin{align*}
			F_{K_1}(\tau,\xi)
			&:=
			\mathcal{F}_{t,x}\left[
			\gamma(2^{3\theta+1} K_3^{\theta}\mathfrak{t} - n)
			\eta_0(K_1^{\theta}\mathfrak{t} - n)f_{K_1}(t,x)
			\right](\tau,\xi),
			\\
			F_{K_2}(\tau,\xi)
			&:=
			\mathcal{F}_{t,x}\left[
			\gamma(2^{3\theta+1} K_3^{\theta}\mathfrak{t} - n)
			\eta_0(K_2^{\theta}\mathfrak{t} - n)
			\eta_0(K_3^{\theta}\mathfrak{t})f_{K_2}(t,x)
			\right](\tau,\xi).
		\end{align*}
		Inserting the partition of unity and applying \eqref{e_nl_loc_gamc1}, it follows
		\begin{align*}
			\|P_{K_3}(f_{K_1}f_{K_2})_x\|_{\spna^{K_3}}
			\leq
			K_3\sup_{t_{K_3}}\sum_{n}\|(\mathfrak{m}+iK_3^{\theta})^{-1}\chi_{K_3}(\xi)\left(F_{K_1}\ast F_{K_2}\right)\|_{\spx^{K_3}}.
		\end{align*}
		The function $F_{K_2}$ vanishes for all but at most $2^{3\theta+4}$ many values of $n$ (see Lemma \ref{l_partgamma}). Hence, we can bound the sum in $n$ by that constant times its supremum. 
		Combining this with the definition of the $X^{K_3}$-norm, we conclude
		\begin{align*}
			\|P_{K_3}(f_{K_1}f_{K_2})_x\|_{\spna^{K_3}}
			\!\lesssim\!
			K_3\!\sup_{t_{K_3}, n}\!\sum_{L_3}\!L_3^{\frac{1}{2}}
			\|(\mathfrak{m}+iK_3^{\theta})^{-1}\eta_{L_3}(\mathfrak{m})\chi_{K_3}(\xi)\left(F_{K_1}\!\ast F_{K_2}\right)\!\|_{L^2}.
		\end{align*}
		
		Let us briefly show that we can group all small modulations. More precisely, we consider the sum over $L_3\in\{1,\dots, \lfloor K_3^\theta\rfloor_{\mathbb{D}}\}$. Here, $\lfloor M\rfloor_{\mathbb{D}}$ denotes the largest dyadic number smaller than $M$. In this case we have $|(\mathfrak{m}+iK_3^\theta)^{-1}|\sim \lfloor K_3^\theta\rfloor^{-1}_{\mathbb{D}}$ leading to
		\begin{align*}
			\sum_{0\leq L_3\leq \lfloor K_3^\theta\rfloor_{\mathbb{D}}}L_3^{\frac{1}{2}}
			\|(\mathfrak{m}+iK_3^{\theta})^{-1}\eta_{L_3}(\mathfrak{m})\chi_{K_3}(\xi)\left(F_{K_1}\ast F_{K_2}\right)\|_{L^2}
			\\
			\qquad\lesssim
			\lfloor K_3^\theta\rfloor^{-\frac{1}{2}}_{\mathbb{D}}
			\|\eta_{\leq \lfloor K_3^\theta\rfloor_{\mathbb{D}}}(\mathfrak{m})\chi_{K_3}(\xi)\left(F_{K_1}\ast F_{K_2}\right)\|_{L^2}.			
		\end{align*}
		Thus, we conclude
		\begin{align*}
			\|P_{K_3}(f_{K_1}f_{K_2})_x\|_{\spna^{K_3}}
			\lesssim
			K_3\sup_{t_{K_3}, n}\sum_{L_3}L_3^{-\frac{1}{2}}
			\|\eta_{L_3}(\mathfrak{m})\chi_{K_3}(\xi)\left(F_{K_1}\ast F_{K_2}\right)\|_{L^2},
		\end{align*}
		where the sum in $L_3$ starts at $\lfloor K_3^\theta\rfloor_{\mathbb{D}}$ and we interpret $\eta_{L_3}$ as $\eta_{\leq\lfloor K_3^\theta\rfloor_{\mathbb{D}}}$ if $L_3=\lfloor K_3^\theta\rfloor_{\mathbb{D}}$.
		
		Now, recall that $D_{L_3,K_3}$ denotes the support of $\eta_{L_3}(\mathfrak{m})\chi_{K_3}(\xi)$. Moreover, the inequality $\eta_{L_3}(\mathfrak{m})\chi_{K_3}(\xi)\leq 1_{D_{L_3,K_3}}(\tau,\xi)$ is valid. We localize the functions $F_{K_1}$ and $F_{K_2}$ with respect to their modulation variables, i.e.\@ for $i\in[2]$ we define
		\begin{align*}
			F_{L_i,K_i}(\tau,\xi)
			:=
			\begin{cases}
				\eta_{L_i}(\mathfrak{m})F_{K_i}(\tau,\xi)
				&\text{if }L_i>\lfloor K_3^\theta\rfloor_{\mathbb{D}},
				\\
				\eta_{\leq L_i}(\mathfrak{m})F_{K_i}(\tau,\xi)
				&\text{if }L_i=\lfloor K_3^\theta\rfloor_{\mathbb{D}}.
			\end{cases}
		\end{align*}
		It follows
		\begin{align*}
			\|P_{K_3}(f_{K_1}f_{K_2})_x\|_{\spna^{K_3}}
			\lesssim
			K_3\sup_{t_{K_3}, n}\sum_{L_1,L_2,L_3}L_3^{-\frac{1}{2}}
			\|1_{D_{L_3,K_3}}(F_{L_1,K_1}\ast F_{L_2,K_2})\|_{L^2},
		\end{align*}
		where the sum in $L_1$ and $L_2$ starts at $\lfloor K_3^\theta\rfloor_{\mathbb{D}}$.
		At this point, the appearing $L^2$-norms satisfy the premise of Lemma \ref{l_cnv_l2}.
		Depending on whether $K_3^*>1$ or $K_3^*=1$ holds, we apply estimate \eqref{e_cnv_l2_imp} or \eqref{e_cnv_l2_gen}. Here, we only consider the first case; the second one follows similarly. Using \eqref{e_cnv_l2_imp}, we get the prefactor $(L_1^*L_3^*)^\frac{1}{2}(K_1^*)^{-\frac{1}{2}}$.
		Without loss of generality, assume that $L_1=L_1^*$ and $L_2=L_3^*$ hold. Then, the above implies
		\begin{align*}
			\|P_{K_3}(f_{K_1}f_{K_2})_x\|_{\spna^{K_3}}
			\lesssim
			K_3^\frac{1}{2}\sup_{t_{K_3},n}\sum_{L_3}L_3^{-\frac{1}{4}}
			S_1^{\frac{1}{2}} S_2^{\frac{1}{4}}
		\end{align*}
		with
		\begin{align*}
			S_i^\beta
			:=
			\sum_{L_i\geq \lfloor K_3^\theta\rfloor_{\mathbb{D}}}L_i^{\beta}\|F_{L_i,K_i}(\tau,\xi)\|_{L^2}.
		\end{align*}
		
		Next, we apply Corollary \ref{l_regtime} to $S_2^{\frac{1}{4}}$ (choosing $\delta=\delta'=\delta''\in(0,\frac{1}{8})$) and note that $S_1^\frac{1}{2}$ is essentially equal to the $\spx^{K_1}$-norm of $F_{K_1}$. Then, we apply Lemma \ref{l_locstab} multiple times to the appearing $\spx^{K_1}$- and $\spx^{K_2}$-norms and use the choice of the extensions $f_{K_1}$ and $f_{K_2}$, see \eqref{e_nl_loc_ext}, leading to
		\begin{align*}
			\|P_{K_3}(f_{K_1}f_{K_2})_x\|_{\spna^{K_3}}
			\lesssim_\delta
			T^\delta K_3^{\frac{1}{2}-\theta(\frac{1}{2}-\delta)}
			\|P_{K_1}f\|_{\spfb^{K_1}_T}\|P_{K_2}f\|_{\spfb^{K_2}_T}.
		\end{align*}
		Under the additional assumption $\frac{1}{2}-\theta(\frac{1}{2}-\delta)< 0$, i.e.\@ if $\delta<	 \frac{\theta-1}{2\theta}$, we bound the contribution of $\Sigma_1$ for $r\geq 0$ by
		\begin{align}\label{e_nl_loc_c1}
			\sum_{K_3}
			\left[
			K_3^r
			\sum_{K_1, K_2}1_{\Sigma_1}
			\|P_{K_3} (f_{K_1}f_{K_2})_x\|_{\spna^{K_3}}
			\right]^2
			\lesssim_\delta
			T^{2\delta}\|f_1\|_{\spfc^0_T}^2\|f_2\|_{\spfc^r_T}^2
		\end{align}
		and the contribution for $r=z=-\frac{1}{2}$ by
		\begin{align}\label{e_nl_loc_c11}
			\sum_{K_3}
			\left[
			K_3^r
			\sum_{K_1, K_2}1_{\Sigma_1}
			\|P_{K_3} (f_{K_1}f_{K_2})_x\|_{\spna^{K_3}}
			\right]^2
			\lesssim_\delta
			T^{2\delta}\|f_1\|_{\spfc^r_T}^2\|f_2\|_{\spfc^0_T}^2
		\end{align}
		We obtain the same bound for all other choices of $L_1^*$ and $L_3^*$.
		
		\textit{Summation over $\Sigma_2$.} We can repeat the proof for the contribution of $\Sigma_1$ after relabelling $K_1$ as $K_2$ and vice versa. For $r\in\mathbb{R}$ this leads to
		\begin{align}\label{e_nl_loc_c2}
			\sum_{K_3}
			\left[
			K_3^r
			\sum_{K_1, K_2}1_{\Sigma_2}
			\|P_{K_3} (f_{K_1}f_{K_2})_x\|_{\spna^{K_3}}
			\right]^2
			\lesssim_\delta
			T^{2\delta}\|f_1\|_{\spfc^r_T}^2\|f_2\|_{\spfc^0_T}^2.
		\end{align}
		
		\textit{Summation over $\Sigma_3$.}
		Again, we proceed as for $\Sigma_1$. We hide all unwanted powers of $K_1^*$ in the implicit constant. Thus, we obtain (with an implicit cconstant depending also on a lower bound of $r$)
		\begin{align}\label{e_nl_loc_c3}
			\sum_{K_3}
			\left[
			K_3^r
			\sum_{K_1, K_2}1_{\Sigma_3}
			\|P_{K_3} (f_{K_1}f_{K_2})_x\|_{\spna^{K_3}}
			\right]^2
			\lesssim_\delta
			T^{2\delta}\|f_1\|_{\spfc^r_T}^2\|f_2\|_{\spfc^0_T}^2.
		\end{align}
		
		\textit{Summation over $\Sigma_4$.} In order to bound the contribution of $\Sigma_4$ appropriately, we have to modify the arguments used for $\Sigma_1$ in a non-trivial way. Let $\gamma$ be a smooth function as in Lemma \ref{l_partgamma} satisfying $\sum_{n\in\mathbb{Z}}\gamma^2(\cdot+n)=1$. Assume without loss of generality that $K_1\geq K_2$ holds. Then, we have
		\begin{align}\label{e_nl_loc_gamc4}
			\gamma(2K_1^{\theta}\mathfrak{t}-n)
			=
			\gamma(2K_1^{\theta}\mathfrak{t}-n)
			\eta_0(K_1^{\theta}\mathfrak{t}-n)
			\eta_0(K_2^{\theta}\mathfrak{t}-n)
		\end{align}
		for all $n\in\mathbb{N}$. As before, we use the notation $\mathfrak{t}=t-t_{K_3}$ and $\mathfrak{m}=\tau-\omega(\xi)$. Rewriting the definition of the $\spna^{K_3}$- and $\spx^{K_3}$-norms, we get
		\begin{align*}
			&\|P_{K_3}(f_{K_1}f_{K_2})_x\|_{\spna^{K_3}}
			\\&\qquad\lesssim
			K_3\sup_{t_{K_3}}\sum_{L_3}L_3^{\frac{1}{2}}\|(\mathfrak{m}+iK_3^{\theta})^{-1}\eta_{L_3}(\mathfrak{m})\chi_{K_3}(\xi)\mathcal{F}_{t,x}\left[\eta_0(K_3^{\theta}\mathfrak{t})f_{K_1}f_{K_2}\right]\|_{L^2}.
		\end{align*}
		
		We continue by grouping all small modulations, i.e.\@ all summands for $L_3\leq \lfloor K_3^\theta\rfloor_{\mathbb{D}}$. Understanding $\eta_{L_3}$ as $\eta_{\leq\lfloor K_3^\theta\rfloor_{\mathbb{D}}}$ if $L_3=\lfloor K_3^\theta\rfloor_{\mathbb{D}}$, this leads to
		\begin{align}\label{e_nl_loc_c4red}
			&\|P_{K_3}(f_{K_1}f_{K_2})_x\|_{\spna^{K_3}}\nonumber
			\\&\qquad\lesssim
			K_3\sup_{t_{K_3}}\sum_{L_3}L_3^{-\frac{1}{2}}\|\eta_{L_3}(\mathfrak{m})\chi_{K_3}(\xi)\mathcal{F}_{t,x}\left[\eta_0(K_3^{\theta}\mathfrak{t})f_{K_1}f_{K_2}\right]\|_{L^2},
		\end{align}
		where the sum is considered over $L_3\geq\lfloor K_3^\theta\rfloor_{\mathbb{D}}$.
		
		Now, we study the contribution of the summands for $L_3\in\{\lfloor K_3^\theta\rfloor_{\mathbb{D}},\dots, K_1K_3\}$. For this purpose, define
		\begin{align*}
			F_{K_1}(\tau,\xi)
			&:=
			\mathcal{F}\left[
			\gamma(2K_1^{\theta}\mathfrak{t} - n)
			\eta_0(K_1^{\theta}\mathfrak{t} - n)f_{K_1}(t,x)
			\right](\tau,\xi),
			\\
			F_{K_2}(\tau,\xi)
			&:=
			\mathcal{F}\left[
			\gamma(2K_1^{\theta}\mathfrak{t} - n)
			\eta_0(K_2^{\theta}\mathfrak{t} - n)
			\eta_0(K_3^{\theta}\mathfrak{t})f_{K_2}(t,x)
			\right](\tau,\xi),
		\end{align*}
		as well as
		\begin{align*}
			F_{L_i,K_i}(\tau,\xi)
			:=
			\begin{cases}
				\eta_{L_i}(\mathfrak{m})F_{K_i}(\tau,\xi)
				&\text{if }L_i>\lfloor K_1^\theta\rfloor_{\mathbb{D}},
				\\
				\eta_{\leq L_i}(\mathfrak{m})F_{K_i}(\tau,\xi)
				&\text{if }L_i=\lfloor K_1^\theta\rfloor_{\mathbb{D}}.
			\end{cases}
		\end{align*} 
		According to Lemma \ref{l_partgamma}, the function $F_{K_2}$ vanishes for all but at most $16K_1^\theta K_3^{-\theta}$ many values of $n$. Inserting the identity \eqref{e_nl_loc_gamc4} into \eqref{e_nl_loc_c4red} and using the previous definitions, it follows
		\begin{align*}
			&\|P_{K_3}(f_{K_1}f_{K_2})_x\|_{\spna^{K_3}}
			\\&\qquad\lesssim
			K_1^\theta K_3^{(1-\theta)}\sup_{t_{K_3},n}\sum_{L_1,L_2,L_3}L_3^{-\frac{1}{2}}\|\eta_{L_3}(\mathfrak{m})\chi_{K_3}(\xi)(F_{L_1,K_1}\ast F_{L_2,K_2})\|_{L^2}.
		\end{align*}
		Above, the sum is taken over $L_1,L_2\geq \lfloor K_1^\theta\rfloor_{\mathbb{D}}$ and $L_3\in\{\lfloor K_3^\theta\rfloor_{\mathbb{D}}, \dots, K_1K_3\}$. If $K_3^*>1$, we apply \eqref{e_cnv_l2_imp} and otherwise we use \eqref{e_cnv_l2_gen}. At this point, we restrict to showing the former case. Independently of the relative size of $L_1$, $L_2$, and $L_3$, we obtain
		\begin{align*}
			\|P_{K_3}(f_{K_1}f_{K_2})_x\|_{\spna^{K_3}}
			\lesssim
			K_1^{\theta-\frac{1}{2}}K_3^{(1-\theta)}\sup_{t_{K_3},n}\sum_{L_3}S^\frac{1}{2}_1S^0_2
		\end{align*}
		with
		\begin{align*}
			S_i^\beta
			:=
			\sum_{L_i\geq \lfloor K_1^\theta\rfloor_\mathbb{D}}L_i^{\beta}\|F_{L_i,K_i}(\tau,\xi)\|_{L^2}.
		\end{align*}
		Note that the sum in $L_3$ is of size $\log_2(K_1)$. Hence, we can bound it by $c_\delta K_1^{\delta}$. Next, we apply Corollary \ref{c_regtime} to $S^0_2$, note that $S^\frac{1}{2}_1$ is essentially the $\spx^{K_1}$-norm of $F_{K_1}$, and apply Lemma \ref{l_locstab} multiple times to the appearing $\spx^{K_1}$- and $\spx^{K_2}$-norm. Lastly, we use our choice of the extension in \eqref{e_nl_loc_ext}. Thus, for all $\delta\in(0,\frac{1}{4})$ and $\epsilon>0$, we obtain
		\begin{align*}
			\|P_{K_3}(f_{K_1}f_{K_2})_x\|_{\spna^{K_3}}
			\lesssim_{\delta}
			T^\delta K_1^{\delta-\frac{1}{2}+\theta(\frac{1}{2}+\delta)}K_3^{(1-\theta)}\|f_1\|_{\spfb_T^{K_1}}\|f_2\|_{\spfb_T^{K_2}}.
		\end{align*}
		If $r>0$ we choose $\theta$ sufficiently small such that $r+(1-\theta)\frac{1}{2}\geq 0$. Moreover, choose $\delta<\frac{2-\theta}{2(1+\theta)}$; then, we conclude
		\begin{align}\label{e_nl_loc_c41}
			\sum_{K_3}
			\left[
			K_3^r
			\sum_{K_1, K_2}1_{\Sigma_4}
			\|P_{K_3}(f_{K_1}f_{K_2})_x\|_{\spna^{k_3}}
			\right]^2
			\lesssim_\delta
			T^{2\delta}\|f_1\|_{\spfc^r_T}^2\|f_2\|_{\spfc^0_T}^2.
		\end{align}
		Otherwise, if $r=z=-\frac{1}{2}$, then we choose $\delta<\frac{2s-\theta}{2(1+\theta)}$ obtaining $s>\delta+\theta(\frac{1}{2}+\delta)$. This leads to
		\begin{align}\label{e_nl_loc_c42}
			\sum_{K_3}
			\left[
			K_3^r
			\sum_{K_1, K_2}1_{\Sigma_4}
			\|P_{K_3}(f_{K_1}f_{K_2})_x\|_{\spna^{k_3}}
			\right]^2
			\lesssim_\delta
			T^{2\delta}\|f_1\|_{\spfc^r_T}^2\|f_2\|_{\spfc^s_T}^2.
		\end{align}
		
		Now, we consider the contribution of \eqref{e_nl_loc_c4red} for $L_3\geq K_1K_3$. In this case, we drop the factor $\eta_{L_3}\chi_{K_3}$ and the sum in $L_3$ contributes a factor $(K_1K_3)^{-\frac{1}{2}}$. To bound the $L^2$-norm, we use Plancherel's theorem, insert the partition of unity, and use almost orthogonality leading to
		\begin{align*}
			\|\mathcal{F}\left[\eta_0(K_3^{\theta}\mathfrak{t})f_{K_1}f_{K_2}\right](\tau,\xi)\|_{L^2}^2
			\lesssim
			\|\eta_0(K_3^{\theta}\mathfrak{t})\sum_n\gamma^2(2K_1^{\theta}\mathfrak{t}-n)f_{K_1}f_{K_2}\|_{L^2}^2
			\\
			\qquad\lesssim
			\sum_n\sum_{m\in\{-1,0,1\}}\|\eta_0(K_3^{\theta}\mathfrak{t})\gamma(2K_1^{\theta}\mathfrak{t}-n)\gamma(2K_1^{\theta}\mathfrak{t}-n+m)f_{K_1}f_{K_2}\|_{L^2}^2.
		\end{align*}
		Then, inserting \eqref{e_nl_loc_gamc4}, using the definitions
		\begin{align*}
			F_{K_1}(\tau,\xi)
			&:=
			\mathcal{F}\left[
			\gamma(2K_1^{\theta}\mathfrak{t} - n+m)
			\eta_0(K_1^{\theta}\mathfrak{t} - n)f_{K_1}(t,x)
			\right](\tau,\xi),
			\\
			F_{K_2}(\tau,\xi)
			&:=
			\mathcal{F}\left[
			\gamma(2K_1^{\theta}\mathfrak{t} - n)
			\eta_0(K_2^{\theta}\mathfrak{t} - n)
			\eta_0(K_3^{\theta}\mathfrak{t})f_{K_2}(t,x)
			\right](\tau,\xi),
		\end{align*}
		and observing that $F_{K_2}$ vanishes for all but at most $2^4K_1^\theta K_3^{-\theta}$ many values of $n$ (see Lemma \ref{l_partgamma}), we conclude
		\begin{align*}
			\|\mathcal{F}_{t,x}\left[\eta_0(K_3^{\theta}\mathfrak{t})f_{K_1}f_{K_2}\right]\|_{L^2}^2
			\lesssim
			\sum_{n,m}\|F_{K_1}\ast F_{K_2}\|_{L^2}^2
			\lesssim
			K_1^\theta K_3^{-\theta}\sup_{n,m}\|F_{K_1}\ast F_{K_2}\|_{L^2}^2.
		\end{align*}
		So far, we have shown the estimate
		\begin{align*}
			\|P_{K_3}(f_{K_1}f_{K_2})_x\|_{\spna^{K_3}}
			\lesssim
			K_1^{\frac{\theta-1}{2}}K_3^{\frac{1-\theta}{2}}\sup_{t_{K_3},n,m}\|F_{K_1}\ast F_{K_2}\|_{L^2}.
		\end{align*}
		Next, we define
		\begin{align*}
			F_{L_i,K_i}(\tau,\xi)
			:=
			\begin{cases}
				\eta_{L_i}(\mathfrak{m})F_{K_i}(\tau,\xi)
				&\text{if }L_i>\lfloor K_1^\theta\rfloor_{\mathbb{D}},
				\\
				\eta_{\leq L_i}(\mathfrak{m})F_{K_i}(\tau,\xi)
				&\text{if }L_i=\lfloor K_1^\theta\rfloor_{\mathbb{D}},
			\end{cases}
		\end{align*}
		for $i\in[2]$ and conclude
		\begin{align*}
			\|P_{K_3}(f_{K_1}f_{K_2})_x\|_{\spna^{K_3}}
			\lesssim
			K_1^{\frac{\theta-1}{2}}K_3^{\frac{1-\theta}{2}}\sup_{t_{K_3},n,m}\sum_{L_1,L_2}\|F_{L_1,K_1}\ast F_{L_2,K_2}\|_{L^2}.
		\end{align*}
		Above, the sums in $L_1$ and $L_2$ are taken over $L_1, L_2\geq \lfloor K_1^\theta\rfloor_{\mathbb{D}}$.
		
		Without loss of generality, assume $L_1\geq L_2$. Using \eqref{e_cnv_l2}, we get
		\begin{align*}
			\|P_{K_3}(f_{K_1}f_{K_2})_x\|_{\spna^{K_3}}
			\lesssim
			K_1^{\frac{\theta-1}{2}}K_3^{\frac{1-\theta}{2}}\sup_{t_{K_3},n,m}S^{\frac{1}{4}}_1S^{\frac{1}{2}}_2
		\end{align*}
		with
		\begin{align*}
			S_i^\beta
			:=
			\sum_{L_i\geq \lfloor K_1^\theta\rfloor_\mathbb{D}}L_i^{\beta}\|F_{L_i,K_i}(\tau,\xi)\|_{L^2}.
		\end{align*}
		Arguing as before, i.e.\@ with Corollary \ref{c_regtime}, Lemma \ref{l_locstab} and \eqref{e_nl_loc_ext}, we conclude
		\begin{align*}
			\|P_{K_3}(f_{K_1}f_{K_2})_x\|_{\spna^{K_3}}
			\lesssim
			K_1^{\theta(\frac{1}{4}+\delta)-\frac{1}{2}}K_3^{\frac{1-\theta}{2}}\|f_1\|_{\spfb_T^{K_1}}\|f_2\|_{\spfb_T^{K_2}}
		\end{align*}
		for $\delta\in(0,\frac{1}{8})$. If $r\geq0$, we choose $\theta$ and $\delta$ small such that $\theta(\frac{1}{4}+\delta)\leq\frac{1}{2}$ holds; this implies
		\begin{align}\label{e_nl_loc_c43}
			\sum_{K_3}
			\left[
			K_3^r
			\sum_{K_1, K_2}1_{\Sigma_4}
			\|P_{K_3} (f_{K_1}f_{K_2})_x\|_{\spna^{K_3}}
			\right]^2
			\lesssim_\delta
			T^{2\delta}\|f_1\|_{\spfc^r_T}^2\|f_2\|_{\spfc^0_T}^2.
		\end{align}
		If $r=-\frac{1}{2}$ and $s>\frac{1}{4}$, then we choose $\theta$ and $\delta$ small such that $s\geq \theta(\frac{1}{4}+\delta)$ holds. We conclude
		\begin{align}\label{e_nl_loc_c44}
			\sum_{K_3}
			\left[
			K_3^r
			\sum_{K_1, K_2}1_{\Sigma_4}
			\|P_{K_3} (f_{K_1}f_{K_2})_x\|_{\spna^{K_3}}
			\right]^2
			\lesssim_\delta
			T^{2\delta}\|f_1\|_{\spfc^r_T}^2\|f_2\|_{\spfc^s_T}^2.
		\end{align}
		All in all, we have proven the claimed estimates.
	\end{proof}
	\begin{lem}\label{l_nl_bnd}
		Let $s\geq 0$ and $\sigma>\max\{s,\frac{1}{2}\}$. Then, for all sufficiently small $\theta=\theta(\sigma)>1$ and $\delta=\delta(\sigma, \theta)>0$, every $T\in(0,1)$, all $f\in\spfc^s_T$, and all $g\in\spbc^\sigma_T$, we have
		\begin{align*}
			\|(fg)_x\|_{\spnc^s_T}
			\lesssim_\delta
			T^\delta
			\|f\|_{\spfc^s_T}
			\|g\|_{\spbc^\sigma_T}.
		\end{align*}
		Let $z=-\frac{1}{2}$ and $\sigma>1$. Then, for all sufficiently small $\theta=\theta(\sigma)>1$ and $\delta=\delta(\sigma, \theta)>0$, all $f\in\spfc^z_T$, and all $g\in\spbc^\sigma_T$, we have
		\begin{align*}
			\|(fg)_x\|_{\spnc^z_T}
			\lesssim_\delta
			T^\delta
			\|f\|_{\spfc^z_T}
			\|g\|_{\spbc^\sigma_T}.
		\end{align*}
	\end{lem}
	\begin{proof}
		As in the previous proof, we begin by choosing extensions $f_{K_1}$ of $P_{K_1}f$ and $g_{K_2}$ of $P_{K_2}g$.
		According to Lemma \ref{l_ext}, we can assume that $f_{K_1}$ has temporal support in $[-2T,2T]$ and satisfies the estimate
		\begin{align}\label{e_nl_bnd_extf}
			\|f_{K_1}\|_{\spna^{K_1}}
			\lesssim
			\|P_{K_1} f\|_{\spnb^{K_1}_T}
		\end{align}
		for an implicit constant independent of $K_1$.
		The function $P_{K_2}g$ is extended to $[-2T,2T]$ such that the extension $g_{K_2}$ satisfies
		\begin{align}\label{e_nl_bnd_extg}
			\|g_{K_2}\|_{L^\infty}
			\lesssim
			\|P_{K_2}g\|_{L^\infty}.
		\end{align}
		Let $r\in\mathbb{R}$. Using the definitions of the $\spnc^r_T$- and $\spnb^{K_3}_T$-norms, inserting Littlewood-Paley projectors to $f$ and $g$, and making use of the fact that $f_{K_1}g_{K_2}$ is an extension of $P_{K_1}fP_{K_2}g$, we arrive at
		\begin{align}\label{e_nl_bnd_red}
			\|(fg)_x\|_{\spnc^r_T}^2
			\leq
			\sum_{K_3}
			\left[
			K_3^r
			\sum_{K_1, K_2}
			\|P_{K_3} (f_{K_1}g_{K_2})_x\|_{\spna^{K_3}}
			\right]^2.
		\end{align}
		
		In the following, we will bound the right-hand side of \eqref{e_nl_bnd_red} by considering different frequency interactions separately. For this, we will split the domain of summation into the following four sets
		\begin{align*}
		\Sigma_1
		&:=
		\{K_2\sim K_3\gg K_1\},
		&&\Sigma_2
		:=
		\{K_1\sim K_3\gg K_2\},
		\\
		\Sigma_3
		&:=
		\{K_1, K_2, K_3 \leq 64\},
		&&\Sigma_4
		:=
		\{K_1\sim K_2\gtrsim K_3\}.
		\end{align*}
		
		\textit{Summation over $\Sigma_1$.}
		We define $\mathfrak{t}=t-t_{K_3}$ and $\mathfrak{m}=\tau-\omega(\xi)$. Using the definition of the $\spna^{K_3}$-norm, it follows
		\begin{align*}
			\|P_{K_3}(f_{K_1}g_{K_2})_x\|_{\spna^{K_3}}
			\!\leq\!
			K_3 \sup_{t_{K_3}}\!\sum_{n}\!\|(\mathfrak{m}+iK_3^{\theta})^{-1}\chi_{K_3}(\xi)\mathcal{F}_{t,x}\!\left[\eta_0(K_3^{\theta}\mathfrak{t})f_{K_1}g_{K_2}\right]\!\|_{\spx^{K_3}}.
		\end{align*}
		Next, we fix a smooth function $\gamma$ (as in Lemma \ref{l_partgamma}) such that translations of $\gamma$ constitute a partition of unity, i.e.\@ $\sum_{n\in\mathbb{Z}}\gamma(\cdot+n)=1$. As before, we have
		\begin{align}\label{e_nl_bnd_gamc1}
			\gamma(2^{3\theta+1}K_3^{\theta}\mathfrak{t} - n)
			=
			\gamma(2^{3\theta+1}K_3^{\theta}\mathfrak{t} - n)
			\eta_0(K_1^{\theta}\mathfrak{t} - n).
		\end{align}
		We define
		\begin{align*}
			F_{K_1}(\tau,\xi)
			&:=
			\mathcal{F}\left[
			\gamma(2^{3\theta+1}K_3^{\theta}\mathfrak{t} - n)
			\eta_0(K_1^{\theta}\mathfrak{t} - n)\eta_0(K_3^{\theta}\mathfrak{t})f_{K_1}(t,x)
			\right](\tau,\xi),
			\\
			G_{K_2}(\tau,\xi)
			&:=
			\mathcal{F}\left[
			g_{K_2}(t,x)
			\right](\tau,\xi).
		\end{align*}
		Inserting the partition of unity and applying \eqref{e_nl_bnd_gamc1}, we obtain
		\begin{align*}
			\|P_{K_3}(f_{K_1}g_{K_2})_x\|_{\spna^{K_3}}
			\!\leq\!
			K_3\sup_{t_{K_3}}\!\sum_{n}\!\|(\mathfrak{m}+iK_3^{\theta})^{-1}\chi_{K_3}(\xi)\left(F_{K_1}\ast G_{K_2}\right)(\tau,\xi)\|_{\spx^{K_3}}.
		\end{align*}
		Note that $F_{K_1}$ vanishes for all but at most $2^{3\theta+4}$ many values of $n$ (see Lemma \ref{l_partgamma}). Hence, we can bound the sum in $n$ by a constant times its supremum. 
		Combining this with the definition of the $\spx^{K_3}$-norm and grouping small modulations together (as in the proof of Lemma \ref{l_nl_loc}), we conclude
		\begin{align*}
			\|P_{K_3}(f_{K_1}g_{K_2})_x\|_{\spna^{K_3}}
			\!\lesssim\!
			K_3\sup_{t_{K_3}, n}\!\sum_{L_3}\!L_3^{-\frac{1}{2}}
			\|\eta_{L_3}(\mathfrak{m}_3)\chi_{K_3}(\xi_3)\left(F_{K_1}\ast G_{K_2}\right)\|_{L^2},
		\end{align*}
		where the sum in $L_3$ starts at $\lfloor K_3^\theta\rfloor_\mathbb{D}$ and we interpret $\eta_{L_3}$ as $\eta_{\leq\lfloor K_3^\theta\rfloor_\mathbb{D}}$ if $L_3=\lfloor K_3^\theta\rfloor_\mathbb{D}$.
		
		Recall that $\eta_{L_3}(\mathfrak{m})\chi_{K_3}(\xi)\leq 1$ holds. Again, we localize $F_{K_1}$ with respect to its modulation variable, i.e.\@ we set
		\begin{align*}
			F_{L_1,K_1}(\tau,\xi)
			:=
			\begin{cases}
				\eta_{L_1}(\mathfrak{m})F_{K_1}(\tau,\xi)
				&\text{if }L_1>\lfloor K_3^\theta\rfloor_\mathbb{D},
				\\
				\eta_{\leq L_1}(\mathfrak{m})F_{K_1}(\tau,\xi)
				&\text{if }L_1=\lfloor K_3^\theta\rfloor_\mathbb{D}.
			\end{cases}
		\end{align*}
		Consequently, it follows
		\begin{align*}
			\|P_{K_3}(f_{K_1}g_{K_2})_x\|_{\spna^{K_3}}
			\lesssim
			K_3\sup_{t_{K_3}, n}\sum_{L_1,L_3}L_3^{-\frac{1}{2}}
			\|F_{L_1,K_1}\ast G_{K_2}\|_{L^2},
		\end{align*}
		where the sum in $L_1$ starts at $\lfloor K_3^\theta\rfloor_\mathbb{D}$. At this point, Plancherel's theorem and Hölder's inequality lead to
		\begin{align*}
			\|P_{K_3}(f_{K_1}g_{K_2})_x\|_{\spna^{K_3}}
			\lesssim
			K_3\sup_{t_{K_3}, n}\sum_{L_3}L_3^{-\frac{1}{2}}S^0_1\|g_{K_2}\|_{L^\infty}
		\end{align*}
		with
		\begin{align*}
			S_1^\beta
			:=
			\sum_{L_1\geq \lfloor K_3^\theta\rfloor_\mathbb{D}}L_i^{\beta}\|F_{L_1,K_1}(\tau,\xi)\|_{L^2}.
		\end{align*}
		Now, we evaluate the sum in $L_3$, use Corollary \ref{c_regtime} applied to $S^0_1$ (choosing $\delta=\delta'=\delta''\in(0,\frac{1}{4})$), and apply Lemma \ref{l_locstab} to the appearing $\spx^{K_1}$-norm. Finally, together with the choice of the extensions $f_{K_1}$ and $g_{K_2}$, see \eqref{e_nl_bnd_extf} and \eqref{e_nl_bnd_extg}, we get
		\begin{align*}
			\|P_{K_3}(f_{K_1}g_{K_2})_x\|_{\spna^{K_3}}
			\lesssim_\delta
			T^\delta K_3^{(1-\theta(1-\delta))}\|P_{K_1}f\|_{\spfb^{K_1}_T}\|P_{K_2}g\|_{L^\infty}.
		\end{align*}
		Assume $\delta\leq\frac{\theta-1}{1+\theta}$. Then, we have $1-\theta(1-\delta)\leq-\delta$ and bound the contribution of $\Sigma_1$ for $r\geq 0$ by
		\begin{align}\label{e_nl_bnd_c1}
			\sum_{K_3}
			\left[
			K_3^r
			\sum_{K_1, K_2}1_{\Sigma_1}
			\|P_{K_3} (f_{K_1}g_{K_2})_x\|_{\spna^{K_3}}
			\right]^2
			\lesssim_\delta
			T^{2\delta}\|f\|_{\spfc^0_T}^2\|g\|_{\spbc^r_T}^2
		\end{align}
		and for $r\leq 0$ by
		\begin{align}\label{e_nl_bnd_c11}
			\sum_{K_3}
			\left[
			K_3^r
			\sum_{K_1, K_2}1_{\Sigma_1}
			\|P_{K_3} (f_{K_1}g_{K_2})_x\|_{\spna^{K_3}}
			\right]^2
			\lesssim_\delta
			T^{2\delta}\|f\|_{\spfc^r_T}^2\|g\|_{\spbc^0_T}^2.
		\end{align}
		
		\textit{Summation over $\Sigma_2$.} We can repeat the proof given above for $\Sigma_1$ -- only the summation in the last step is different. For $r\in\mathbb{R}$, it follows
		\begin{align}\label{e_nl_bnd_c2}
			\sum_{K_3}
			\left[
			K_3^r
			\sum_{K_1, K_2}1_{\Sigma_2}
			\|P_{K_3} (f_{K_1}g_{K_2})_x\|_{\spna^{K_3}}
			\right]^2
			\lesssim_\delta
			T^{2\delta}\|f\|_{\spfc^r_T}^2\|g\|_{\spbc^0_T}^2.
		\end{align}
		
		\textit{Summation over $\Sigma_3$.}
		In this case we proceed as before hiding all unwanted powers of $K_1^*$ in the implicit constant. We obtain (with an implicit constant depending on a lower bound of $r$)
		\begin{align}\label{e_nl_bnd_c3}
			\sum_{K_3}
			\left[
			K_3^r
			\sum_{K_1, K_2}1_{\Sigma_3}
			\|P_{K_3} (f_{K_1}g_{K_2})_x\|_{\spna^{K_3}}
			\right]^2
			\lesssim_\delta
			T^{2\delta}\|f_1\|_{\spfc^0_T}^2\|g\|_{\spbc^0_T}^2.
		\end{align}
		
		\textit{Summation over $\Sigma_4$.} As in the proof of Lemma \ref{l_nl_loc}, the contribution of $\Sigma_4$ requires some non-trivial modifications of the arguments used for $\Sigma_1$. Let $\gamma$ be a smooth function as in Lemma \ref{l_partgamma} satisfying $\sum_{n\in\mathbb{Z}}\gamma(\cdot+n)=1$. Without loss of generality, assume that $K_1\geq K_2$ holds. Then, we have
		\begin{align}\label{e_nl_bnd_gamc4}
			\gamma(2K_1^{\theta}\mathfrak{t}-n)
			=
			\gamma(2K_1^{\theta}\mathfrak{t}-n)
			\eta_0(K_1^{\theta}\mathfrak{t}-n)
		\end{align}
		 for all $n\in\mathbb{N}$. Again, we use the notation $\mathfrak{t}=t-t_{K_3}$ and $\mathfrak{m}=\tau-\omega(\xi)$. Together with the definitions of the $\spna^{K_3}$- and $\spx^{K_3}$-norms, we conclude
		\begin{align*}
			&\|P_{K_3}(f_{K_1}g_{K_2})_x\|_{\spna^{K_3}}
			\\&\qquad\lesssim
			K_3\sup_{t_{K_3}}\sum_{L_3}L_3^{\frac{1}{2}}\|(\mathfrak{m}+iK_3^{\theta})^{-1}\eta_{L_3}(\mathfrak{m})\chi_{K_3}(\xi)\mathcal{F}\left[\eta_0(K_3^{\theta}\mathfrak{t})f_{K_1}g_{K_2}\right]\|_{L^2}.
		\end{align*}
		Next, we group small modulations, i.e.\@ all summands with $L_3\leq \lfloor K_3^\theta\rfloor_\mathbb{D}$. Understanding $\eta_{L_3}$ as $\eta_{\leq\lfloor K_3^\theta\rfloor_\mathbb{D}}$ if $L_3=\lfloor K_3^\theta\rfloor_\mathbb{D}$, this leads to
		\begin{align}\label{e_nl_bnd_c4red}
			&\|P_{K_3}(f_{K_1}g_{K_2})_x\|_{\spna^{K_3}}\nonumber
			\\&\qquad\lesssim
			K_3\sup_{t_{K_3}}\sum_{L_3}L_3^{-\frac{1}{2}}\|\eta_{L_3}(\mathfrak{m})\chi_{K_3}(\xi)\mathcal{F}_{t,x}\left[\eta_0(K_3^{\theta}\mathfrak{t})f_{K_1}g_{K_2}\right]\|_{L^2},
		\end{align}
		where the sum runs over $L_3\geq\lfloor K_3^\theta\rfloor_\mathbb{D}$.
		
		We define
		\begin{align*}
			F_{K_1}(\tau,\xi)
			&:=
			\mathcal{F}\left[
			\gamma(2K_1^{\theta}\mathfrak{t} - n)
			\eta_0(K_1^{\theta}\mathfrak{t} - n)\eta_0(K_3^{\theta}\mathfrak{t})f_{K_1}(t,x)
			\right](\tau,\xi),
			\\
			G_{K_2}(\tau,\xi)
			&:=
			\mathcal{F}\left[g_{K_1}(t,x)
			\right](\tau,\xi),
		\end{align*}
		as well as
		\begin{align*}
			F_{L_1,K_1}(\tau,\xi)
			:=
			\begin{cases}
				\eta_{L_1}(\mathfrak{m})F_{K_1}(\tau,\xi)
				&\text{if }L_1>\lfloor K_1^\theta\rfloor_\mathbb{D},
				\\
				\eta_{\leq L_1}(\mathfrak{m})F_{K_1}(\tau,\xi)
				&\text{if }L_1=\lfloor K_1^\theta\rfloor_\mathbb{D}.
			\end{cases}
		\end{align*} 
		According to Lemma \ref{l_partgamma}, the function $F_{K_1}$ vanishes for all but at most $16K_1^\theta K_3^{-\theta}$ many values of $n$. Now, we insert the identity \eqref{e_nl_bnd_gamc4} into \eqref{e_nl_bnd_c4red} and use the previous definitions as well as the estimate $\eta_{L_3}\chi_{K_3}\leq1$, which all in all leads to
		\begin{equation*}
			\|P_{K_3}(f_{K_1}g_{K_2})_x\|_{\spna^{K_3}}
			\lesssim
			K_1^{\theta}K_3^{(1-\theta)}\sup_{t_{K_3},n}\sum_{L_1,L_3}L_3^{-\frac{1}{2}}\|F_{L_1,K_1}\ast G_{K_2}\|_{L^2}.
		\end{equation*}
		Together with Plancherel's theorem and Hölder's inequality, it follows
		\begin{align*}
			\|P_{K_3}(f_{K_1}f_{K_2})_x\|_{\spna^{K_3}}
			\lesssim
			K_1^{\theta}K_3^{(1-\frac{3}{2}\theta)}\sup_{t_{K_3},n}S^0_1\|g_{K_2}\|_{L^\infty}
		\end{align*}
		with
		\begin{align*}
			S_1^\beta
			:=
			\sum_{L_1\geq \lfloor K_1^\theta\rfloor_\mathbb{D}}L_1^{\beta}\|F_{L_1,K_1}(\tau,\xi)\|_{L^2}.
		\end{align*}
		We apply Corollary \ref{c_regtime} (choosing $\delta=\delta'=\delta''\in(0,\frac{1}{4})$) to $S^0_1$ and Lemma \ref{l_locstab} to the appearing $\spx^{K_1}$-norm. Then, together with the choice of the extensions $f_{K_1}$ and $g_{K_2}$, see \eqref{e_nl_bnd_extf} and \eqref{e_nl_bnd_extg}, we obtain
		\begin{align*}
			\|P_{K_3}(f_{K_1}g_{K_2})_x\|_{\spna^{K_3}}
			\lesssim
			T^\delta K_1^{\theta(\frac{1}{2}+\delta)}K_3^{(1-\frac{3}{2}\theta)}\|f\|_{\spfb_T^{K_1}}\|P_{K_2}g\|_{L^\infty}.
		\end{align*}
		For $r\geq 0$ and $\sigma>\frac{1}{2}$, we choose $\theta>1$ and then $\delta>0$ such that $\theta(\frac{1}{2}+\delta)<\sigma$ holds. We conclude
		\begin{align}\label{e_nl_bnd_c4}
			\sum_{K_3}
			\left[
			K_3^r
			\sum_{K_1, K_2}1_{\Sigma_4}
			\|P_{K_3} (f_{K_1}g_{K_2})_x\|_{\spna^{K_3}}
			\right]^2
			\lesssim_\delta
			T^{2\delta}\|f\|_{\spfc^r_T}^2\|g\|_{\spbc^\sigma_T}^2.
		\end{align}
		The same estimate also holds for $r=z=-\frac{1}{2}$ and $\sigma>1$.
		Now, the claimed estimates follow directly from \eqref{e_nl_bnd_red} and the estimates in \eqref{e_nl_bnd_c1}, \eqref{e_nl_bnd_c11}, \eqref{e_nl_bnd_c2}, \eqref{e_nl_bnd_c3}, and \eqref{e_nl_bnd_c4}.
	\end{proof}
\section{Energy estimates}\label{s_ee}
	The goal of this section is to prove three low-regularity energy estimates for classical solutions to \eqref{eq_bo_split}. Note that the existence of such solutions will follow from Section \ref{ss_lwp_highreg}. We show an estimate on the $\spec^s_T$-norm of solutions, an estimate on the $\spec^z_T$-norm, $z=-\frac{1}{2}$, for differences of solutions, and, assuming more regularity, we also prove an estimate on the $\spec^s_T$-norm for differences of solutions.
	Below, we state the precise estimates:
	
	\begin{lem}\label{l_ee_i}
		Let $s\geq s_0>\frac{1}{4}$, $\varepsilon>0$, $\zeta>0$, $T\in(0,1)$, and $N\in\mathbb{D}$. Fix $\theta=\theta(s_0)>1$ and $\delta>0$ sufficiently small. Let $b$ and $f$ satisfy \eqref{as_b_f}. Then, every classical solution $u$ to \eqref{eq_bo_split} satisfies
		\begin{align}\label{e_ee_i}
			\begin{split}
				\|u\|_{\spec^s_T}^2
				\!\lesssim_{\delta,\varepsilon,\zeta}\!
				\|u_0\|_{H^s}^2
				\!&+\!
				(T+TN^{2}+N^{-\zeta})\|u\|_{\spec^s_T}^2
				\|u\|_{\spec^{s_0}_T}
				\!+\!
				T\|u\|_{\spec^s_T}^2
				\\&+\!
				T\|u\|_{\spec^s_T}^2
				(
				\|b\|_{\spbc^{s+1+\varepsilon}_T}
				\!+\!
				\|f\|_{L^\infty H^{s+\varepsilon}}
				)
				\\\!&+\!
				T^\delta\|u\|_{\spfc^s_T}^2
				(
				\|u\|_{\spfc^{s_0}_T}^2
				\!+\!
				\|u\|_{\spfc^{s_0}_T}\|b\|_{\spbc^{s}_T}
				)
				\\\!&+\!
				T\|f\|_{L^\infty H^{s+\varepsilon}}^2
				.
			\end{split}
		\end{align}
	\end{lem}
	\begin{lem}\label{l_ee_ii}
		Let $s>\frac{1}{2}$, $z=-\frac{1}{2}$, $\varepsilon>0$, $\zeta>0$, $T\in(0,1)$, and $N\in\mathbb{D}$. Fix $\theta=\theta(s)>1$ and $\delta>0$ sufficiently small. Let $b$ and $f$ satisfy \eqref{as_b_f}. Then, the difference $v=u_1-u_2$ of classical solutions $u_1$ and $u_2$ to \eqref{eq_bo_split} satisfies
		\begin{align}\label{e_ee_ii}
			\begin{split}
				\|v\|_{\spec^z_T}^2
				\!\lesssim_{\delta,\varepsilon,\zeta}\!
				\|v_0\|_{H^z}^2
				\!&+\!
				(T
				\!+\!
				TN^2
				\!+\!
				N^{-\zeta})
				\|v\|_{\spec^z_T}^2
				(
				\|u_1\|_{\spec^s_T}
				\!+\!
				\|u_1\|_{\spec^s_T}
				)
				\\\!&+\!
				T\|v\|_{\spec_T^z}^2
				(
				\|b\|_{\spbc^{1+\varepsilon}_T}
				\!+\!
				\|f\|_{L^\infty H^s}
				)
				\\\!&+\!
				T^\delta\|v\|_{\spfc^z_T}^2
				(
				\|u_1\|_{\spfc^s_T}
				\!+\!
				\|u_2\|_{\spfc^s_T}
				)
				(
				\|u_1\|_{\spfc^s_T}
				\!+\!
				\|u_2\|_{\spfc^s_T}
				\!+\!
				\|b\|_{\spbc^s_T}
				).
			\end{split}
		\end{align}
	\end{lem}
	
	\begin{lem}\label{l_ee_iii}
		Let $s>\frac{1}{2}$, $z=-\frac{1}{2}$, $\varepsilon>0$, $\zeta>0$, $T\in(0,1)$, and $N\in\mathbb{D}$. Fix $\theta=\theta(s)>1$ and $\delta>0$ sufficiently small. Let $b$ and $f$ satisfy \eqref{as_b_f}. Then, the difference $v=u_1-u_2$ of classical solutions $u_1$ and $u_2$ to \eqref{eq_bo_split} satisfies
		\begin{align}\label{e_ee_iii}
			\begin{split}
				\|v\|_{\spec^s_T}^2
				\!\lesssim_{\delta,\varepsilon,\zeta}\!
				\|v_0\|_{H^s}^2
				\!&+\!
				(T
				\!+\!
				TN^2
				\!+\!
				N^{-\zeta})
				\|v\|_{\spec^s_T}^2
				(
				\|v\|_{\spec^s_T}
				\!+\!
				\|u_1\|_{\spec^s_T}
				\!+\!
				\|u_2\|_{\spec^s_T}
				)
				\\\!&+\!
				T\|v\|_{\spec_T^s}^2
				(
				\|b\|_{\spbc^{1+s+\varepsilon}_T}
				\!+\!
				\|f\|_{L^\infty H^s}
				)
				\\\!&+\!
				T\|v\|_{\spec_T^z}\|v\|_{\spec_T^s}
				\|u_2\|_{\spec_T^{s+1}}
				\\\!&+\!
				T^\delta\|v\|_{\spfc^s_T}^2
				(
				\|v\|_{\spfc^{s}_T}
				\!+\!
				\|u_2\|_{\spfc^s_T}
				)
				(
				\|u_1\|_{\spfc^s_T}
				\!+\!
				\|u_2\|_{\spfc^s_T}
				\!+\!
				\|b\|_{\spbc^s_T}
				)
				\\\!&+\!
				T^\delta\|v\|_{\spfc^z_T}\|v\|_{\spfc^s_T}\|u_2\|_{\spfc^{s+1}_T}
				(
				\|u_1\|_{\spfc^s_T}
				\!+\!
				\|u_2\|_{\spfc^s_T}
				\!+\!
				\|b\|_{\spbc^{s+1}_T}
				).
			\end{split}
		\end{align}
	\end{lem}
		
	The estimates in all three lemmata are proven similarly to classical energy estimates. However, since the function spaces involve frequency-dependent time-localizations, the proofs are quite tedious. In particular, we need to use the equation \eqref{eq_bo_split} twice (by means of integration by parts), which leads to some quadrilinear terms. In Section \ref{ss_stqe} we prove four estimates for such quadrilinear terms. Sections \ref{ss_ee_i}, \ref{ss_ee_ii}, and \ref{ss_ee_iii} are dedicated to the proofs of Lemma \ref{l_ee_i}, \ref{l_ee_ii}, and \ref{l_ee_iii}.
	

	\subsection{Some quadrilinear estimates}\label{ss_stqe}
	Given a symbol $\sigma:\Gamma_4\rightarrow\mathbb{R}$ and a sufficiently nice function $u_i:[0,T]\times\mathbb{R}\rightarrow\mathbb{R}$, $i\in[4]$, we define the Fourier multiplier operator by
	\begin{align*}
		M_{\sigma}[u_1,u_2,u_3,u_4](T)
		:=
		\int_0^T\int_{\Gamma_4}
		\sigma(\xi)
		\prod_{i=1}^4\hat{u}_i(r,\xi_i)d\gamma_4(\xi)dr.
	\end{align*}
	Here, $\hat{u}_i$ denotes the spatial Fourier transform of $u_i$ and $\Gamma_4$ and $\gamma_4$ are defined in Section \ref{ss_res}.
	In this work, the symbol $\sigma$ will always be a product of powers of dyadic numbers $K_i$ and localizations $\chi_{K_i}$ for $i\in[4]$.
	In particular, such $\sigma$ does not change much on dyadic scales. For $K=(K_1,K_2,K_3,K_4)\in\mathbb{D}^4$, we define 
	\begin{align*}
		\sigma_{\mathbb{D}}(K)
		:=
		\sup\{|\sigma(\xi)|\,:
		\xi\in\Gamma_4 \text{ and }\forall i\in[4]:\xi_i\in\text{supp }\chi_{K_i}\}.
	\end{align*}
	The following two lemmata provide bounds on $M_\sigma$.
	\begin{lem}\label{l_stqe}
		Let $\theta>1$ and $\delta\in(0,\frac{1}{2})$. Then, for all $T\in(0,1)$ and all $u_j\in\spfb^{K_j}_T$ for $j\in[4]$, we have
		\begin{align*}
			|M_\sigma[u_1,u_2,u_3,u_4](T)|
			\lesssim_\delta
			T^\delta\sum_{K\in\mathbb{D}^4}
			\sigma_{\mathbb{D}}(K)(K_1^*)^{(\theta-1)\frac{1}{2}}(K_4^*)^{\frac{1}{2}}\prod_{j=1}^{4}\|P_{K_j}u_j\|_{\spfb^{K_j}_T}
		\end{align*}
		and
		\begin{align*}
			|M_\sigma[u_1,u_2,u_3,u_4](T)|
			\lesssim_\delta
			T^\delta\sum_{K\in\mathbb{D}^4}
			\sigma_{\mathbb{D}}(K)(K_3^*K_4^*)^{\frac{1}{2}}\prod_{j=1}^{4}\|P_{K_j}u_j\|_{\spfb^{K_j}_T}.
		\end{align*}
	\end{lem}
	\begin{lem}\label{l_stqe_bnd}
		Let $\theta>1$ and $\delta\in(0,\frac{1}{2})$. Then, for all $T\in(0,1)$, all $u_j\in\spfb^{K_j}_T$ for $j\in[3]$, and all $u_4\in L^\infty([0,T]\times\mathbb{R};\mathbb{R})$, we have
		\begin{align}
			|M_\sigma[u_1,u_2,u_3,u_4](T)|
			&\lesssim_\delta
			T^\delta\sum_{K\in\mathbb{D}^4}
			\sigma_{\mathbb{D}}(K)(K_1^\#)^\frac{\theta}{4}\prod_{j=1}^3\|P_{K_j}u_j\|_{\spfb^{K_j}_T}\|P_{K_4}u_4\|_{L^\infty}
		\end{align}
		and
		\begin{align}
			|M_\sigma[u_1,u_2,u_3,u_4](T)|
			&\lesssim_\delta
			T^\delta\sum_{K\in\mathbb{D}^4}
			\sigma_{\mathbb{D}}(K)(K_3^\#)^\frac{1}{2}\prod_{j=1}^3\|P_{K_j}u_j\|_{\spfb^{K_j}_T}\|P_{K_4}u_4\|_{L^\infty}.
		\end{align}
		Here, $K_1^\#$, resp. $K_3^\#$, denote the maximum, resp. minimum, of $K_1$, $K_2$, and $K_3$.
	\end{lem}
	\begin{proof}[Proof of Lemma \ref{l_stqe}]
		We begin by a reduction to dyadic scales. Decomposing each $u_j$ by using Littlewood-Paley projectors, we obtain
		\begin{align*}
			|M_\sigma[u_1,u_2,u_3,u_4](T)|
			\lesssim
			\sum_{K}\sigma_{\mathbb{D}}(K)
			\underset{=:Q(K)}{\underbrace{\int_0^T\int_{\Gamma_4}\prod_{j=1}^4\chi_{K_j}(\xi_j)|\hat{u}_j|(r,\xi_j)d\gamma_4(\xi)dr}}.
		\end{align*}
		Thus, the claim will follow if we prove
		\begin{equation}
			Q(K)
			\lesssim
			T^\delta (K_1^*)^{\frac{\theta-1}{2}}(K_4^*)^{\frac{1}{2}}\prod_{j=1}^4\|P_{K_j}u_j\|_{\spfb_T^{K_j}}.
		\end{equation}
		
		We continue by rewriting $Q(K)$. Choose extensions $v_{K_j}\in\spfa^{K_j}$ of $1_{[0,T]}P_{K_j}v_j$, where $v_j:=(|\hat{u}_j|)^{\vee}$ for $j\in[4]$. According to Lemma \ref{l_ext}, we can assume that each $v_{K_j}$ has temporal support in $[-2T,2T]$ and satisfies the estimate
		\begin{align}
			\|v_{K_j}\|_{\spfa^{K_j}}
			\lesssim
			\|P_{K_j}v_j\|_{\spfb^{K_j}_T}
			=
			\|P_{K_j}u_j\|_{\spfb^{K_j}_T}.
		\end{align}
		Moreover, fix a smooth function $\gamma$ with $\sum_{n\in\mathbb{Z}}\gamma^4(\cdot+n)=1$. For $n\in\mathbb{Z}$, we define
		\begin{align*}
			h_{K_j,n}(t)
			=
			\gamma((K_1^*)^{\theta}t-n)\eta_0(K_j^{\theta}t-n)
			=
			\gamma((K_1^*)^{\theta}t-n)
		\end{align*}
		and conclude
		\begin{align*}
			Q(K)
			=
			\int_\mathbb{R}1_{[0,T]}(r)\sum_{n\in\mathbb{Z}}\mathcal{F}_x\left[\prod_{j=1}^4h_{K_j,n}(r)v_{K_j}(r,\cdot)\right](0)dr.
		\end{align*}
		Clearly, we only need to consider integers $n$ satisfying $1_{[0,T]}(\cdot)\gamma((K_1^*)^{\theta}\cdot-n)\not\equiv0$ on $\mathbb{R}$. Hence, it suffices to study the contribution of the sums over the following two sets
		\begin{align*}
			M_1
			&:=
			\{n\in\mathbb{Z}:
			\gamma((K_1^*)^{\theta}\cdot-n)\equiv1_{[0,T]}(\cdot)\gamma((K_1^*)^{\theta}\cdot-n)\not\equiv0\},
			\\
			M_2
			&:=
			\{n\in\mathbb{Z}:
			\gamma((K_1^*)^{\theta}\cdot-n)\not\equiv1_{[0,T]}(\cdot)\gamma((K_1^*)^{\theta}\cdot-n)\not\equiv0\}.
		\end{align*}
		
		\textit{Summation over $M_1$}: In the following, we restrict the summation in $Q(K)$ to $n\in M_1$. Note that we can drop the indicator function and that $|M_1|$ is bounded by $T(K_1^*)^{\theta}$. We conclude
		\begin{align*}
			Q(K)|_{M_1}
			&=
			\sum_{n\in M_1}\int_\mathbb{R}\mathcal{F}_x\left[\prod_{j=1}^4h_{K_j,n}(r)v_{K_j}(r,\cdot)\right](0)dr
			\\&\lesssim
			T(K_1^*)^{\theta}\sup_{n\in M_1}\ast_{j=1}^4\mathcal{F}_{r,x}\left[h_{K_j,n}(r)v_{K_j}(r,x)\right](0,0).
		\end{align*}
		Next, we localize each term in the convolution above with respect to its modulation variable. For this, define
		\begin{align*}
			\phi_{L_j,K_j}(\tau,\xi)
			:=
			\begin{cases}
				\eta_{L_j}(\mathfrak{m})\mathcal{F}\left[h_{K_j,n}v_{K_j}\right](\tau,\xi)
				&\text{if }L_j>\lfloor (K_1^*)^\theta\rfloor_{\mathbb{D}},
				\\
				\eta_{\leq \lfloor L_j\rfloor}(\mathfrak{m})\mathcal{F}\left[h_{K_j,n}v_{K_j}\right](\tau,\xi)
				&\text{if }L_j=\lfloor (K_1^*)^\theta\rfloor_{\mathbb{D}}.
			\end{cases}
		\end{align*}
		As before, $\lfloor M\rfloor_\mathbb{D}$ denotes the largest dyadic number smaller than $M$ and we write $\mathfrak{m}=\tau-\omega(\xi)$. Then, we have
		\begin{align*}
			Q(K)|_{M_1}
			\lesssim
			T(K_1^*)^{\theta}\sup_{n\in M_1}\sum_{L\in\mathbb{D}^4}\ast_{j=1}^4\phi_{L_j,K_j}(0,0).
		\end{align*}
		Since each $\phi_{L_j,K_j}$ satisfies the assumptions of Lemma \ref{l_cnv}, we can apply \eqref{e_cnv_qua_imp}, whenever $K_3^*> 1$ holds, and obtain
		\begin{align*} 
			Q(K)|_{M_1}
			&\lesssim
			T(K_1^*)^{\theta-\frac{1}{2}}(K_4^*)^{\frac{1}{2}}\sup_{n\in M_1}\sum_{L\in\mathbb{D}^4}(L_1^*L_2^*L_4^*)^\frac{1}{2}\prod_{j=1}^4\|\phi_{L_j,K_j}\|_{L^2}
			\\&\lesssim
			T(K_1^*)^{\frac{\theta-1}{2}}(K_4^*)^{\frac{1}{2}}\prod_{j=1}^4\|v_{K_j}\|_{\spfb^{K_j}}.
		\end{align*}
		For $K_3^*=1$, an application of \eqref{e_cnv_qua_gen} leads to
		\begin{align*}
			Q(K)|_{M_1}
			\!\lesssim
			T(K_1^*)^{\theta}\!\sup_{n\in M_1}
			\!\sum_{L\in\mathbb{D}^4}\!(L_3^*L_4^*)^\frac{1}{2}\!\prod_{j=1}^4\|\phi_{L_j,K_j}\|_{L^2}
			\!\lesssim
			T(K_1^*)^{\frac{\theta-1}{2}}\!\prod_{j=1}^4\!\|P_{K_j}u_j\|_{\spfb^{K_j}_T}.
		\end{align*}
		
		\textit{Summation over $M_2$}: This case is similar to the previous one, which is why we only focus on the necessary modifications. First, note that $M_2$ has at most four elements. Thus, we have
		\begin{align*}
			Q(K)|_{M_2}
			\lesssim
			\sup_{n\in M_2}\ast_{j=1}^4\mathcal{F}_{r,x}\left[1_{[0,T]}(r)h_{K_j,n}(r)v_{K_j}(r,x)\right](0,0).
		\end{align*}
		Without loss of generality, we can assume $L_3^*=L_4$. Define $\phi_{L_j,K_j}$ as in the previous case for $j\in[3]$. The additional factor $1_{[0,T]}$ is put into the definition of $\phi_{L_4,K_4}$, which we define by
		\begin{align*}
			\phi_{L_4,K_4}(\tau,\xi)
			:=
			\begin{cases}
				\eta_{L_4}(\mathfrak{m})\mathcal{F}_{r,x}\left[1_{[0,T]}h_{K_4,n}v_{K_4}\right](\tau,\xi)
				&\text{if }L_4>\lfloor (K_1^*)^\theta\rfloor_\mathbb{D},
				\\
				\eta_{\leq L_4}(\mathfrak{m})\mathcal{F}_{r,x}\left[1_{[0,T]}h_{K_4,n}v_{K_4}\right](\tau,\xi)
				&\text{if }L_4=\lfloor (K_1^*)^\theta\rfloor_\mathbb{D}.
			\end{cases}
		\end{align*}
		Now, if $K_3^*> 1$, we use \eqref{e_cnv_qua_imp} as well as Corollary \ref{c_regtime} leading to
		\begin{align*}
			Q(K)|_{M_2}
			&\lesssim
			\sup_{n\in M_2}\sum_{L}\ast_{j=1}^4\phi_{L_j,K_j}(0,0)
			\\&\lesssim
			(K_1^*)^{-\frac{1}{2}}(K_4^*)^{\frac{1}{2}}\prod_{j=1}^3\|v_{K_j}\|_{\spfb_T^{K_j}}\sup_{n\in M_2}\sum_{L_4}\|\phi_{L_4,K_4}\|_{L^2}
			\\&\lesssim_\delta
			T^\delta(K_1^*)^{\delta-1}(K_4^*)^{\frac{1}{2}}\prod_{j=1}^4\|v_{K_j}\|_{\spfb_T^{K_j}}.
		\end{align*}
		As before, we use \eqref{e_cnv_qua_gen} in the case $K_3^*=1$.
		
		The proof for the second estimate in Lemma \ref{l_stqe} follows by replacing each application of \eqref{e_cnv_qua_imp} in the proof given above by an application of \eqref{e_cnv_qua_gen}.
	\end{proof}
	\begin{proof}[Proof of Lemma \ref{l_stqe_bnd}]
		We can repeat the proof of Lemma \ref{l_stqe}, but replace all applications of \eqref{e_cnv_qua_gen} and \eqref{e_cnv_qua_imp} by either \eqref{e_cnv_qua_bnd_one} or \eqref{e_cnv_qua_bnd_two}.
	\end{proof}
	\subsection{Proof of the first energy estimate}\label{ss_ee_i}
	
	In this section we prove Lemma \ref{l_ee_i}.
	Fix $s\geq s_0>\frac{1}{4}$, $T\in(0,1)$, and let $\theta>1$ be small enough to satisfy all restrictions appearing in this section. Choose $(u,b,f)$ such that $b$ and $f$ satisfy \eqref{as_b_f} and such that $u$ is a classical solution to \eqref{eq_bo_split}.
	
	Since $u$ is a solution of \eqref{eq_bo_split}, we obtain for $K_1\in\mathbb{D}$
	\begin{align*}
		\frac{d}{dr}\|P_{K_1}u(r,x)\|_{L^2_x}^2
		&=
		-2\int_\mathbb{R} P_{K_1}uP_{K_1}\left[\mathcal{H}u_{xx}+(uu)_x+(ub)_x+f\right]dx
		\\&=
		2\int_\mathbb{R} P_{K_1}^2u_xuu dx
		+
		2\int_\mathbb{R} P_{K_1}^2u_xub dx
		-
		2\int_\mathbb{R} P_{K_1}uP_{K_1}f dx.
	\end{align*}
	Integrating the equation above with respect to $r$ over $[0,t]$, taking the supremum over $t\in[0,T]$, multiplying by $K_1^{2s}$, and taking the sum over $K_1\in\mathbb{D}$, we get
	\begin{align*}
		\|u\|_{\spec^s_T}^2
		-
		\|u_0\|_{H^s}^2
		&\leq
		2\sum_{K_1}K_1^{2s}\sup_t\int_0^t\int_\mathbb{R}[P_{K_1}^2u_xuu](r,x)dxdr
		&=:I_1
		\\&+
		2\sum_{K_1}K_1^{2s}\sup_t\int_0^t\int_\mathbb{R}[P_{K_1}^2u(ub)_x](r,x)dxdr
		&=:I_2
		\\&+
		2\sum_{K_1}K_1^{2s}\sup_t\int_0^t\int_\mathbb{R}[P_{K_1}uP_{K_1}(-f)](r,x)dxdr
		&=:I_3.
	\end{align*}
	
	We begin by bounding $I_3$. 
	\begin{lem}\label{l_ee_trv}
		Let $s\geq 0$ and $\varepsilon>0$. Then, we have
		\begin{equation*}
			|I_3|
			\lesssim_\varepsilon
			T\|u\|_{L^\infty H^s}\|f\|_{L^\infty H^{s+\varepsilon}}.
		\end{equation*}
	\end{lem}
	\begin{proof}
		The inequality is a direct consequence of Hölder's inequality. We have
		\begin{align*}
			&|\sum_{K_1}K_1^{2s}\sup_t\int_0^t\int_\mathbb{R}[P_{K_1}uP_{K_1}(-f)](r,x)dxdr|
			\\ \leq\,
			&T\sum_{K_1}K_1^{2s}\|P_{K_1}u\|_{L^\infty L^2}\|P_{K_1}f\|_{L^\infty L^2}		
			\lesssim_\varepsilon
			T\|u\|_{L^\infty H^s}\|f\|_{L^\infty H^{s+\varepsilon}}.
		\end{align*}
	\end{proof}
	Next, we bound $I_2$.
	\begin{lem}\label{l_ee_trv_uub}
		Let $s\geq0$ and $\varepsilon>0$. Then, we have
		\begin{align*}
			|I_2|
			\lesssim_\varepsilon
			T\|u\|_{\spec_T^s}^2\|b\|_{\mathfrak{B}_T^{s+1+\varepsilon}}.
		\end{align*}
	\end{lem}
	\begin{proof}
		We apply Littlewood-Paley projectors to localize the second and third factor in the integrand in $I_2$. We get
		\begin{align*}
			|I_2|
			\leq
			\sum_{K_i}K_1^{2s}\sup_t\int_0^t\int_\mathbb{R}[P_{K_1}^2u_xP_{K_2}uP_{K_3}b](r,x)dxdr,
		\end{align*}
		where the sums is over all dyadic numbers $K_1$, $K_2$, and $K_3$.
		If $K_3^*=K_1$ or $K_3^*=K_2$, we conclude
		\begin{align*}
			|I_2|
			&\leq
			T\sum_{K_i}K_1^{2s+1}\|P_{K_1}u\|_{L^\infty L^2}\|P_{K_2}u\|_{L^\infty L^2}\|P_{K_3}b\|_{L^\infty L^\infty}
			\\&\lesssim_\varepsilon
			T\|u\|_{L^\infty H^s}^2\|b\|_{\spbc^{s+1+\varepsilon}_T}.
		\end{align*}
		Now, we consider the case $K_3^*=K_3$. Here, both high-frequency factors are Littlewood-Paley localizations of the same function. Hence, the change of variables $(\xi_1,\xi_2)\mapsto(\xi_2,\xi_1)$ leads to
		\begin{align*}
			&\int_\mathbb{R}P_{K_1}^2uP_{K_2}u_xP_{K_3}bdx
			\\=
			&\int_{\Gamma_3}\chi^2_{K_1}(\xi_1)\chi_{K_2}(\xi_2)\chi_{K_3}(\xi_3)(-i\xi_2)\hat{u}(\xi_1)\hat{u}(\xi_2)\hat{b}(\xi_3)d\gamma_3(\xi)
			\\=
			&\int_{\Gamma_3}\sigma(\xi)\hat{u}(\xi_1)\hat{u}(\xi_2)\hat{b}(\xi_3)d\gamma_3(\xi),
		\end{align*}
		where
		\begin{align*}
			\sigma(\xi)
			:=
			\frac{1}{2}[\chi^2_{K_1}(\xi_1)\chi_{K_2}(\xi_2)\chi_{K_3}(\xi_3)(-i\xi_2)
			+
			\chi^2_{K_1}(\xi_2)\chi_{K_2}(\xi_1)\chi_{K_3}(\xi_3)(-i\xi_1)].
		\end{align*}
		It follows easily (see also below) that $\sigma$ satisfies the uniform bound $|\sigma(\xi)|\lesssim K_3$. Thus, we obtain the bound
		\begin{align*}
			|I_2|
			&\leq
			\int_{\Gamma_3}\sigma(\xi)\hat{u}(\xi_1)\hat{u}(\xi_2)\hat{b}(\xi_3)d\gamma_3(\xi)
			\\
			&\leq
			T\sum_{K_i} K_1^{2s}K_3\|P_{K_1}u\|_{L^\infty L^2}\|P_{K_2}u\|_{L^\infty L^2}\|P_{K_3}b\|_{L^\infty L^\infty}
			\lesssim_\varepsilon
			T\|u\|_{\spec^s_T}^2\|b\|_{\spbc^{1+\varepsilon}_T}.
		\end{align*}
		We remark, that it is necessary to use the stronger $\spec^s_T$-norm in the last estimate.
	\end{proof}
	
	The remaining part of this section is concerned with estimating $I_1$. Proceeding similar to the previous lemma, we obtain the bound $|I_1|\lesssim \|u\|_{\spec^s_T}^3$ for $s>\frac{3}{2}$. To improve this bound, we proceed as follows: First, we symmetrize the appearing symbol $\sigma$. This is already sufficient to deal with the low-frequency contribution. To handle the high-frequency contributions, we perform another integration by parts in time. This is beneficial since the appearing terms involve an inverse power of the resonance function $\Omega_3$. Moreover, due to the quadratic nonlinearity, we obtain integrals involving four factors $u$ respectively $b$, which can be estimated with help of the estimates from Section \ref{ss_stqe}.
	
	We begin with symmetrizing the symbol of $I_1$. For $K_2,K_3\in\mathbb{D}$ let us define
	\begin{align*}
		\sigma_0(\xi)
		:=
		-i\xi_1\chi_{K_1}^2(\xi_1)\chi_{K_2}(\xi_2)\chi_{K_3}(\xi_3)
	\end{align*}
	and
	\begin{align*}
		\sigma_I(\xi_1,\xi_2,\xi_3)
		:=
		\begin{cases}
			\sigma_0(\xi_1,\xi_2,\xi_3)
			&\text{if }K_3^*=K_1,
			\\
			\frac{1}{2}[\sigma_0(\xi_1,\xi_2,\xi_3)+\sigma_0(\xi_3,\xi_2,\xi_1)]
			&\text{if }K_3^*=K_2,
			\\
			\frac{1}{2}[\sigma_0(\xi_1,\xi_2,\xi_3)+\sigma_0(\xi_2,\xi_1,\xi_3)]
			&\text{if }K_3^*=K_3.
		\end{cases}
	\end{align*}
	Then, applying Littlewood-Paley projectors to the second and third factor, we get
	\begin{align}\label{eq_ee_b}
		I_1
		=
		&\sum_{K_1}K_1^{2s}\sup_t\int_0^t\int_\mathbb{R}[P_{K_1}^2u_xuu](r,x)dxdr\nonumber
		\\ \leq\,
		&\sum_{K_i}K_1^{2s}\sup_t\int_0^t\int_{\Gamma_3}\sigma_0(\xi)\hat{u}(r,\xi_1)\hat{u}(r,\xi_2)\hat{u}(r,\xi_3)d\gamma_3(\xi)dr\nonumber
		\\ =\,
		&\sum_{K_i}K_1^{2s}\sup_t\int_0^t\int_{\Gamma_3}\sigma_I(\xi)\hat{u}(r,\xi_1)\hat{u}(r,\xi_2)\hat{u}(r,\xi_3)d\gamma_3(\xi)dr.
	\end{align}
	Here, the last equation holds due to the symmetry of the integrand. The symbol $\sigma_I$ is better than $\sigma_0$ since it satisfies the uniform bound $|\sigma_I(\xi)|\lesssim K_3^*$. This claim is verified easily; for example if $K_3^*=K_3$, then we write
	\begin{align*}
		i\left(\sigma_0(\xi_1,\xi_2,\xi_3)+\sigma_0(\xi_2,\xi_1,\xi_3)\right)
		&=
		[\chi_{K_1}^2(\xi_1)-\chi_{K_1}^2(\xi_2)]\chi_{K_2}(\xi_2)\chi_{K_3}(\xi_3)\xi_1
		\\&+
		\chi_{K_1}^2(\xi_2)[\chi_{K_2}(\xi_2)-\chi_{K_1}(\xi_1)]\chi_{K_3}(\xi_3)\xi_1
		\\&+
		\chi_{K_1}^2(\xi_2)\chi_{K_2}(\xi_1)\chi_{K_3}(\xi_3)[\xi_2+\xi_1]
	\end{align*}
	and easily deduce the estimate for each term on the right-hand side.
	
	Next, we want to perform integration by parts in time. Let $N\in\mathbb{D}$ denote a parameter. We split the domain of summation $K_1,K_2,K_3\in\mathbb{D}$ according to the relative sizes of the frequencies; more precisely, we set
	\begin{align*}
		\sum_{C_L}
		\!:=\!
		\sum_{K_i}1_{\{K_3^*=1\}},
		\quad
		\sum_{C_M}
		\!:=\!
		\sum_{K_i}1_{\{K_3^*>1, K_1^*\leq N\}},
		\quad
		\sum_{C_H}
		\!:=\!
		\sum_{K_i}1_{\{K_3^*>1, K_1^*> N\}}.
	\end{align*}
	Roughly speaking, the condition $K_3^*>1$ will guarantee that we can control the resonance function $\Omega_3$ well, whereas the condition $K_1^*>N$ allows us to choose the boundary term arbitrary small by increasing $N$.
	
	For functions $g$ and $h$, we use the abbreviations
	\begin{align*}
		\kappa_{3}[g](\xi)
		&:=
		\hat{g}(\xi_1)\hat{g}(\xi_2)\hat{g}(\xi_3),
		\\
		\kappa_{1,2}[g,h](\xi)
		&:=
		\hat{g}(\xi_1)\hat{h}(\xi_2)\hat{h}(\xi_3)+\hat{h}(\xi_1)\hat{g}(\xi_2)\hat{h}(\xi_3)+\hat{h}(\xi_1)\hat{h}(\xi_2)\hat{g}(\xi_3).
	\end{align*}
	With this notation and $u$ being a solution to \eqref{eq_bo_split}, we obtain
	\begin{align}\label{eq_ee_ibp}
		\left(-\kappa_{3}\left[u\right]\right)_t
		=
		\Omega_3\kappa_{3}\left[u\right]
		+
		\kappa_{1,2}\left[(uu)_x,u\right]
		+
		\kappa_{1,2}\left[(ub)_x,u\right]
		+
		\kappa_{1,2}\left[f,u\right].
	\end{align}
	Note that the integrand in \eqref{eq_ee_b} has the form $\sigma_I\kappa_{3}\left[u\right]$. Thus, if $\Omega_3$ does not vanish, we can multiply \eqref{eq_ee_ibp} by $\sigma_I\Omega_3^{-1}$ and see that the first term on the right-hand side will be $\sigma_I\kappa_{3}\left[u\right]$ - hence, we can use this new equation to rewrite the integrand.
	
	Starting from \eqref{eq_ee_b}, we split the sum over the frequency sizes and use \eqref{eq_ee_ibp} to rewrite all terms contributing to the sum over $C_H$. This implies the following inequality:
	\begin{align*}
		I_1
		\leq
		&\sum_{C_L}K_1^{2s}\sup_t\int_0^t\int_{\Gamma_3}(\sigma_I\kappa_{3}\left[u\right])(\xi)d\gamma_3(\xi)dr
		&=:E_{I,1}
		\\ +
		&\sum_{C_M}K_1^{2s}\sup_t\int_0^t\int_{\Gamma_3}(\sigma_I\kappa_{3}\left[u\right])(\xi)d\gamma_3(\xi)dr
		&=:E_{I,2}
		\\ +
		&\sum_{C_H}K_1^{2s}\sup_t\left[\int_{\Gamma_3}(\sigma_I\Omega_3^{-1}\kappa_{3}\left[u\right])(\xi)d\gamma_3(\xi)\right]_{r=0}^{t}
		&=:E_{I,3}
		\\ +
		&\sum_{C_H}K_1^{2s}\sup_t\int_0^t\int_{\Gamma_3}(\sigma_I\Omega_3^{-1}\kappa_{1,2}\left[(uu)_x,u\right])(\xi)d\gamma_3(\xi)dr
		&=:E_{I,4}
		\\ +
		&\sum_{C_H}K_1^{2s}\sup_t\int_0^t\int_{\Gamma_3}(\sigma_I\Omega_3^{-1}\kappa_{1,2}\left[(ub)_x,u\right])(\xi)d\gamma_3(\xi)dr
		&=:E_{I,5}
		\\ +
		&\sum_{C_H}K_1^{2s}\sup_t\int_0^t\int_{\Gamma_3}(\sigma_I\Omega_3^{-1}\kappa_{1,2}\left[f,u\right])(\xi)d\gamma_3(\xi)dr
		&=:E_{I,6}.
	\end{align*}
	Note that the integrands appearing above (up to the factor $\sigma_I$) are invariant under permutations of $\xi_1$, $\xi_2$, and $\xi_3$. Moreover, the modulus of $\sigma_I$ is always bounded by $K_3^*$. These observations together with $s\geq 0$ allow us to restrict to $K_1\sim K_2\gtrsim K_3$ for the rest of this section.	
	
	We continue by estimating the low-frequency terms $E_{I,1}$ and $E_{I,2}$, the boundary term $E_{I,3}$, and the term involving the forcing $E_{I,6}$.
	\begin{lem}\label{l_ee_tri}
		Let $s\geq 0$ and $\zeta\in(0,\frac{1}{2})$. Then, we have
		\begin{align*}
			|E_{I,1}|
			&\lesssim
			T\|u\|_{\spec^s_T}^2\|u\|_{L^\infty L^2},
			\\
			|E_{I,2}|
			&\lesssim
			TN^{2}\|u\|_{\spec^s_T}^2\|u\|_{L^\infty L^2},
			\\
			|E_{I,3}|
			&\lesssim_\zeta
			N^{-\zeta}\|u\|_{L^\infty H^s}^2\|u\|_{L^\infty L^2},
			\\
			|E_{I,6}|
			&\lesssim
			T\|u\|_{L^\infty H^s}^2\|f\|_{L^\infty H^{s}}.
		\end{align*}
	\end{lem}
	\begin{proof}
		Applying the uniform bound $|\sigma_I(\xi)|\lesssim K_3$ as well as Young's, Hölder's, and Bernstein's inequality, we obtain
		\begin{align*}
			|E_{I,1}|
			\lesssim
			T\sum_{C_L}K_1^{2s}K_3^{\frac{3}{2}}
			\prod_{j=1}^3\|P_{K_j}u\|_{L^\infty L^2}
			\lesssim
			T\|u\|_{\spec^s_T}^2\|u\|_{L^\infty L^2},
		\end{align*}
		where we used $K_3=K_3^*=1$ in the last estimate. Similarly, we obtain the bound for $E_{I,2}$ using $K_1^*\leq N$, whereas the claim for $E_{I,6}$ follows using $|\Omega_3(\xi)|\gtrsim K_1^*K_3^*$ (see \eqref{e_res_asy}).
		
		The bound on $E_{I,3}$ follows from
		\begin{align*}
			|E_{I,3}|
			\lesssim
			\sum_{C_H}
			K_1^{(2s-1)}K_3^{\frac{1}{2}}
			\prod_{j=1}^3\|P_{K_j}u\|_{L^\infty L^2}
			\lesssim_{\zeta}
			N^{-\zeta}\|u\|_{L^\infty H^s}^2\|u\|_{L^\infty L^2},
		\end{align*}
		where we used $K_1^*\leq N$ as well as $|\Omega_3(\xi)|\gtrsim K_1^*K_3^*$.
	\end{proof}
	
	We continue by bounding $E_{I,4}$.
	\begin{lem}\label{l_ee_qua}
		Let $s\geq s_0>\frac{1}{4}$. Then, for all sufficiently small $\theta=\theta(s_0)>1$ and $\delta\in(0,\frac{1}{2})$, we have
		\begin{align*}
			|E_{I,4}|
			&\lesssim_\delta
			T^\delta\|u\|_{\spfc^s_T}^2\|u\|_{\spfc^{s_0}_T}^2.
		\end{align*}
	\end{lem}
	\begin{proof}
		We split $E_{I,4}=S_1+S_2+S_3$, where the right-hand side is defined by
		\begin{align*}
			S_1
			&:=
			\sum_{C_H}K_1^{2s}\sup_t\int_0^t\int_{\Gamma_3}\left[\frac{\sigma_I}{\Omega_3}(\widehat{u_x^2}\hat{u}\hat{u})\right](\xi)d\gamma_3(\xi)dr,
			\\
			S_2
			&:=
			\sum_{C_H}K_1^{2s}\sup_t\int_0^t\int_{\Gamma_3}\left[\frac{\sigma_I}{\Omega_3}(\hat{u}\widehat{u^2_x}\hat{u})\right](\xi)d\gamma_3(\xi)dr,
			\\
			S_3
			&:=
			\sum_{C_H}K_1^{2s}\sup_t\int_0^t\int_{\Gamma_3}\left[\frac{\sigma_I}{\Omega_3}(\hat{u}\hat{u}\widehat{u^2_x})\right](\xi)d\gamma_3(\xi)dr.
		\end{align*}
		Recall that we assume $K_1\sim K_2\gtrsim K_3$. In particular, the change of variables $(\xi_1,\xi_2,\xi_3)\mapsto(\xi_2,\xi_1,\xi_3)$ shows $S_1=S_2$. Hence, it suffices to estimate $S_1$ and $S_3$.
		Define
		\begin{align*}
			\Pi(K_i,K_j,K_g,K_h)
			:=
			\|P_{K_i}u\|_{\spfb^{K_i}_T}
			\|P_{K_j}u\|_{\spfb^{K_j}_T}
			\|P_{K_g}u\|_{\spfb^{K_g}_T}
			\|P_{K_h}u\|_{\spfb^{K_h}_T}.
		\end{align*}
		
		We begin by estimating $S_1$. Writing $\widehat{u^2_x}=(\hat{u}\ast\hat{u})_x$ and localizing the new variables $\xi_a$ and $\xi_b$ to dyadic frequency ranges of size $K_a$ and $K_b$, we get
		\begin{align*}
			|S_1|
			\lesssim
			\sum_{K_j}K_1^{2s}\sup_t\int_0^t\int_{\Gamma_4}\frac{\sigma_I}{\Omega_3}&(\xi_{ab},\xi_2,\xi_3)(i\xi_{ab})\chi_{K_a}(\xi_a)\chi_{K_b}(\xi_b)
			\\
			&\times\hat{u}(\xi_a)\hat{u}(\xi_b)\hat{u}(\xi_2)\hat{u}(\xi_3)d\gamma_4(\xi_a,\xi_b,\xi_2,\xi_3)dr.
		\end{align*}
		Note that the factor $\frac{\sigma_I}{\Omega_3}(\xi_{ab},\xi_2,\xi_3)(i\xi_{ab})$ is of size one due to our choice of $\sigma_I$ and \eqref{e_res_asy}.
		
		Now, we differentiate three cases. First, let us consider the case $K_a\sim K_b\gtrsim K_1$. Then, Lemma \ref{l_stqe} yields
		\begin{align*}
			|S_1|
			\lesssim_\delta
			T^\delta
			\sum_{K_i}\Pi(K_a,K_b,K_2,K_3)K_a^{(\theta-1)\frac{1}{2}}K_2^{2s}K_3^{\frac{1}{2}}
			\lesssim
			T^\delta
			\|u\|_{\spfc^s_T}^4
		\end{align*}
		for all sufficiently small $\theta$.
		
		Next, we consider the case $K_1\sim K_a\gg K_b$; the case $K_1\sim K_b\gg K_a$ follows by minor modifications. If we additionally assume $K_2\sim K_3$, then Lemma \ref{l_stqe} yields
		\begin{align*}
			|S_1|
			\lesssim_\delta
			T^\delta
			\sum_{K_i }\Pi(K_a,K_b,K_2,K_3)K_2^{(\theta-1+4s)\frac{1}{2}}K_b^{\frac{1}{2}}
			\lesssim
			T^\delta
			\|u\|_{\spfc^s_T}^4.
		\end{align*}
		Here, the last estimate is valid, since we can distribute the exponent $(\theta-1)\frac{1}{2}+\frac{1}{4}+2s$ of $K_2$ onto the three high-frequency terms $\|P_{K_a}u\|$, $\|P_{K_2}u\|$, and $\|P_{K_3}u\|$ - for this, we need to choose $\theta>1$ sufficiently small such that $(\theta-1)\frac{1}{2}+\frac{1}{4}<s_0$ holds.
		
		Now, we assume  $K_1\sim K_a\gg K_b$ as well as $K_2\gg K_3$. Then, the previous argument fails for any value of $s$ since the exponent $(\theta-1)\frac{1}{2}$ of $K_2$ is strictly positive. To overcome this problem, we transform half of the integral by the change of variables $(\xi_a,\xi_2)\mapsto (\xi_2,\xi_a)$. This leads to
		\begin{align*}
			|S_1|
			\lesssim
			\sum_{K_j}K_1^{2s}\sup_t\int_0^t\int_{\Gamma_4}m&(\xi_a,\xi_b,\xi_2,\xi_3)\chi_{K_a}(\xi_a)\chi_{K_b}(\xi_b)
			\\
			&\times\hat{u}(\xi_a)\hat{u}(\xi_b)\hat{u}(\xi_2)\hat{u}(\xi_3)d\gamma_4(\xi_a,\xi_b,\xi_2,\xi_3)dr,
		\end{align*}
		where the new multiplier $m$ is given by
		\begin{align}\label{d_m}
			m(\xi_a,\xi_b,\xi_2,\xi_3)
			:=
			\frac{\sigma_I}{2\Omega_3}(\xi_{2b},\xi_a,\xi_3)(i\xi_{2b})
			+
			\frac{\sigma_I}{2\Omega_3}(\xi_{ab},\xi_2,\xi_3)(i\xi_{ab}).
		\end{align}
		
		We claim that
		\begin{equation}\label{e_bnd_m}
			|m(\xi_a,\xi_b,\xi_2,\xi_3)|
			\lesssim
			K_a^{-1}\max\{K_3,K_b\}
		\end{equation}
		holds.
		To prove this, let us write
		\begin{align*}
			(-2i)m(\xi_a,\xi_b,\xi_2,\xi_3)
			&=
			\sigma_I(\xi_{2b},\xi_a,\xi_3)\Omega_3^{-1}(\xi_{2b},\xi_a,\xi_3)[\xi_{2b}+\xi_{ab}]
			&=: m_1
			\\&+
			\sigma_I(\xi_{2b},\xi_a,\xi_3)[\Omega_3^{-1}(\xi_{ab},\xi_2,\xi_3)-\Omega_3^{-1}(\xi_{2b},\xi_a,\xi_3)]\xi_{ab}
			&=: m_2
			\\&+
			[\sigma_I(\xi_{ab},\xi_2,\xi_3)-\sigma_I(\xi_{2b},\xi_a,\xi_3)]\Omega_3^{-1}(\xi_{ab},\xi_2,\xi_3)\xi_{ab}
			&=: m_3
		\end{align*}
		and bound each term on the right-hand side separately.
		Using \eqref{e_res_asy}, the uniform bound on $|\sigma_I|$, and $\xi_{2bab}=\xi_b-\xi_3+\xi_{ab23}=\xi_b-\xi_3$, it follows
		\begin{align*}
			|m_1|
			\lesssim
			K_3 K_a^{-1}K_3^{-1}|\xi_b-\xi_3|
			\lesssim
			K_a^{-1}\max\{K_3,K_b\}.
		\end{align*}
		The bound for $m_2$ follows by Lemma \ref{l_res_dif}. More precisely, we have
		\begin{align*}
			|m_2|
			\lesssim
			K_3 K_2^{-2}K_3^{-1}K_bK_a
			\sim
			K_2^{-1}K_b.
		\end{align*}
		It remains to bound $m_3$. The definitions of $\sigma_I$ and $\sigma_0$ yield
		\begin{align}\label{e_bnd_m3}
			\begin{split}
				&2\left[\sigma_I(\xi_{ab},\xi_2,\xi_3)
				-
				\sigma_I(\xi_{2b},\xi_a,\xi_3)\right]
				\\=\,
				&\sigma_0(\xi_{ab},\xi_2,\xi_3)
				+
				\sigma_0(\xi_2,\xi_{ab},\xi_3)
				-
				\sigma_0(\xi_{2b},\xi_a,\xi_3)
				-
				\sigma_0(\xi_a,\xi_{2b},\xi_3),
			\end{split}		
		\end{align}
		which we expand as follows
		\begin{align*}
			&-2i[\sigma_I(\xi_{ab},\xi_2,\xi_3)-\sigma_I(\xi_{2b},\xi_a,\xi_3)]
			\\&=
			\xi_a[\chi_{K_1}^2(\xi_a)-\chi_{K_1}^2(\xi_{a3})][\chi_{K_2}(\xi_{a3})-\chi_{K_2}(\xi_{ab3})]\chi_{K_3}(\xi_3)
			\\&-
			\xi_a[\chi_{K_1}^2(\xi_{a3})-\chi_{K_1}^2(\xi_{ab3})][\chi_{K_2}(\xi_{a})-\chi_{K_2}(\xi_{a3})]\chi_{K_3}(\xi_3)
			\\&+
			\xi_a[[\chi_{K_1}^2(\xi_a)-\chi_{K_1}^2(\xi_{a3})-\chi_{K_1}^2(\xi_{ab})+\chi_{K_1}^2(\xi_{ab3})]]\chi_{K_2}(\xi_{ab3})\chi_{K_3}(\xi_3)
			\\&-
			\xi_a\chi_{K_1}^2(\xi_{ab3})[[\chi_{K_2}(\xi_{a})-\chi_{K_2}(\xi_{a3})-\chi_{K_2}(\xi_{ab})+\chi_{K_2}(\xi_{ab3})]]\chi_{K_3}(\xi_3)
			\\&+
			\xi_b\chi_{K_1}^2(\xi_{ab})[\chi_{K_2}(\xi_{ab})-\chi_{K_2}(\xi_{ab3})]\chi_{K_3}(\xi_3)
			\\&-
			\xi_b[\chi_{K_1}^2(\xi_{ab})-\chi_{K_1}^2(\xi_{ab3})]\chi_{K_2}(\xi_{ab})\chi_{K_3}(\xi_3)
			\\&-
			\xi_3\chi_{K_1}^2(\xi_{a3})[\chi_{K_2}(\xi_{a})-\chi_{K_2}(\xi_{ab})]\chi_{K_3}(\xi_3)
			\\&-
			\xi_3[\chi_{K_1}^2(\xi_{a3})-\chi_{K_1}^2(\xi_{ab3})]\chi_{K_2}(\xi_{ab})\chi_{K_3}(\xi_3).
		\end{align*}
		Now, we estimate each term on the right-hand side above by using the fundamental theorem of calculus. Then, factors in square brackets have modulus bounded by either $K_a^{-1}K_b$ or $K_a^{-1}K_3$ and factors in double square brackets have modulus bounded by $K_a^{-2}K_3K_b$. We conclude
		\begin{align*}
			|m_3|
			\lesssim
			K_2^{-1}K_3K_bK_a^{-1}K_3^{-1}K_a
			\sim
			K_a^{-1}K_b.
		\end{align*}
		
		Having proved \eqref{e_bnd_m}, let us proceed with bounding $|S_1|$ in the case $K_1\sim K_a\gg K_b$ and  $K_2\gg K_3$. Using $\eqref{e_bnd_m}$ and Lemma \ref{l_stqe}, we conclude
		\begin{align*}
			|S_1|
			\!\lesssim_\delta
			T^\delta
			\sum_{K_i}\Pi(K_a,\!K_b,\!K_2,\!K_3)K_2^{2s-1+(\theta-1)\frac{1}{2}}\!\max\{K_3,\!K_b\}^\frac{1}{2}(K_3K_b)^\frac{1}{2}
			\!\lesssim
			T^\delta\|u\|_{\spfc^s_T}^4.
		\end{align*}
		Thus, the estimate for $S_1$ is complete.
		
		Now, we bound $S_3$. Again, we write $\widehat{u^2_x}=(\hat{u}\ast\hat{u})_x$ and localize the new variables $\xi_a$ and $\xi_b$ to dyadic frequency ranges of size $K_a$ and $K_b$ yielding
		\begin{align*}
			|S_3|
			\lesssim
			\sum_{K_j}K_1^{2s}\sup_t\int_0^t\int_{\Gamma_4}\frac{\sigma_I}{\Omega_3}&(\xi_1,\xi_2,\xi_{ab})(i\xi_{ab})\chi_{K_a}(\xi_a)\chi_{K_b}(\xi_b)
			\\
			&\times\hat{u}(\xi_1)\hat{u}(\xi_2)\hat{u}(\xi_a)\hat{u}(\xi_b)d\gamma_4(\xi_1,\xi_2,\xi_a,\xi_b)dr.
		\end{align*}
		Applying Lemma \ref{l_stqe}  to the right-hand side above leads to
		\begin{equation*}
			|S_3|\!
			\lesssim_\delta
			T^\delta\sum_{K_i}\Pi(K_1,\!K_2,\!K_a,\!K_b)\times
			\begin{cases}
				K_a^{(\theta-1)\frac{1}{2}}K_1^{2s-\frac{1}{2}}K_3
				&\text{if } 
				K_a\!\sim\!K_b\!\gg\!K_1,
				\\
				\textstyle K_1^{2s-1+(\theta-1)\frac{1}{2}}K_a^{\frac{1}{2}}K_3
				&\text{if }
				K_1\!\gtrsim\!K_a\!\sim\!K_b\!\gtrsim\!K_3,
				\\
				\textstyle K_1^{2s-1+(\theta-1)\frac{1}{2}}K_3K_b^{\frac{1}{2}}
				&\text{if } 
				K_a\!\sim\!K_3\!\gg\!K_b,
				\\
				\textstyle K_1^{2s-1+(\theta-1)\frac{1}{2}}K_3K_a^{\frac{1}{2}}
				&\text{if } 
				K_b\!\sim\!K_3\!\gg\!K_a.
			\end{cases}
		\end{equation*}
		Each of the four estimates above is sufficient to obtain the desired upper bound in the corresponding case.
	\end{proof}
	
	Lastly, we bound $E_{I,5}$.
		
	\begin{lem}\label{l_ee_qua_bnd}
		Let $s\geq s_0>\frac{1}{4}$. Then, for all sufficiently small $\theta=\theta(s_0)>1$ and $\delta\in(0,\frac{1}{2})$, we have
		\begin{align*}
			|E_{I,5}|
			\lesssim_\delta
			T^\delta\|u\|_{\spfc_T^s}^2\|u\|_{\spfc_T^{s_0}}\|b\|_{\spbc^s_T}.
		\end{align*}
	\end{lem}
	\begin{proof}
		We write $E_{I,5}=S_1+S_2+S_3$ for
		\begin{align*}
			S_1
			&:=
			\sum_{C_H}K_1^{2s}\sup_t\int_0^t\int_{\Gamma_3}\left[\frac{\sigma_I}{\Omega_3}(\widehat{(ub)_x}\hat{u}\hat{u})\right](\xi)d\gamma_3(\xi)dr,
			\\
			S_2
			&:=
			\sum_{C_H}K_1^{2s}\sup_t\int_0^t\int_{\Gamma_3}\left[\frac{\sigma_I}{\Omega_3}(\hat{u}\widehat{(ub)_x}\hat{u})\right](\xi)d\gamma_3(\xi)dr,
			\\
			S_3
			&:=
			\sum_{C_H}K_1^{2s}\sup_t\int_0^t\int_{\Gamma_3}\left[\frac{\sigma_I}{\Omega_3}(\hat{u}\hat{u}\widehat{(ub)_x})\right](\xi)d\gamma_3(\xi)dr.
		\end{align*}
		As in the previous lemma, we can assume $K_1\sim K_2\gtrsim K_3$. Again, we observe that $S_1=S_2$ holds.
		Define
		\begin{align*}
			\Pi(K_i,K_j,K_g,K_h)
			:=
			\|P_{K_i}u\|_{\spfb^{K_i}_T}
			\|P_{K_j}u\|_{\spfb^{K_j}_T}
			\|P_{K_g}u\|_{\spfb^{K_g}_T}
			\|P_{K_h}b\|_{L^\infty}.
		\end{align*}
		
		Let us begin by estimating $S_1$. We write $\widehat{(ub)_x}=(\hat{u}\ast\hat{b})_x$ and localize the new variables $\xi_a$ and $\xi_b$ to dyadic frequency ranges of size $K_a$ and $K_b$. This leads to
		\begin{align*}
			|S_1|
			\lesssim
			\sum_{K_i}K_1^{2s}\sup_t\int_0^t\int_{\Gamma_4}\frac{\sigma_I}{\Omega_3}&(\xi_{ab},\xi_2,\xi_3)(i\xi_{ab})\chi_{K_a}(\xi_a)\chi_{K_b}(\xi_b)
			\\
			&\times\hat{u}(\xi_a)\hat{b}(\xi_b)\hat{u}(\xi_2)\hat{u}(\xi_3)d\gamma_4(\xi_a,\xi_b,\xi_2,\xi_3)dr.
		\end{align*}
		
		First, consider the case that $\xi_b$ is not the smallest frequency. Then, a direct application of the first estimate in Lemma \ref{l_stqe_bnd} suffices. Indeed, we have
		\begin{align*}
			|S_1|
			\lesssim_\delta
			T^\delta\sum_{K_i}
			\Pi(K_2,K_3,K_a,K_b)
			\times
			\begin{cases}
				K_1^{2s}K_3^{\frac{1}{2}}
				&\text{if }
				K_a\!\sim\!K_b\!\gg\!K_1,
				\\
				K_1^{2s}K_3^{\frac{1}{2}}
				&\text{if }
				K_a\!\sim\!K_1\!\gtrsim\!K_b\!\gtrsim\!K_3,
				\\
				K_1^{2s}K_3^{\frac{1}{2}}
				&\text{if }
				K_b\!\sim\!K_1\!\gtrsim\!K_a\!\gtrsim\!K_3,
				\\
				K_1^{2s}K_a^{\frac{1}{2}}
				&\text{if }
				K_b\!\sim\!K_1\!\gtrsim\!K_3\!\gtrsim\!K_a.
			\end{cases}
		\end{align*}
		All estimates given above lead to the desired bound.
				
		Next, we consider the case $K_a\sim K_1\gtrsim K_3\gtrsim K_b$. A direct application of Lemma \ref{l_stqe_bnd} is only sufficient for $s_0>\frac{1}{2}$. To obtain the full range $s_0>\frac{1}{4}$, we differentiate between two subcases. First, we assume $K_3\sim K_1$. Then, the second estimate in Lemma \ref{l_stqe_bnd} yields
		\begin{align*}
			|S_1|
			\lesssim_\delta
			T^\delta\sum_{K_i}\Pi(K_2,K_3,K_a,K_b)
			K_1^{2s+\frac{\theta}{4}}
			\lesssim
			T^\delta\|u\|_{\spfc^s_T}^2\|u\|_{\spfc^{s_0}_T}\|b\|_{\spbc^0_T}.
		\end{align*}
		In the case $K_a\sim K_1\gg K_3\gtrsim K_b$, we can repeat large parts of the proof of the previous lemma including the introduction of the multiplier $m$ in \eqref{d_m} leading to
		\begin{align*}
			|S_1|
			\lesssim
			\sum_{K_i}K_1^{2s}\sup_t\int_0^t\int_{\Gamma_4}m&(\xi_a,\xi_b,\xi_2,\xi_3)\chi_{K_a}(\xi_a)\chi_{K_b}(\xi_b)
			\\
			&\times\hat{u}(\xi_a)\hat{b}(\xi_b)\hat{u}(\xi_2)\hat{u}(\xi_3)d\gamma_4(\xi_a,\xi_b,\xi_2,\xi_3)dr.
		\end{align*}
		Using the bound on $m$ from \eqref{e_bnd_m} and applying the second estimate in Lemma \ref{l_stqe_bnd}, we obtain
		\begin{align*}
			|S_1|
			\lesssim_\delta
			T^\delta\sum_{K_i}\Pi(K_2,K_3,K_a,K_b)
			K_1^{2s+\frac{\theta}{4}-1}K_3
			\lesssim
			T^\delta\|u\|_{\spfc^s_T}^2\|u\|_{\spfc^{s_0}_T}\|b\|_{\spbc^0_T}.
		\end{align*}
		This completes the estimate on $S_1$.
		
		It remains to bound $S_3$. Proceeding as for $S_1$, we obtain
		\begin{align*}
			|S_3|
			\lesssim
			\sum_{K_i}K_1^{2s}\sup_t\int_0^t\int_{\Gamma_4}\frac{\sigma_I}{\Omega_3}&(\xi_1,\xi_2,\xi_{ab})(i\xi_{ab})\chi_{K_a}(\xi_a)\chi_{K_b}(\xi_b)
			\\
			&\times\hat{u}(\xi_1)\hat{u}(\xi_2)\hat{u}(\xi_a)\hat{b}(\xi_b)d\gamma_4(\xi_1,\xi_2,\xi_a,\xi_b)dr.
		\end{align*}
		By applying the first, respectively the second estimate in Lemma \ref{l_stqe_bnd}, we get
		\begin{align*}
			|S_3|
			\lesssim_\delta
			T^\delta\sum_{K_i}
			\Pi(K_1,\!K_2,\!K_a,\!K_b)
			\times
			\begin{cases}
				K_1^{2s-\frac{1}{2}}K_3
				&\text{if }
				K_a\!\sim\!K_b\!\gtrsim\!K_1,
				\\
				K_1^{2s-1}K_a^\frac{1}{2}K_3
				&\text{if }
				K_1\!\gg\!K_a\!\sim\!K_b\!\gtrsim\!K_3,
				\\
				K_1^{2s-1+\frac{\theta}{4}}K_3
				&\text{if }
				K_3\!\sim\!K_a\!\gg\!K_b,
				\\
				K_1^{2s-1}K_3K_a^{\frac{1}{2}}
				&\text{if }
				K_3\!\sim\!K_b\!\gg\!K_a.
			\end{cases}
		\end{align*}
		Again, the estimates above lead to the desired bound.
	\end{proof}
	Let us show how the results of this section prove Lemma \ref{l_ee_i}. In the beginning of this section, we deduced that
	\begin{align*}
		\|u\|_{\spec^s_T}^2
		\lesssim
		\|u_0\|_{H^s}^2
		+
		I_1
		+
		I_2
		+
		I_3
		\end{align*}
	holds. Now, $I_2$ respectively $I_3$ are estimated in Lemma \ref{l_ee_trv} respectively Lemma \ref{l_ee_trv_uub} and the bound for $I_1$ is a direct consequence of Lemmata \ref{l_ee_tri}, \ref{l_ee_qua}, and \ref{l_ee_qua_bnd}.
	\subsection{Proof of the second energy estimate}\label{ss_ee_ii}
	In this section we prove Lemma \ref{l_ee_ii}. Fix $s>\frac{1}{2}$, $z=-\frac{1}{2}$, $T\in(0,1)$, and let $\theta(s)>1$ be small enough to satisfy all restrictions appearing in this section. Choose $b$ and $f$ satisfying \eqref{as_b_f} and let and $u_1$ and $u_2$ be classical solutions to \eqref{eq_bo_split} with initial datum $u_{1,0}$ and $u_{2,0}$. In the following, we write $v=u_1-u_2$, $v_0=u_{1,0}-u_{2,0}$, and $w=u_1+u_2$.
	
	Since $u_1$ and $u_2$ are solutions to \eqref{eq_bo_split}, the difference $v$ is a solution to
	\begin{align*}
		\begin{cases}
			v_t
			+
			\mathcal{H}v_{xx}
			+
			(vw)_x
			+
			(vb)_x
			&=
			0,
			\\
			v(0,\cdot)
			&=
			v_0.
		\end{cases}
	\end{align*}
	For $K_1\in\mathbb{D}$ it follows
	\begin{align*}
		\frac{d}{dr}\|P_{K_1}v(r,x)\|_{L^2_x}^2
		&=
		2\int_\mathbb{R} -P_{K_1}v\left[P_{K_1}\mathcal{H}v_{xx}+P_{K_1}(vw)_x+P_{K_1}(vb)_x\right] dx
		\\&=
		2\int_\mathbb{R} P_{K_1}^2v_x(vw)dx
		+
		2\int_\mathbb{R} P_{K_1}^2v_xvb dx.
	\end{align*}
	After multiplying this equation with $K_1^{2z}$, integrating with respect to $r$ over $[0,t]$, taking the supremum over $t\in[0,T]$, as well as after summation over $K_1\in\mathbb{D}$, we obtain
	\begin{align*}
		\|v\|_{\spec^z_T}^2
		-
		\|v_0\|_{H^z}^2
		&\leq
		2\sum_{K_1}K_1^{2z}\sup_t\int_0^t\int_\mathbb{R}[P_{K_1}^2v_xvw](r,x)dxdr
		&=:I\!I_1\nonumber
		\\&+
		2\sum_{K_1}K_1^{2z}\sup_t\int_0^t\int_\mathbb{R}[P_{K_1}^2v_xvb](r,x)dxdr
		&=:I\!I_2.
	\end{align*}
	
	As in the previous section, the bound for $I\!I_2$ can be derived easily.
	
	\begin{lem}\label{l_ee_ii_trv}
		Let $\varepsilon>0$. Then, we have
		\begin{align*}
			|I\!I_2|
			\lesssim_\varepsilon
			T\|v\|_{\spec^z_T}^2\|b\|_{\spbc^{1+\varepsilon}_T}.
		\end{align*}
	\end{lem}
	\begin{proof}
		After localizing the second and third factor using Littlewood-Paley projectors, it remains to estimate
		\begin{align*}
			I\!I_2
			=
			\sum_{K_i}K_1^{2z}\sup_t\int_0^t\int_{\Gamma_3}(-i\xi_1)\chi_{K_1}^2(\xi_1)\chi_{K_2}(\xi_2)\chi_{K_3}(\xi_3)\hat{v}(\xi_1)\hat{v}(\xi_2)\hat{b}(\xi_3)d\gamma_3(\xi)dr.
		\end{align*}
		In the cases $K_3^*=K_1$ respectively $K_3^*=K_2$, we obtain
		\begin{align*}
			|I\!I_2|
			\lesssim
			T\sum_{K_i}K_1^{2z+1}\|v_{K_1}\|_{L^\infty L^2}\|v_{K_2}\|_{L^\infty L^2}\|b_{K_3}\|_{L^\infty L^\infty}
			\lesssim
			T\|v\|_{\spec^z_T}^2\|b\|_{\spbc^{1+\varepsilon}_T},
		\end{align*}
		where we used that $z\geq -1$ respectively $z\leq 0$ holds. For the case $K_3^*=K_3$, we symmetrize the symbol as in Lemma \ref{l_ee_trv_uub} and get
		\begin{align*}
			|I\!I_2|
			\lesssim
			T\sum_{K_i}K_1^{2z}K_3^1
			\|v_{K_1}\|_{L^\infty L^2}\|v_{K_2}\|_{L^\infty L^2}\|b_{K_3}\|_{L^\infty L^2}
			\lesssim
			T\|v\|_{\spec^z_T}^2\|b\|_{\spbc^{1+\varepsilon}_T}.
		\end{align*}
		This proves the claimed estimate for $I\!I_2$.
	\end{proof}
	
	In the remaining part of this section, we derive a bound for $I\!I_1$. The procedure is similar to the one for $I_1$ in the previous section, which is why we omit a few details. However, let us emphasize the main difference: In the setting of this section the integrand $v_xvw$ is less symmetric than the integrand $u_xuu$ in the previous section and thus the symbol cannot by symmetrized as before, see the definition of $\sigma_{I\!I}$ below. This deficiency leads to the condition $s>\frac{1}{2}$ in contrast to $s>\frac{1}{4}$ as assumed in the previous section.
	
	Define
	\begin{align*}
		\sigma_0(\xi)
		:=
		-i\xi_1\chi_{K_1}^2(\xi_1)\chi_{K_2}(\xi_2)\chi_{K_3}(\xi_3)
	\end{align*}
	and
	\begin{align*}
		\sigma_{I\!I}(\xi_1,\xi_2,\xi_3)
		:=
		\begin{cases}
			\sigma_0(\xi_1,\xi_2,\xi_3)
			&\text{if }K_3^*=K_1,
			\\
			\sigma_0(\xi_1,\xi_2,\xi_3)
			&\text{if }K_3^*=K_2,
			\\
			\frac{1}{2}[\sigma_0(\xi_1,\xi_2,\xi_3)+\sigma_0(\xi_2,\xi_1,\xi_3)]
			&\text{if }K_3^*=K_3.
		\end{cases}
	\end{align*}
	Then, localizing the second and third factor in the integrand of $I\!I_1$, we conclude
	\begin{align}\label{eq_ee_ii_b}
		I\!I_1
		=
		&\sum_{K_1}K_1^{2z}\sup_t\int_0^t\int_\mathbb{R}[P_{K_1}^2v_xvw](r,x)dxdr\nonumber
		\\ \leq
		&\sum_{K_i}K_1^{2z}\sup_t\int_0^t\int_{\Gamma_3}\sigma_0(\xi)\hat{v}(r,\xi_1)\hat{v}(r,\xi_2)\hat{w}(r,\xi_3)d\gamma_3(\xi)dr,\nonumber
		\\ =
		&\sum_{K_i}K_1^{2z}\sup_t\int_0^t\int_{\Gamma_3}\sigma_{I\!I}(\xi)\hat{v}(r,\xi_1)\hat{v}(r,\xi_2)\hat{w}(r,\xi_3)d\gamma_3(\xi)dr.
	\end{align}
	Note that we have $|\sigma(\xi)|\lesssim K_3^*$ if $K_3^*\neq K_2$, but only $|\sigma(\xi)|\lesssim K_1^*$ if $K_3^*=K_2$.
	
	Next, we integrate by parts in time. For this, let us fix a parameter $N\in\mathbb{D}$ and define
	\begin{align*}
		\sum_{C_L}
		\!:=\!
		\sum_{K_i}1_{\{K_3^*=1\}},
		\quad
		\sum_{C_M}
		\!:=\!
		\sum_{K_i}1_{\{K_3^*>1, K_1^*\leq N\}},
		\quad
		\sum_{C_H}
		\!:=\!
		\sum_{K_i}1_{\{K_3^*>1, K_1^*> N\}}.
	\end{align*}
	Additionally, given arbitrary $f_1$, $f_2$, and $f_3$, we introduce the abbreviation $\left[f_1,f_2,f_3\right](\xi)$ for $f_1(\xi_1)f_2(\xi_2)f_3(\xi_3)$. Using that $u_1$ and $u_2$ are solutions to \eqref{eq_bo_split}, a straightforward calculation shows:
	\begin{align}\label{eq_ee_ii_ibp}
		\begin{split}
			-[\hat{v},\hat{v},\hat{w}]_t
			&=
			\Omega_3[\hat{v},\hat{v},\hat{w}]
			\\&+
			[\widehat{(vw)_x},\hat{v},\hat{w}]
			+
			[\hat{v},\widehat{(vw)_x},\hat{w}]
			+
	 		[\hat{v},\hat{v},\widehat{(u_1u_1)_x}]
	 		+
	 		[\hat{v},\hat{v},\widehat{(u_2u_2)_x}]
			\\&+
			[\widehat{(vb)_x},\hat{v},\hat{w}]
			+
			[\hat{v},\widehat{(vb)_x},\hat{w}]
			+
			[\hat{v},\hat{v},\widehat{(wb)_x}]
			+
			2[\hat{v},\hat{v},\hat{f}]
		\end{split}
	\end{align}
	
	Now, split the sum in \eqref{eq_ee_ii_b} into the three sums defined above. Then, for each summand contributing to $C_H$, we use \eqref{eq_ee_ii_ibp} to rewrite the integrand $\sigma_{I\!I}[\hat{v},\hat{v},\hat{w}]$. We obtain:
	\begin{align*}
		I\!I_1
		\leq
		&\sum_{C_L}K_1^{2z}\sup_t\int_0^t\int_{\Gamma_3}(\sigma_{I\!I}[\hat{v},\hat{v},\hat{w}])(\xi)d\gamma_3(\xi)dr
		&&:=E_{I\!I,1}
		\\ +
		&\sum_{C_M}K_1^{2z}\sup_t\int_0^t\int_{\Gamma_3}(\sigma_{I\!I}[\hat{v},\hat{v},\hat{w}])(\xi)d\gamma_3(\xi)dr
		&&:=E_{I\!I,2}
		\\ +
		&\sum_{C_H}K_1^{2z}\sup_t\left[\int_{\Gamma_3}(\sigma_{I\!I}\Omega_3^{-1}[\hat{v},\hat{v},\hat{w}])(\xi)d\gamma_3(\xi)\right]_{r=0}^t
		&&:=E_{I\!I,3}
		\\ +
		&\sum_{C_H}K_1^{2z}\sup_t\int_0^t\int_{\Gamma_3}(\sigma_{I\!I}\Omega_3^{-1}[\widehat{(vw)_x},\hat{v},\hat{w}])(\xi)d\gamma_3(\xi)dr
		&&:=E_{I\!I,4}
		\\ +
		&\sum_{C_H}K_1^{2z}\sup_t\int_0^t\int_{\Gamma_3}(\sigma_{I\!I}\Omega_3^{-1}[\hat{v},\widehat{(vw)_x},\hat{w}])(\xi)d\gamma_3(\xi)dr
		&&:=E_{I\!I,5}
		\\ +
		&\sum_{C_H}K_1^{2z}\sup_t\int_0^t\int_{\Gamma_3}(\sigma_{I\!I}\Omega_3^{-1}[\hat{v},\hat{v},\widehat{(u_1u_1)_x}+\widehat{(u_2u_2)_x}])(\xi)d\gamma_3(\xi)dr
		&&:=E_{I\!I,6}
		\\ +
		&\sum_{C_H}K_1^{2z}\sup_t\int_0^t\int_{\Gamma_3}(\sigma_{I\!I}\Omega_3^{-1}[\widehat{(vb)_x},\hat{v},\hat{w}])(\xi)d\gamma_3(\xi)dr
		&&:=E_{I\!I,7}
		\\ +
		&\sum_{C_H}K_1^{2z}\sup_t\int_0^t\int_{\Gamma_3}(\sigma_{I\!I}\Omega_3^{-1}[\hat{v},\widehat{(vb)_x},\hat{w}])(\xi)d\gamma_3(\xi)dr
		&&:=E_{I\!I,8}
		\\ +
		&\sum_{C_H}K_1^{2z}\sup_t\int_0^t\int_{\Gamma_3}(\sigma_{I\!I}\Omega_3^{-1}[\hat{v},\hat{v},\widehat{(wb)_x}])(\xi)d\gamma_3(\xi)dr
		&&:=E_{I\!I,9}
		\\ +
		&\sum_{C_H}K_1^{2z}\sup_t\int_0^t\int_{\Gamma_3}(2\sigma_{I\!I}\Omega_3^{-1}[\hat{v},\hat{v},\hat{f}])(\xi)d\gamma_3(\xi)dr
		&&:=E_{I\!I,10}
	\end{align*}
	
	We continue by estimating the low-frequency terms $E_{I\!I,1}$ and $E_{I\!I,2}$, the boundary term $E_{I\!I,3}$, and the term involving the forcing $E_{I\!I,10}$. Minor adjustments to the proof of Lemma \ref{l_ee_tri} lead to:
	\begin{lem}\label{l_ee_ii_tri}
		Let $\zeta\in(0,\frac{1}{2})$, $z=-\frac{1}{2}$, and $s>\frac{1}{2}$. Then, we have
		\begin{align*}
			|E_{I\!I,1}|
			&\lesssim
			T\|v\|_{\spec^z_T}^2\|w\|_{\spec^s_T},
			\\
			|E_{I\!I,2}|
			&\lesssim
			TN^{2}\|v\|_{\spec^z_T}^2\|w\|_{L^\infty H^s},
			\\
			|E_{I\!I,3}|
			&\lesssim_\zeta
			N^{-\zeta}\|v\|_{L^\infty H^z}^2\|w\|_{L^\infty H^s},
			\\
			|E_{I\!I,10}|
			&\lesssim
			T\|v\|_{L^\infty H^z}^2\|f\|_{L^\infty H^s}.
		\end{align*}
	\end{lem}
	In the following three lemmata, we provide estimates for $E_{I\!I,4}$, $E_{I\!I,5}$, and $E_{I\!I,6}$.
	\begin{lem}\label{l_ee_ii_qua_i}
		Let $s>\frac{1}{2}$. Then, for all sufficiently small $\theta=\theta(s)>1$ and $\delta\in(0,\frac{1}{2})$, we have
		\begin{align*}
			|E_{I\!I,4}|
			\lesssim_\delta
			T^\delta\|v\|_{\spfc^z_T}^2\|w\|_{\spfc^s_T}^2.
		\end{align*}
	\end{lem}
	\begin{proof}
		We split $E_{I\!I,4}=S_1+S_2+S_3$ with
		\begin{align*}
			S_j
			:=
			\begin{cases}
				E_{I\!I,4}|_{K_2\sim K_3\gg K_1}
				&\text{if }
				j=1,
				\\
				E_{I\!I,4}|_{K_1\sim K_3\gg K_2}
				&\text{if }
				j=2,
				\\
				E_{I\!I,4}|_{K_1\sim K_2\gtrsim K_3}
				&\text{if }
				j=3.
			\end{cases}
		\end{align*}
		Using $\widehat{(vw)}=(\hat{v}\ast\hat{w})$ and localizing these two factors separately, we obtain
		\begin{align*}
			|S_j|
			\lesssim
			\sum_{K_j}K_1^{2z}\sup_t\int_0^t\int_{\Gamma_4}\frac{\sigma_{I\!I}}{\Omega_3}&(\xi_{ab},\xi_2,\xi_3)(i\xi_{ab})\chi_{K_a}(\xi_a)\chi_{K_b}(\xi_b)
			\\
			&\times\hat{v}(\xi_a)\hat{w}(\xi_b)\hat{v}(\xi_2)\hat{w}(\xi_3)d\gamma_4(\xi_a,\xi_b,\xi_2,\xi_3)dr
		\end{align*}
		for $j\in[3]$. Define
		\begin{align*}
			\Pi(K_a,\!K_b,\!K_2,\!K_3)
			:=
			\|P_{K_a}v\|_{\spfb^{K_a}_T}
			\|P_{K_b}w\|_{\spfb^{K_b}_T}
			\|P_{K_2}v\|_{\spfb^{K_2}_T}
			\|P_{K_3}w\|_{\spfb^{K_3}_T}.
		\end{align*}
		
		We begin by estimating $S_1$. As the derivative is on a low-frequency, we have $|\frac{\sigma_{I\!I}}{\Omega_3}(\xi_{ab},\xi_2,\xi_3)(i\xi_{ab})|\lesssim K_1K_2^{-1}$ and it suffices to apply Lemma \ref{l_stqe} leading to
		\begin{align*}
			|S_1|
			\lesssim_\delta
			T^\delta
			\sum_{K_j}\Pi(K_a,\!K_b,\!K_2,\!K_3)
			\times
			\begin{cases}
				1
				&\text{if }
				K_a\!\sim\!K_b\!\gg\!K_2,
				\\
				K_2^{-1}K_a^\frac{1}{2}K_b^\frac{1}{2}&\text{if }
				K_2\!\gtrsim\!K_a\!\sim\!K_b\!\gg\!K_1,
				\\
				K_2^{-1}K_a^\frac{1}{2}K_b^\frac{1}{2}
				&\text{if }
				K_a\!\sim\!K_1\!\gg\!K_b,
				\\
				K_2^{-1}K_b^\frac{1}{2}K_a^\frac{1}{2}
				&\text{if }
				K_b\!\sim\!K_1\!\gg\!K_a.
			\end{cases}
		\end{align*}
		
		Next, we consider $S_2$. Then, the derivative is on a high-frequency factor, but the factor $v_{K_2}$ is always localized to low frequencies. Hence, $|\frac{\sigma_{I\!I}}{\Omega_3}(\xi_{ab},\xi_2,\xi_3)(i\xi_{ab})|\lesssim K_1K_2^{-1}$ and Lemma \ref{l_stqe} implies
		\begin{align*}
			|S_2|
			\lesssim_\delta
			T^\delta
			\sum_{K_j}\Pi(K_a,\!K_b,\!K_2,\!K_3)
			\times
			\begin{cases}
				K_3^\frac{1}{2}K_2^{-\frac{1}{2}}
				&\text{if }
				K_a\!\sim\!K_b\!\gg\!K_1,
				\\
				K_b^\frac{1}{2}K_2^{-\frac{1}{2}}
				&\text{if }
				K_a\!\sim\!K_1\!\gg\!K_b,
				\\
				K_a^\frac{1}{2}K_2^{-\frac{1}{2}}
				&\text{if }
				K_b\!\sim\!K_1\!\gg\!K_a.
			\end{cases}
		\end{align*}
		
		Lastly, we establish a bound for $S_3$. Again, the derivative is on a high-frequency factor and additionally, both $v_{K_a}$ and $v_{K_2}$ might be localized to high frequencies. Nonetheless, we obtain $|\frac{\sigma_{I\!I}}{\Omega_3}(\xi_{ab},\xi_2,\xi_3)(i\xi_{ab})|\lesssim 1$ and Lemma \ref{l_stqe} yields
		\begin{align*}
			|S_3|
			\lesssim_\delta
			T^\delta
			\sum_{K_j}\Pi(K_a,\!K_b,\!K_2,\!K_3)
			\times
			\begin{cases}
				K_2^{-\frac{1}{2}}K_3^\frac{1}{2}
				&\text{if }
				K_a\!\sim\!K_b\!\gtrsim\!K_1,
				\\
				K_2^{-1}K_b^\frac{1}{2}K_3^\frac{1}{2}
				&\text{if }
				K_a\!\sim\!K_1\!\gg\!K_b,
				\\
				K_2^{-1}K_a^\frac{1}{2}K_3^\frac{1}{2}
				&\text{if }
				K_b\!\sim\!K_1\!\gg\!K_a.
			\end{cases}
		\end{align*}
	\end{proof}
	\begin{lem}\label{l_ee_ii_qua_ii}
		Let $s>\frac{1}{2}$. Then, for all sufficiently small $\theta=\theta(s)>1$ and $\delta\in(0,\frac{1}{2})$, we have
		\begin{align*}
			|E_{I\!I,5}|
			\lesssim_\delta
			T^\delta\|v\|_{\spfc^z_T}^2\|w\|_{\spfc^s_T}^2.
		\end{align*}
	\end{lem}
	\begin{proof}
		We split $E_{I\!I,5}=S_1+S_2+S_3$ with
		\begin{align*}
			S_j
			:=
			\begin{cases}
				E_{I\!I,5}|_{K_2\sim K_3\gg K_1}
				&\text{if }
				j=1,
				\\
				E_{I\!I,5}|_{K_1\sim K_3\gg K_2}
				&\text{if }
				j=2,
				\\
				E_{I\!I,5}|_{K_1\sim K_2\gtrsim K_3}
				&\text{if }
				j=3.
			\end{cases}
		\end{align*}
		Using $\widehat{(vw)}=(\hat{v}\ast\hat{w})$ and localizing these two factors separately, we obtain
		\begin{align*}
			|S_j|
			\lesssim
			\sum_{K_j}K_1^{2z}\sup_t\int_0^t\int_{\Gamma_4}\frac{\sigma_{I\!I}}{\Omega_3}&(\xi_1,\xi_{ab},\xi_3)(i\xi_{ab})\chi_{K_a}(\xi_a)\chi_{K_b}(\xi_b)
			\\
			&\times\hat{v}(\xi_1)\hat{v}(\xi_a)\hat{w}(\xi_b)\hat{w}(\xi_3)d\gamma_4(\xi_1,\xi_a,\xi_b,\xi_3)dr
		\end{align*}
		for $j\in[3]$. Define
		\begin{align*}
			\Pi(K_1,\!K_a,\!K_b,\!K_3)
			:=
			\|P_{K_1}v\|_{\spfb^{K_1}_T}
			\|P_{K_a}v\|_{\spfb^{K_a}_T}
			\|P_{K_b}w\|_{\spfb^{K_b}_T}
			\|P_{K_3}w\|_{\spfb^{K_3}_T}.
		\end{align*}
		
		We begin by estimating $S_1$. Here, we have $|\frac{\sigma_{I\!I}}{\Omega_3}(\xi_{ab},\xi_2,\xi_3)(i\xi_{ab})|\lesssim 1$. Hence, it suffices to apply Lemma \ref{l_stqe} leading to
		\begin{align*}
			|S_1|
			\lesssim_\delta
			T^\delta
			\sum_{K_j}\Pi(K_1,\!K_a,\!K_b,\!K_3)
			\times
			\begin{cases}
				K_3^\frac{1}{2}K_1^{-\frac{1}{2}}
				&\text{if }
				K_a\!\sim\!K_b\!\gg\!K_2,
				\\
				K_1^{-\frac{1}{2}}K_b^\frac{1}{2}
				&\text{if }
				K_a\sim K_2\gg K_b,
				\\
				K_1^{-\frac{1}{2}}K_a^\frac{1}{2}
				&\text{if }
				K_b\sim K_2\gg K_a.
			\end{cases}
		\end{align*}
		
		Next, we consider $S_2$. Again, we have $|\frac{\sigma_{I\!I}}{\Omega_3}(\xi_{ab},\xi_2,\xi_3)(i\xi_{ab})|\lesssim 1$. Together with Lemma \ref{l_stqe}, we get
		\begin{align*}
			|S_2|
			\lesssim_\delta
			T^\delta
			\sum_{K_j}\Pi(K_1,\!K_a,\!K_b,\!K_3)
			\times
			\begin{cases}
				K_1^{-\frac{1}{2}}K_3^\frac{1}{2}
				&\text{if }
				K_a\!\sim\!K_b\!\gg\!K_1,
				\\
				K_1^{-1}K_a^\frac{1}{2}K_b^\frac{1}{2}
				&\text{if }
				K_1\gtrsim K_a\!\sim\!K_b\!\gtrsim\!K_2,
				\\
				K_1^{-1}K_a^\frac{1}{2}K_b^\frac{1}{2}
				&\text{if }
				K_a\!\sim\!K_2\!\gg\!K_b,
				\\
				K_1^{-1}K_a^\frac{1}{2}K_b^\frac{1}{2}
				&\text{if }
				K_b\!\sim\!K_2\!\gg\!K_a.
			\end{cases}
		\end{align*}
		
		Lastly, we establish a bound for $S_3$. Again, we have $|\frac{\sigma_{I\!I}}{\Omega_3}(\xi_{ab},\xi_2,\xi_3)(i\xi_{ab})|\lesssim 1$. An application of Lemma \ref{l_stqe} yields
		\begin{align*}
			|S_3|
			\lesssim_\delta
			T^\delta
			\sum_{K_j}\Pi(K_1,\!K_a,\!K_b,\!K_3)
			\times
			\begin{cases}
				K_1^{-\frac{1}{2}}K_3^\frac{1}{2}
				&\text{if }
				K_a\!\sim\!K_b\!\gtrsim\!K_a,
				\\
				K_a^{-1}K_b^\frac{1}{2}K_3^\frac{1}{2}
				&\text{if }
				K_a\!\sim\!K_2\!\gg\!K_b,
				\\
				K_1^{-1}K_a^\frac{1}{2}K_3^\frac{1}{2}
				&\text{if }
				K_b\!\sim\!K_2\!\gg\!K_a.
			\end{cases}
		\end{align*}
	\end{proof}
	
	\begin{lem}\label{l_ee_ii_qua_iii}
		Let $s>\frac{1}{2}$. Then, for all sufficiently small $\theta=\theta(s)>1$ and $\delta\in(0,\frac{1}{2})$, we have
		\begin{align*}
			|E_{I\!I,6}|
			\lesssim_\delta
			T^\delta\|v\|_{\spfc^z_T}^2(\|u_1\|_{\spfc^s_T}^2+\|u_2\|_{\spfc^s_T}^2).
		\end{align*}
	\end{lem}
	\begin{proof}
		It suffices to consider the case $u_2=0$. We write $E_{I\!I,6}=S_1+S_2+S_3$ with
		\begin{align*}
			S_j
			:=
			\begin{cases}
				E_{I\!I,6}|_{K_2\sim K_3\gg K_1}
				&\text{if }
				j=1,
				\\
				E_{I\!I,6}|_{K_1\sim K_3\gg K_2}
				&\text{if }
				j=2,
				\\
				E_{I\!I,6}|_{K_1\sim K_2\gtrsim K_3}
				&\text{if }
				j=3.
			\end{cases}
		\end{align*}
		Using $\widehat{(u_1u_1)}=(\hat{u}_1\ast\hat{u}_1)$ and localizing these two factors separately, we obtain
		\begin{align*}
			|S_j|
			\lesssim
			\sum_{K_j}K_1^{2z}\sup_t\int_0^t\int_{\Gamma_4}\frac{\sigma_{I\!I}}{\Omega_3}&(\xi_1,\xi_2,\xi_{ab})(i\xi_{ab})\chi_{K_a}(\xi_a)\chi_{K_b}(\xi_b)
			\\
			&\times\hat{v}(\xi_1)\hat{v}(\xi_2)\hat{u}_1(\xi_a)\hat{u}_1(\xi_b)d\gamma_4(\xi_1,\xi_2,\xi_a,\xi_b)dr
		\end{align*}
		for $j\in[3]$. Define
		\begin{align*}
			\Pi(K_1,\!K_2,\!K_a,\!K_b)
			:=
			\|P_{K_1}v\|_{\spfb^{K_1}_T}
			\|P_{K_2}v\|_{\spfb^{K_2}_T}
			\|P_{K_a}u_1\|_{\spfb^{K_a}_T}
			\|P_{K_b}u_1\|_{\spfb^{K_b}_T}.
		\end{align*}
		
		Again, we start by estimating $S_1$. Here, we have $|\frac{\sigma_{I\!I}}{\Omega_3}(\xi_1,\xi_2,\xi_{ab})(i\xi_{ab})|\lesssim 1$. Thus, Lemma \ref{l_stqe} yields
		\begin{align*}
			|S_1|
			\lesssim_\delta
			T^\delta
			\sum_{K_j}\Pi(K_1,\!K_2,\!K_a,\!K_b)
			\times
			\begin{cases}
				K_2^\frac{1}{2}K_1^{-\frac{1}{2}}
				&\text{if }
				K_a\!\sim\!K_b\!\gg\!K_2,
				\\
				K_1^{-\frac{1}{2}}K_b^\frac{1}{2}
				&\text{if }
				K_a\!\sim\!K_2\!\gg\!K_b,
				\\
				K_1^{-\frac{1}{2}}K_a^\frac{1}{2}
				&\text{if }
				K_b\!\sim\!K_2\!\gg\!K_a.
			\end{cases}
		\end{align*}
		
		In the setting of $S_2$, we have $|\frac{\sigma_{I\!I}}{\Omega_3}(\xi_1,\xi_2,\xi_{ab})(i\xi_{ab})|\lesssim K_1K_2^{-1}$ and thus obtain
		\begin{align*}
			|S_2|
			\lesssim_\delta
			T^\delta
			\sum_{K_j}\Pi(K_1,\!K_2,\!K_a,\!K_b)
			\times
			\begin{cases}
				K_1^\frac{1}{2}K_2^{-\frac{1}{2}}
				&\text{if }
				K_a\!\sim\!K_b\!\gg\!K_3,
				\\
				K_2^{-\frac{1}{2}}K_b^\frac{1}{2}
				&\text{if }
				K_a\!\sim\!K_3\!\gg\!K_b,
				\\
				K_2^{-\frac{1}{2}}K_a^\frac{1}{2}
				&\text{if }
				K_b\!\sim\!K_3\!\gg\!K_a.
			\end{cases}
		\end{align*}
		
		Lastly, we consider $S_3$. Then, the derivative is applied to a low-frequency factor and we get $|\frac{\sigma_{I\!I}}{\Omega_3}(\xi_1,\xi_2,\xi_{ab})(i\xi_{ab})|\lesssim 1$. Lemma \ref{l_stqe} leads to
		\begin{align*}
			|S_3|
			\lesssim_\delta
			T^\delta
			\sum_{K_j}\Pi(K_1,\!K_2,\!K_a,\!K_b)\times
			\begin{cases}
				1
				&\text{if }
				K_a\!\sim\!K_b\!\gg\!K_1,
				\\
				K_1^{-1}K_a^\frac{1}{2}K_b^\frac{1}{2}
				&\text{if }
				K_1\!\gtrsim\!K_a\!\sim\!K_b\!\gtrsim\!K_3,
				\\
				K_1^{-1}K_a^\frac{1}{2}K_b^\frac{1}{2}
				&\text{if }
				K_a\!\sim\!K_3\!\gg\!K_b,
				\\
				K_1^{-1}K_b^\frac{1}{2}K_a^\frac{1}{2}
				&\text{if }
				K_b\!\sim\!K_3\!\gg\!K_a.
			\end{cases}
		\end{align*}
	\end{proof}
	
	The remaining three lemmata of this section deal with $E_{I\!I,7}$, $E_{I\!I,8}$, and $E_{I\!I,9}$.
	\begin{lem}\label{l_ee_ii_qua_bnd_i}
		Let $s>\frac{1}{2}$. Then, for all sufficiently small $\theta=\theta(s)>1$ and $\delta\in(0,\frac{1}{2})$, we have
		\begin{align*}
			|E_{I\!I,7}|
			\lesssim_\delta
			T^\delta\|v\|_{\spfc^z_T}^2\|w\|_{\spfc^s_T}\|b\|_{\spbc^s_T}.
		\end{align*}
	\end{lem}
	\begin{proof}
		We write $E_{I\!I,7}=S_1+S_2+S_3$ with
		\begin{align*}
			S_j
			:=
			\begin{cases}
				E_{I\!I,7}|_{K_2\sim K_3\gg K_1}
				&\text{if }
				j=1,
				\\
				E_{I\!I,7}|_{K_1\sim K_3\gg K_2}
				&\text{if }
				j=2,
				\\
				E_{I\!I,7}|_{K_1\sim K_2\gtrsim K_3}
				&\text{if }
				j=3.
			\end{cases}
		\end{align*}
		Using $\widehat{(vw)}=(\hat{v}\ast\hat{w})$ and localizing these two factors separately, we obtain
		\begin{align*}
			|S_j|
			\lesssim
			\sum_{K_j}K_1^{2z}\sup_t\int_0^t\int_{\Gamma_4}\frac{\sigma_{I\!I}}{\Omega_3}&(\xi_{ab},\xi_2,\xi_3)(i\xi_{ab})\chi_{K_a}(\xi_a)\chi_{K_b}(\xi_b)
			\\
			&\times\hat{v}(\xi_a)\hat{b}(\xi_b)\hat{v}(\xi_2)\hat{w}(\xi_3)d\gamma_4(\xi_a,\xi_b,\xi_2,\xi_3)dr
		\end{align*}
		for $j\in[3]$. Define
		\begin{align*}
			\Pi(K_a,\!K_b,\!K_2,\!K_3)
			:=
			\|P_{K_a}v\|_{\spfb^{K_a}_T}
			\|P_{K_b}b\|_{\spfb^{K_b}_T}
			\|P_{K_2}v\|_{\spfb^{K_2}_T}
			\|P_{K_3}w\|_{\spfb^{K_3}_T}.
		\end{align*}
		
		We begin with estimating $S_1$. Here, the derivative is on a low-frequency factor and we have $|\frac{\sigma_{I\!I}}{\Omega_3}(\xi_{ab},\xi_2,\xi_3)(i\xi_{ab})|\lesssim K_2^{-1}K_1$. Hence, Lemma \ref{l_stqe_bnd} leads to
		\begin{align*}
			|S_1|
			\lesssim_\delta
			T^\delta
			\sum_{K_j}\Pi(K_a,\!K_b,\!K_2,\!K_3)\times
			\begin{cases}
				K_2^{-\frac{1}{2}}
				&\text{if }
				K_a\!\sim\!K_b\!\gg\!K_2,
				\\
				K_2^{-1}K_a\frac{1}{2}
				&\text{if }
				K_2\!\gtrsim\!K_a\!\sim\!K_b\!\gg\!K_1,
				\\
				K_2^{-1}K_a\frac{1}{2}
				&\text{if }
				K_a\!\sim\!K_1\!\gg\!K_b,
				\\
				K_2^{-1}K_a\frac{1}{2}
				&\text{if }
				K_b\!\sim\!K_1\!\gg\!K_a.
			\end{cases}
		\end{align*}
		
		Next, we estimate $S_2$. We have $|\frac{\sigma_{I\!I}}{\Omega_3}(\xi_{ab},\xi_2,\xi_3)(i\xi_{ab})|\lesssim K_2^{-1}K_1$. Again, Lemma \ref{l_stqe_bnd} yields
		\begin{align*}
			|S_2|
			\lesssim_\delta
			T^\delta
			\sum_{K_j}\Pi(K_a,\!K_b,\!K_2,\!K_3)\times
			\begin{cases}
				K_2^{-\frac{1}{2}}
				&\text{if }
				K_a\!\sim\!K_b\!\gg\!K_1,
				\\
				K_2^{-\frac{1}{2}}
				&\text{if }
				K_a\!\sim\!K_1\!\gg\!K_b,
				\\
				K_2^{-1}\min\{K_2,K_a\}^\frac{1}{2}
				&\text{if }
				K_b\!\sim\!K_1\!\gg\!K_a.
			\end{cases}
		\end{align*}
		
		We continue by estimating $S_3$. The derivative is on a high-frequency factor and additionally, both $v_{K_a}$ and $v_{K_2}$ might be localized to high frequencies. Nonetheless, we use Lemma \ref{l_stqe_bnd} and get
		\begin{align*}
			|S_3|
			\lesssim_\delta
			T^\delta
			\sum_{K_j}\Pi(K_a,\!K_b,\!K_2,\!K_3)\times
			\begin{cases}
				K_1^{-\frac{1}{2}}
				&\text{if }
				K_a\!\sim\!K_b\!\gtrsim\!K_1,
				\\
				K_1^{-1}K_3^{\frac{1}{2}}
				&\text{if }
				K_a\!\sim\!K_1\!\gg\!K_b,
				\\
				K_1^{-1}\min\{K_a,K_3\}^\frac{1}{2}
				&\text{if }
				K_b\!\sim\!K_1\!\gg\!K_a.
			\end{cases}
		\end{align*}
	\end{proof}
	
	\begin{lem}\label{l_ee_ii_qua_bnd_ii}
		Let $s>\frac{1}{2}$. Then, for all sufficiently small $\theta=\theta(s)>1$ and $\delta\in(0,\frac{1}{2})$, we have
		\begin{align*}
			|E_{I\!I,8}|
			\lesssim_\delta
			T^\delta\|v\|_{\spfc^z_T}^2\|w\|_{\spfc^s_T}\|b\|_{\spbc^s_T}.
		\end{align*}
	\end{lem}
	\begin{proof}
		We write $E_{I\!I,8}=S_1+S_2+S_3$ with
		\begin{align*}
			S_j
			:=
			\begin{cases}
				E_{I\!I,8}|_{K_2\sim K_3\gg K_1}
				&\text{if }
				j=1,
				\\
				E_{I\!I,8}|_{K_1\sim K_3\gg K_2}
				&\text{if }
				j=2,
				\\
				E_{I\!I,8}|_{K_1\sim K_2\gtrsim K_3}
				&\text{if }
				j=3.
			\end{cases}
		\end{align*}
		Using $\widehat{(vw)}=(\hat{v}\ast\hat{w})$ and localizing these two factors separately, we obtain
		\begin{align*}
			|S_j|
			\lesssim
			\sum_{K_j}K_1^{2z}\sup_t\int_0^t\int_{\Gamma_4}\frac{\sigma_{I\!I}}{\Omega_3}&(\xi_1,\xi_{ab},\xi_3)(i\xi_{ab})\chi_{K_a}(\xi_a)\chi_{K_b}(\xi_b)
			\\
			&\times\hat{v}(\xi_1)\hat{v}(\xi_a)\hat{b}(\xi_b)\hat{w}(\xi_3)d\gamma_4(\xi_1,\xi_a,\xi_b,\xi_3)dr
		\end{align*}
		for $j\in[3]$. Define
		\begin{align*}
			\Pi(K_1,\!K_a,\!K_b,\!K_3)
			:=
			\|P_{K_1}v\|_{\spfb^{K_2}_T}
			\|P_{K_a}v\|_{\spfb^{K_a}_T}
			\|P_{K_b}b\|_{\spfb^{K_b}_T}
			\|P_{K_3}w\|_{\spfb^{K_3}_T}.
		\end{align*}
		
		For $S_1$, the derivative is on a low-frequency factor which leads to the bound $|\frac{\sigma_{I\!I}}{\Omega_3}(\xi_1,\xi_{ab},\xi_3)(i\xi_{ab})|\lesssim 1$. Hence, Lemma \ref{l_stqe_bnd} yields
		\begin{align*}
			|S_1|
			\lesssim_\delta
			T^\delta
			\sum_{K_j}\Pi(K_1,\!K_a,\!K_b,\!K_3)\times
			\begin{cases}
				K_1^{-\frac{1}{2}}
				&\text{if }
				K_a\!\sim\!K_b\!\gg\!K_2,
				\\
				K_1^{-\frac{1}{2}}
				&\text{if }
				K_2\!\gtrsim\!K_a\!\sim\!K_b,
				\\
				K_1^{-1}K_a^\frac{1}{2}
				&\text{if }
				K_b\!\sim\!K_2\!\gg\!K_a.
			\end{cases}
		\end{align*}
		
		In $S_2$, the derivative is on a high-frequency factor and $|\frac{\sigma_{I\!I}}{\Omega_3}(\xi_1,\xi_{ab},\xi_3)(i\xi_{ab})|\lesssim 1$. By applying Lemma \ref{l_stqe_bnd}, we have
		\begin{align*}
			|S_2|
			\lesssim_\delta
			T^\delta
			\sum_{K_j}\Pi(K_1,\!K_a,\!K_b,\!K_3)\times
			\begin{cases}
				K_1^{-\frac{1}{2}}
				&\text{if }
				K_a\!\sim\!K_b\!\gg\!K_1,
				\\
				K_1^{-1}K_a^{-\frac{1}{2}}
				&\text{if }
				K_1\!\gtrsim\!K_a\!\sim\!K_b\!\gtrsim\!K_2,
				\\
				K_1^{-1}K_a^{-\frac{1}{2}}
				&\text{if }
				K_a\!\sim\!K_2\!\gg\!K_b,
				\\
				K_1^{-1}K_a^{-\frac{1}{2}}
				&\text{if }
				K_b\!\sim\!K_2\!\gg\!K_a.
			\end{cases}
		\end{align*}
		
		We continue by estimating $S_3$. We get $|\frac{\sigma_{I\!I}}{\Omega_3}(\xi_1,\xi_{ab},\xi_3)(i\xi_{ab})|\lesssim 1$ and Lemma \ref{l_stqe_bnd} implies
		\begin{align*}
			|S_3|
			\lesssim_\delta
			T^\delta
			\sum_{K_j}\Pi(K_a,\!K_b,\!K_2,\!K_3)\times
			\begin{cases}
				K_1^{-\frac{1}{2}}
				&\text{if }
				K_a\!\sim\!K_b\!\gtrsim\!K_2,
				\\
				K_1^{-1}K_3^\frac{1}{2}
				&\text{if }
				K_a\!\sim\!K_2\!\gg\!K_b,
				\\
				K_1^{-1}K_a^\frac{1}{2}
				&\text{if }
				K_b\!\sim\!K_2\!\gg\!K_a.
			\end{cases}
		\end{align*}
	\end{proof}
	
	\begin{lem}\label{l_ee_ii_qua_bnd_iii}
		Let $s>\frac{1}{2}$. Then, for all sufficiently small $\theta=\theta(s)>1$ and $\delta\in(0,\frac{1}{2})$, we have
		\begin{align*}
			|E_{I\!I,9}|
			\lesssim_\delta
			T^\delta\|v\|_{\spfc^z_T}^2\|w\|_{\spfc^s_T}\|b\|_{\spbc^s_T}.
		\end{align*}
	\end{lem}
	\begin{proof}
		We write $E_{I\!I,9}=S_1+S_2+S_3$ with
		\begin{align*}
			S_j
			:=
			\begin{cases}
				E_{I\!I,9}|_{K_2\sim K_3\gg K_1}
				&\text{if }
				j=1,
				\\
				E_{I\!I,9}|_{K_1\sim K_3\gg K_2}
				&\text{if }
				j=2,
				\\
				E_{I\!I,9}|_{K_1\sim K_2\gtrsim K_3}
				&\text{if }
				j=3.
			\end{cases}
		\end{align*}
		Using $\widehat{(wb)}=(\hat{w}\ast\hat{b})$ and localizing these two factors separately, we obtain
		\begin{align*}
			|S_j|
			\lesssim
			\sum_{K_j}K_1^{2z}\sup_t\int_0^t\int_{\Gamma_4}\frac{\sigma_{I\!I}}{\Omega_3}&(\xi_1,\xi_2,\xi_{ab})(i\xi_{ab})\chi_{K_a}(\xi_a)\chi_{K_b}(\xi_b)
			\\
			&
			\times\hat{v}(\xi_1)\hat{v}(\xi_2)\hat{w}(\xi_a)\hat{b}(\xi_b)d\gamma_4(\xi_1,\xi_2,\xi_a,\xi_b)dr
		\end{align*}
		for $j\in[3]$. Define
		\begin{align*}
			\Pi(K_1,\!K_2,\!K_a,\!K_b)
			:=
			\|P_{K_1}v\|_{\spfb^{K_1}_T}
			\|P_{K_2}v\|_{\spfb^{K_2}_T}
			\|P_{K_a}w\|_{\spfb^{K_a}_T}
			\|P_{K_b}b\|_{\spfb^{K_b}_T}.
		\end{align*}
		
		First, we consider $S_1$. Here, $|\frac{\sigma_{I\!I}}{\Omega_3}(\xi_1,\xi_2,\xi_{ab})(i\xi_{ab})|\lesssim 1$. It suffices to apply Lemma \ref{l_stqe_bnd} yielding
		\begin{align*}
			|S_1|
			\lesssim_\delta
			T^\delta
			\sum_{K_j}\Pi(K_1,\!K_2,\!K_a,\!K_b)\times
			\begin{cases}
				K_1^{-\frac{1}{2}}
				&\text{if }
				K_a\!\sim\!K_b\!\gtrsim\!K_3,
				\\
				K_2^\frac{1}{2}K_1^{-1}
				&\text{if }
				K_3\!\sim\!K_a\!\gg\!K_b,
				\\
				K_1^{-1}\min\{K_2,K_a\}^\frac{1}{2}
				&\text{if }
				K_3\!\sim\!K_b\!\gg\!K_a.
			\end{cases}
		\end{align*}
		
		For $S_2$ we get the bound $|\frac{\sigma_{I\!I}}{\Omega_3}(\xi_1,\xi_2,\xi_{ab})(i\xi_{ab})|\lesssim K_1K_2^{-1}$. It suffices to apply Lemma \ref{l_stqe_bnd} leading to
		\begin{align*}
			|S_2|
			\lesssim_\delta
			T^\delta
			\sum_{K_j}\Pi(K_1,\!K_2,\!K_a,\!K_b)\times
			\begin{cases}
				K_2^{-\frac{1}{2}}
				&\text{if }
				K_a\!\sim\!K_b\!\gtrsim\!K_3,
				\\
				K_2^{-\frac{1}{2}}
				&\text{if }
				K_3\!\sim\!K_a\!\gg\!K_b,
				\\
				K_2^{-1}\min\{K_2,K_a\}^\frac{1}{2}
				&\text{if }
				K_3\!\sim\!K_b\!\gg\!K_a.
			\end{cases}
		\end{align*}
		
		Last, we consider $S_3$. Here, we have $|\frac{\sigma_{I\!I}}{\Omega_3}(\xi_1,\xi_2,\xi_{ab})(i\xi_{ab})|\lesssim K_3K_1^{-1}$ and Lemma \ref{l_stqe_bnd} implies
		\begin{align*}
			|S_3|
			\lesssim_\delta
			T^\delta
			\sum_{K_j}\Pi(K_1,\!K_2,\!K_a,\!K_b)\times
			\begin{cases}
				K_1^{-\frac{3}{2}}K_3
				&\text{if }
				K_a\!\sim\!K_b\!\gg\!K_1,
				\\
				K_1^{-2}K_a^\frac{1}{2}K_3
				&\text{if }
				K_1\!\gtrsim\!K_a\!\sim\!K_b\!\gtrsim\!K_3,
				\\
				K_1^{-2}K_a^\frac{1}{2}K_3
				&\text{if }
				K_a\!\sim\!K_3\!\gg\!K_b,
				\\
				K_1^{-2}K_a^\frac{1}{2}K_3
				&\text{if }
				K_b\!\sim\!K_3\!\gg\!K_a.
			\end{cases}
		\end{align*}
	\end{proof}
	
	Finally, Lemma \ref{l_ee_ii} follows from the estimates obtained in this section as follows:	From the beginning of this section, we know that
	\begin{align*}
		\|v\|_{\spec^z_T}^2
		\lesssim
		\|v_0\|_{H^z}^2
		+
		I\!I_1
		+
		I\!I_2
	\end{align*}
	holds. $I\!I_2$ is bounded in Lemma \ref{l_ee_ii_trv} and the estimate for $I\!I_1$ is a direct consequence of Lemmata \ref{l_ee_ii_qua_i}, \ref{l_ee_ii_qua_ii}, \ref{l_ee_ii_qua_iii}, \ref{l_ee_ii_qua_bnd_i}, \ref{l_ee_ii_qua_bnd_ii}, and \ref{l_ee_ii_qua_bnd_iii}.
\subsection{Proof of the third energy estimate}\label{ss_ee_iii}
In this section we prove Lemma \ref{l_ee_ii}. As before, fix $s>\frac{1}{2}$, $z=-\frac{1}{2}$, $T\in(0,1)$, and let $\theta(s)>1$ be small enough to satisfy all restrictions appearing in this section. Choose $b$ and $f$ satisfying \eqref{as_b_f} and let and $u_1$ and $u_2$ be classical solutions to \eqref{eq_bo_split} with initial datum $u_{1,0}$ and $u_{2,0}$. In the following, we write $v=u_1-u_2$, $v_0=u_{1,0}-u_{2,0}$, and $w=u_1+u_2$.

Clearly, $v$ is a classical solution to 
\begin{align*}
	\begin{cases}
		v_t
		+
		\mathcal{H}v_{xx}
		+
		(vv)_x
		+
		2(vu_2)_x
		+
		(vb)_x
		&=
		0,
		\\
		v(0,\cdot)
		&=
		v_0.
	\end{cases}
\end{align*}
For $K_1\in\mathbb{D}$ we obtain
\begin{align*}
	\frac{d}{dr}\|P_{K_1}v(r,x)\|_{L^2_x}^2
	&=
	2\int_\mathbb{R} -P_{K_1}v\left[P_{K_1}\mathcal{H}v_{xx}+P_{K_1}(vw)_x+P_{K_1}(vb)_x\right] dx
	\\&=
	2\int_\mathbb{R} P_{K_1}^2v_xvvdx
	+
	4\int_\mathbb{R} P_{K_1}^2v_xvu_2dx
	+
	2\int_\mathbb{R} P_{K_1}^2v_xvb dx.
\end{align*}
After multiplying with $K_1^{2s}$, integrating with respect to $r$ over $[0,t]$, taking the supremum over $t\in[0,T]$, and after performing summation over $K_1\in\mathbb{D}$, the above leads to
\begin{align*}\label{eq_ee_iii_red}
	\|v\|_{\spec^s_T}^2
	-
	\|v_0\|_{H^s}^2
	&\leq
	2\sum_{K_1}K_1^{2s}\sup_t\int_0^t\int_\mathbb{R}[P_{K_1}^2v_xvv](r,x)dxdr
	&:=I\!I\!I
	\\&+
	4\sum_{K_1}K_1^{2s}\sup_t\int_0^t\int_\mathbb{R}[P_{K_1}^2v_xvu_2](r,x)dxdr
	&:=I\!V
	\\&+
	2\sum_{K_1}K_1^{2s}\sup_t\int_0^t\int_\mathbb{R}[P_{K_1}^2v_xvb](r,x)dxdr
	&:=V.
\end{align*}

Note that $I\!I\!I$ is of the same form as $I_1$ in Section \ref{ss_ee_i}, which allows to bound $I\!I\!I$ by reusing large parts of Section \ref{ss_ee_i} with only one minor modification.

Similary, the term $I\!V$ has the same form as $I\!I_1$ in Section \ref{ss_ee_ii}. The integrand is the product of two low-regularity functions and one high-regularity function. Again, this suggests that the bound for $I\!V$ follows by arguing as in Section \ref{ss_ee_ii}. However, since $I\!V$ shall be estimated at regularity $s>\frac{1}{2}$, whereas $I\!I_1$ is estimated at regularity $z=-\frac{1}{2}$, there are some non-trivial changes in the proof which lead to slightly different estimates. Thus, we give those proofs in detail in this section.

The estimate for $V$ follows directly from Lemma \ref{l_ee_trv_uub}.
\begin{lem}\label{l_ee_iii_trv}
	Let $\varepsilon>0$. Then, we have
	\begin{align*}
		V
		\lesssim_\varepsilon
		T\|v\|_{\spec^s_T}^2\|b\|_{\spbc^{1+s+\varepsilon}_T}.
	\end{align*}
\end{lem}

Next, we bound $I\!I\!I$. Using the notation introduced in Section \ref{ss_ee_i} and repeating the steps done in the beginning of that section (for $u$ replaced by $v$), we conclude
\begin{align*}
	I\!I\!I
	=
	&\sum_{K_1}K_1^{2s}\sup_t\int_0^t\int_\mathbb{R}[P_{K_1}^2v_xvv](r,x)dxdr
	\\
	\leq
	&\sum_{C_L}K_1^{2s}\sup_t\int_0^t\int_{\Gamma_3}(\sigma_I\kappa_{3}\left[v\right])(\xi)d\gamma_3(\xi)dr
	&=:E_{I\!I\!I,1}
	\\ +
	&\sum_{C_M}K_1^{2s}\sup_t\int_0^t\int_{\Gamma_3}(\sigma_I\kappa_{3}\left[v\right])(\xi)d\gamma_3(\xi)dr
	&=:E_{I\!I\!I,2}
	\\ +
	&\sum_{C_H}K_1^{2s}\sup_t\left[\int_{\Gamma_3}(\sigma_I\Omega_3^{-1}\kappa_{3}\left[v\right])(\xi)d\gamma_3(\xi)\right]_{r=0}^{t}
	&=:E_{I\!I\!I,3}
	\\ +
	&\sum_{C_H}K_1^{2s}\sup_t\int_0^t\int_{\Gamma_3}(\sigma_I\Omega_3^{-1}\kappa_{1,2}\left[
	(vv)_x,v\right])(\xi)d\gamma_3(\xi)dr
	&=:E_{I\!I\!I,4}
	\\ +
	&2\sum_{C_H}K_1^{2s}\sup_t\int_0^t\int_{\Gamma_3}(\sigma_I\Omega_3^{-1}\kappa_{1,2}\left[(vu_2)_x,v\right])(\xi)d\gamma_3(\xi)dr
	&=:E_{I\!I\!I,*}
	\\ +
	&\sum_{C_H}K_1^{2s}\sup_t\int_0^t\int_{\Gamma_3}(\sigma_I\Omega_3^{-1}\kappa_{1,2}\left[(vb)_x,v\right])(\xi)d\gamma_3(\xi)dr
	&=:E_{I\!I\!I,5}.
\end{align*}
Each $E_{I\!I\!I,i}$, $i\in[5]$, can be estimated exactly as $E_{I,i}$ in Section \ref{ss_ee_i}. Thus, we conclude from Lemmata \ref{l_ee_tri}, \ref{l_ee_qua}, and \ref{l_ee_qua_bnd}:
\begin{lem}\label{l_ee_iii_re}
	Let $s>\frac{1}{4}$, $\zeta\in(0,\frac{1}{2})$. Then, for all sufficiently small $\theta(s)>1$ and $\delta\in(0,\frac{1}{2})$, we have
	\begin{align*}
		\sum_{i=1}^5 |E_{I\!I\!I,i}|
		\!\lesssim\!
		(T
		\!+\!
		TN^2
		\!+\!
		N^{-\zeta})\|v\|_{\spec_T^s}^2\|v\|_{\spec_T^s}
		\!+\!
		T^\delta\|v\|_{\spfc_T^s}^2\|v\|_{\spfc_T^s}
		(
		\|v\|_{\spfc_T^s}
		\!+\!
		\|b\|_{\spbc_T^s}
		).
	\end{align*}
\end{lem}
Next, we bound $E_{I\!I\!I,*}$.
\begin{lem}\label{l_ee_iii_ne}
	Let $s>\frac{1}{2}$, $\zeta\in(0,\frac{1}{2})$. Then, for all sufficiently small $\theta(s)>1$ and $\delta\in(0,\frac{1}{2})$, we have
	\begin{align*}
		E_{I\!I\!I,*}
		\!\lesssim_\delta\!
		T^\delta\|v\|_{\spfc^s_T}^3\|u_2\|_{\spfc^s_T}
	\end{align*}
\end{lem}
\begin{proof}
	Without loss of generality, we restrict to $K_1\sim K_2\gtrsim K_3$ and replace the integrand $\kappa_{1,2}\left[(vu_2)_x,v\right](\xi)$ by $\widehat{(vu_2)_x}(\xi_1)\hat{v}(\xi_2)\hat{v}(\xi_3)$. Writing $\widehat{(vu_2)}=\hat{v}\ast\hat{u_2}$, this leads us to consider
	\begin{align*}
		S
		:=
		\sum_{K_i}K_1^{2s}\sup_t\int_0^t\int_{\Gamma_4}(\sigma_I\Omega_3^{-1})(\xi)(-i\xi_{ab})\hat{v}(\xi_a)\hat{u_2}(\xi_b)\hat{v}(\xi_2)\hat{v}(\xi_3)d\gamma_4(\xi)dr.
	\end{align*}
	Define
	\begin{align*}
		\Pi(K_a,\!K_b,\!K_2,\!K_3)
		:=
		\|P_{K_a}v\|_{\spfb^{K_a}_T}
		\|P_{K_b}u_2\|_{\spfb^{K_b}_T}
		\|P_{K_2}v\|_{\spfb^{K_2}_T}
		\|P_{K_3}v\|_{\spfb^{K_3}_T}.
	\end{align*}
	Then, $|(\sigma_{I}\Omega_3^{-1})(\xi)(-i\xi_{ab})|\lesssim 1$ and an application of Lemma \ref{l_stqe} yields
	\begin{align*}
		|S|
		\lesssim_\delta
		T^\delta
		\sum_{K_j}\Pi(K_a,\!K_b,\!K_2,\!K_3)
		\times
		\begin{cases}
			K_1^{2s+\frac{1}{2}}K_3^\frac{1}{2}
			&\text{if }
			K_a\!\sim\!K_b\!\gg\!K_2,
			\\
			K_1^{2s}K_b^\frac{1}{2}K_3^\frac{1}{2}
			&\text{if }
			K_2\!\sim\!K_a\!\gtrsim\!K_b,
			\\
			K_1^{2s}K_a^\frac{1}{2}K_3^\frac{1}{2}
			&\text{if }
			K_2\!\sim\!K_b\!\gtrsim\!K_a.
		\end{cases}
	\end{align*}
	In each case the sum can be bounded by $\|v\|_{\spfc^s_T}^3\|u_2\|_{\spfc^s_T}$.
\end{proof}

The remaining part of this section is devoted to estimating $I\!V$. Performing the same steps as in the beginning of Section \ref{ss_ee_ii}, we obtain:
\begin{align*}
	I\!V
	\lesssim
	&\sum_{C_L}K_1^{2s}\sup_t\int_0^t\int_{\Gamma_3}(\sigma_{I\!I}[\hat{v},\hat{v},\hat{u}_2])(\xi)d\gamma_3(\xi)dr
	&&:=E_{I\!V,1}
	\\ +
	&\sum_{C_M}K_1^{2s}\sup_t\int_0^t\int_{\Gamma_3}(\sigma_{I\!I}[\hat{v},\hat{v},\hat{u}_2])(\xi)d\gamma_3(\xi)dr
	&&:=E_{I\!V,2}
	\\ +
	&\sum_{C_H}K_1^{2s}\sup_t\left[\int_{\Gamma_3}(\sigma_{I\!I}\Omega_3^{-1}[\hat{v},\hat{v},\hat{u}_2])(\xi)d\gamma_3(\xi)\right]_{r=0}^t
	&&:=E_{I\!V,3}
	\\ +
	&\sum_{C_H}K_1^{2s}\sup_t\int_0^t\int_{\Gamma_3}(\sigma_{I\!I}\Omega_3^{-1}[\widehat{(vw)_x},\hat{v},\hat{u}_2])(\xi)d\gamma_3(\xi)dr
	&&:=E_{I\!V,4}
	\\ +
	&\sum_{C_H}K_1^{2s}\sup_t\int_0^t\int_{\Gamma_3}(\sigma_{I\!I}\Omega_3^{-1}[\hat{v},\widehat{(vw)_x},\hat{u}_2])(\xi)d\gamma_3(\xi)dr
	&&:=E_{I\!V,5}
	\\ +
	&\sum_{C_H}K_1^{2s}\sup_t\int_0^t\int_{\Gamma_3}(\sigma_{I\!I}\Omega_3^{-1}[\hat{v},\hat{v},\widehat{(u_2u_2)_x}])(\xi)d\gamma_3(\xi)dr
	&&:=E_{I\!V,6}
	\\ +
	&\sum_{C_H}K_1^{2s}\sup_t\int_0^t\int_{\Gamma_3}(\sigma_{I\!I}\Omega_3^{-1}[\widehat{(vb)_x},\hat{v},\hat{u}_2])(\xi)d\gamma_3(\xi)dr
	&&:=E_{I\!V,7}
	\\ +
	&\sum_{C_H}K_1^{2s}\sup_t\int_0^t\int_{\Gamma_3}(\sigma_{I\!I}\Omega_3^{-1}[\hat{v},\widehat{(vb)_x},\hat{u}_2])(\xi)d\gamma_3(\xi)dr
	&&:=E_{I\!V,8}
	\\ +
	&\sum_{C_H}K_1^{2s}\sup_t\int_0^t\int_{\Gamma_3}(\sigma_{I\!I}\Omega_3^{-1}[\hat{v},\hat{v},\widehat{(u_2b)_x}])(\xi)d\gamma_3(\xi)dr
	&&:=E_{I\!V,9}
	\\ +
	&\sum_{C_H}K_1^{2s}\sup_t\int_0^t\int_{\Gamma_3}(\sigma_{I\!I}\Omega_3^{-1}[\hat{v},\hat{v},\hat{f}])(\xi)d\gamma_3(\xi)dr
	&&:=E_{I\!V,10}.
\end{align*}

As before, we start by estimating the low-frequency contributions $E_{I\!V,1}$ and $E_{I\!V,2}$, the boundary term $E_{I\!V,3}$, and the term involving the forcing $E_{I\!V,10}$. Slightly modifying Lemma \ref{l_ee_tri}, we conclude:
\begin{lem}\label{l_ee_iii_tri}
	Let $\zeta\in(0,\frac{1}{2})$, $z=-\frac{1}{2}$, and $s>\frac{1}{2}$. Then, we have
	\begin{align*}
		|E_{I\!V,1}|
		&\lesssim
		T\|v\|_{\spec^s_T}^2\|u_2\|_{\spec^0_T}
		+
		T\|v\|_{\spec^z_T}\|v\|_{\spec^s_T}\|u_2\|_{\spec^{s+1}_T},
		\\
		|E_{I\!V,2}|
		&\lesssim
		TN^{2}\|v\|_{L^\infty H^s}^2\|w\|_{L^\infty H^s},
		\\
		|E_{I\!V,3}|
		&\lesssim_\zeta
		N^{-\zeta}\|v\|_{L^\infty H^s}^2\|w\|_{L^\infty H^s},
		\\
		|E_{I\!V,10}|
		&\lesssim
		T\|v\|_{L^\infty H^s}^2\|f\|_{L^\infty H^s}.
	\end{align*}
\end{lem}
Next, we estimate the terms $E_{I\!V,4}$, $E_{I\!V,5}$, and $E_{I\!V,6}$.
\begin{lem}\label{l_ee_iii_qua_i}
	Let $s>\frac{1}{2}$. Then, for all sufficiently small $\theta=\theta(s)>1$ and $\delta\in(0,\frac{1}{2})$, we have
	\begin{align*}
		|E_{I\!V,4}|
		\lesssim_\delta
		T^\delta\|v\|_{\spfc^s_T}^2\|w\|_{\spfc^s_T}\|u_2\|_{\spfc^s_T}
		+
		T^\delta\|v\|_{\spfc^z_T}\|v\|_{\spfc^s_T}\|w\|_{\spfc^s_T}\|u_2\|_{\spfc^{s+1}_T}.
	\end{align*}
\end{lem}
\begin{proof}
	We split $E_{I\!V,4}=S_1+S_2+S_3$ with
	\begin{align*}
		S_j
		:=
		\begin{cases}
			E_{I\!V,4}|_{K_2\sim K_3\gg K_1}
			&\text{if }
			j=1,
			\\
			E_{I\!V,4}|_{K_1\sim K_3\gg K_2}
			&\text{if }
			j=2,
			\\
			E_{I\!V,4}|_{K_1\sim K_2\gtrsim K_3}
			&\text{if }
			j=3.
		\end{cases}
	\end{align*}
	Using $\widehat{(vw)}=(\hat{v}\ast\hat{w})$ and localizing these two factors separately, we obtain
	\begin{align*}
		|S_j|
		\lesssim
		\sum_{K_j}K_1^{2s}\sup_t\int_0^t\int_{\Gamma_4}\frac{\sigma_{I\!I}}{\Omega_3}&(\xi_{ab},\xi_2,\xi_3)(i\xi_{ab})\chi_{K_a}(\xi_a)\chi_{K_b}(\xi_b)
		\\
		&\times\hat{v}(\xi_a)\hat{w}(\xi_b)\hat{v}(\xi_2)\hat{u}_2(\xi_3)d\gamma_4(\xi_a,\xi_b,\xi_2,\xi_3)dr
	\end{align*}
	for $j\in[3]$. Define
	\begin{align*}
		\Pi(K_a,\!K_b,\!K_2,\!K_3)
		:=
		\|P_{K_a}v\|_{\spfb^{K_a}_T}
		\|P_{K_b}w\|_{\spfb^{K_b}_T}
		\|P_{K_2}v\|_{\spfb^{K_2}_T}
		\|P_{K_3}u_2\|_{\spfb^{K_3}_T}.
	\end{align*}
	
	We begin by estimating $S_1$. As the derivative is on a low-frequency, we have $|\frac{\sigma_{I\!I}}{\Omega_3}(\xi_{ab},\xi_2,\xi_3)(i\xi_{ab})|\lesssim K_1K_2^{-1}$ and it suffices to apply Lemma \ref{l_stqe} leading to
	\begin{align*}
		|S_1|
		\lesssim_\delta
		T^\delta
		\sum_{K_j}\Pi(K_a,\!K_b,\!K_2,\!K_3)
		\times
		\begin{cases}
			K_1^{2s+1}
			&\text{if }
			K_a\!\sim\!K_b\!\gg\!K_2,
			\\
			K_2^{-1}K_a^\frac{1}{2}K_b^\frac{1}{2}K_1^{2s+1}
			&\text{if }
			K_2\!\gtrsim\!K_a\!\sim\!K_b\!\gg\!K_1,
			\\
			K_2^{-1}K_a^{2s+\frac{3}{2}}K_b^\frac{1}{2}
			&\text{if }
			K_a\!\sim\!K_1\!\gg\!K_b,
			\\
			K_2^{-1}K_b^{2s+\frac{3}{2}}K_a^\frac{1}{2}
			&\text{if }
			K_b\!\sim\!K_1\!\gg\!K_a.
		\end{cases}
	\end{align*}
	Dyadic summation in each case leads to the upper bound $\|v\|_{\spfc^s_T}^2\|w\|_{\spfc^s_T}\|u_2\|_{\spfc^s_T}$.
	
	Next, we consider $S_2$. Then, the derivative is on a high frequency, but the factor $v_{K_2}$ is always localized to low frequencies. Hence, $|\frac{\sigma_{I\!I}}{\Omega_3}(\xi_{ab},\xi_2,\xi_3)(i\xi_{ab})|\lesssim K_1K_2^{-1}$ and Lemma \ref{l_stqe} implies
	\begin{align*}
		|S_2|
		\lesssim_\delta
		T^\delta
		\sum_{K_j}\Pi(K_a,\!K_b,\!K_2,\!K_3)
		\times
		\begin{cases}
			K_3^{2s+\frac{3}{2}}K_2^{-\frac{1}{2}}
			&\text{if }
			K_a\!\sim\!K_b\!\gg\!K_1,
			\\
			K_a^{2s+1}K_b^\frac{1}{2}K_2^{-\frac{1}{2}}
			&\text{if }
			K_a\!\sim\!K_1\!\gg\!K_b,
			\\
			K_b^{2s+1}K_a^\frac{1}{2}K_2^{-\frac{1}{2}}
			&\text{if }
			K_b\!\sim\!K_1\!\gg\!K_a.
		\end{cases}
	\end{align*}
	Here, we have too many derivatives on the high-frequencies factor and, thus, we only manage to obtain the upper bound $\|v\|_{\spfc^z_T}\|v\|_{\spfc^s_T}\|w\|_{\spfc^s_T}\|u_2\|_{\spfc^{s+1}_T}$ after summation. 
	
	Lastly, we establish a bound for $S_3$. Again, the derivative is on a high-frequency factor and additionally, both $v_{K_a}$ and $v_{K_2}$ might be localized to high frequencies. Nonetheless, we obtain $|\frac{\sigma_{I\!I}}{\Omega_3}(\xi_{ab},\xi_2,\xi_3)(i\xi_{ab})|\lesssim 1$ and Lemma \ref{l_stqe} yields
	\begin{align*}
		|S_3|
		\lesssim_\delta
		T^\delta
		\sum_{K_j}\Pi(K_a,\!K_b,\!K_2,\!K_3)
		\times
		\begin{cases}
			K_2^{2s+\frac{1}{2}}K_3^\frac{1}{2}
			&\text{if }
			K_a\!\sim\!K_b\!\gtrsim\!K_1,
			\\
			K_2^{2s}K_b^\frac{1}{2}K_3^\frac{1}{2}
			&\text{if }
			K_a\!\sim\!K_1\!\gg\!K_b,
			\\
			K_2^{2s}K_a^\frac{1}{2}K_3^\frac{1}{2}
			&\text{if }
			K_b\!\sim\!K_1\!\gg\!K_a.
		\end{cases}
	\end{align*}
	Dyadic summation leads to the same upper bound as for $|S_1|$.
\end{proof}
\begin{lem}\label{l_ee_iii_qua_ii}
	Let $s>\frac{1}{2}$. Then, for all sufficiently small $\theta=\theta(s)>1$ and $\delta\in(0,\frac{1}{2})$, we have
	\begin{align*}
		|E_{I\!V,5}|
		\lesssim_\delta
		T^\delta\|v\|_{\spfc^s_T}^2\|w\|_{\spfc^s_T}\|u_2\|_{\spfc^s_T}
		+
		T^\delta\|v\|_{\spfc^z_T}\|v\|_{\spfc^s_T}\|w\|_{\spfc^s_T}\|u_2\|_{\spfc^{s+1}_T}.
	\end{align*}
\end{lem}
\begin{proof}
	We split $E_{I\!V,5}=S_1+S_2+S_3$ with
	\begin{align*}
		S_j
		:=
		\begin{cases}
			E_{I\!V,5}|_{K_2\sim K_3\gg K_1}
			&\text{if }
			j=1,
			\\
			E_{I\!V,5}|_{K_1\sim K_3\gg K_2}
			&\text{if }
			j=2,
			\\
			E_{I\!V,5}|_{K_1\sim K_2\gtrsim K_3}
			&\text{if }
			j=3.
		\end{cases}
	\end{align*}
	Using $\widehat{(vw)}=(\hat{v}\ast\hat{w})$ and localizing these two factors separately, we obtain
	\begin{align*}
		|S_j|
		\lesssim
		\sum_{K_j}K_1^{2s}\sup_t\int_0^t\int_{\Gamma_4}\frac{\sigma_{I\!I}}{\Omega_3}&(\xi_1,\xi_{ab},\xi_3)(i\xi_{ab})\chi_{K_a}(\xi_a)\chi_{K_b}(\xi_b)
		\\
		&\times\hat{v}(\xi_1)\hat{v}(\xi_a)\hat{w}(\xi_b)\hat{u}_2(\xi_3)d\gamma_4(\xi_1,\xi_a,\xi_b,\xi_3)dr
	\end{align*}
	for $j\in[3]$. Define
	\begin{align*}
		\Pi(K_1,\!K_a,\!K_b,\!K_3)
		:=
		\|P_{K_1}v\|_{\spfb^{K_1}_T}
		\|P_{K_a}v\|_{\spfb^{K_a}_T}
		\|P_{K_b}w\|_{\spfb^{K_b}_T}
		\|P_{K_3}u_2\|_{\spfb^{K_3}_T}.
	\end{align*}
	
	We begin by estimating $S_1$. Due to $|\frac{\sigma_{I\!I}}{\Omega_3}(\xi_{ab},\xi_2,\xi_3)(i\xi_{ab})|\lesssim 1$, we can apply Lemma \ref{l_stqe} and get
	\begin{align*}
		|S_1|
		\lesssim_\delta
		T^\delta
		\sum_{K_j}\Pi(K_1,\!K_a,\!K_b,\!K_3)
		\times
		\begin{cases}
			K_3^\frac{1}{2}K_1^{2s+\frac{1}{2}}
			&\text{if }
			K_a\!\sim\!K_b\!\gg\!K_2,
			\\
			K_1^{2s+\frac{1}{2}}K_b^\frac{1}{2}
			&\text{if }
			K_a\sim K_2\gg K_b,
			\\
			K_1^{2s+\frac{1}{2}}K_a^\frac{1}{2}
			&\text{if }
			K_b\sim K_2\gg K_a.
		\end{cases}
	\end{align*}
	Dyadic summation leads to the desired bound.
	
	Next, we consider $S_2$. Again, we have $|\frac{\sigma_{I\!I}}{\Omega_3}(\xi_{ab},\xi_2,\xi_3)(i\xi_{ab})|\lesssim 1$. Together with Lemma \ref{l_stqe}, we get
	\begin{align*}
		|S_2|
		\lesssim_\delta
		T^\delta
		\sum_{K_j}\Pi(K_1,\!K_a,\!K_b,\!K_3)
		\times
		\begin{cases}
			K_3^{2s+1}
			&\text{if }
			K_a\!\sim\!K_b\!\gg\!K_1,
			\\
			K_1^{2s}K_a^\frac{1}{2}K_b^\frac{1}{2}
			&\text{if }
			K_1\gtrsim K_a\!\sim\!K_b\!\gtrsim\!K_2,
			\\
			K_1^{2s}K_a^\frac{1}{2}K_b^\frac{1}{2}
			&\text{if }
			K_a\!\sim\!K_2\!\gg\!K_b,
			\\
			K_1^{2s}K_a^\frac{1}{2}K_b^\frac{1}{2}
			&\text{if }
			K_b\!\sim\!K_2\!\gg\!K_a.
		\end{cases}
	\end{align*}
	
	Lastly, we establish a bound for $S_3$. Again, we have $|\frac{\sigma_{I\!I}}{\Omega_3}(\xi_{ab},\xi_2,\xi_3)(i\xi_{ab})|\lesssim 1$. An application of Lemma \ref{l_stqe} yields
	\begin{align*}
		|S_3|
		\lesssim_\delta
		T^\delta
		\sum_{K_j}\Pi(K_1,\!K_a,\!K_b,\!K_3)
		\times
		\begin{cases}
			K_1^{2s+\frac{1}{2}}K_3^\frac{1}{2}
			&\text{if }
			K_a\!\sim\!K_b\!\gtrsim\!K_2,
			\\
			K_a^{2s}K_b^\frac{1}{2}K_3^\frac{1}{2}
			&\text{if }
			K_a\!\sim\!K_2\!\gg\!K_b,
			\\
			K_b^{2s}K_a^\frac{1}{2}K_3^\frac{1}{2}
			&\text{if }
			K_b\!\sim\!K_2\!\gg\!K_a.
		\end{cases}
	\end{align*}
\end{proof}

\begin{lem}\label{l_ee_iii_qua_iii}
	Let $s>\frac{1}{2}$. Then, for all sufficiently small $\theta=\theta(s)>1$ and $\delta\in(0,\frac{1}{2})$, we have
	\begin{align*}
		|E_{I\!V,6}|
		\lesssim_\delta
		T^\delta\|v\|_{\spfc^s_T}^2\|u_2\|_{\spfc^s_T}^2
		+
		T^\delta\|v\|_{\spfc^z_T}\|v\|_{\spfc^s_T}\|u_2\|_{\spfc^s_T}\|u_2\|_{\spfc^{s+1}_T}.
	\end{align*}
\end{lem}
\begin{proof}
	It suffices to consider the case $u_2=0$. We write $E_{I\!V,6}=S_1+S_2+S_3$ with
	\begin{align*}
		S_j
		:=
		\begin{cases}
			E_{I\!V,6}|_{K_2\sim K_3\gg K_1}
			&\text{if }
			j=1,
			\\
			E_{I\!V,6}|_{K_1\sim K_3\gg K_2}
			&\text{if }
			j=2,
			\\
			E_{I\!V,6}|_{K_1\sim K_2\gtrsim K_3}
			&\text{if }
			j=3.
		\end{cases}
	\end{align*}
	Using $\widehat{(u_2u_2)}=(\hat{u}_2\ast\hat{u}_2)$ and localizing these two factors separately, we obtain
	\begin{align*}
		|S_j|
		\lesssim
		\sum_{K_j}K_1^{2s}\sup_t\int_0^t\int_{\Gamma_4}\frac{\sigma_{I\!I}}{\Omega_3}&(\xi_1,\xi_2,\xi_{ab})(i\xi_{ab})\chi_{K_a}(\xi_a)\chi_{K_b}(\xi_b)
		\\
		&\times\hat{v}(\xi_1)\hat{v}(\xi_2)\hat{u}_2(\xi_a)\hat{u}_2(\xi_b)d\gamma_4(\xi_1,\xi_2,\xi_a,\xi_b)dr
	\end{align*}
	for $j\in[3]$. Define
	\begin{align*}
		\Pi(K_1,\!K_2,\!K_a,\!K_b)
		:=
		\|P_{K_1}v\|_{\spfb^{K_1}_T}
		\|P_{K_2}v\|_{\spfb^{K_2}_T}
		\|P_{K_a}u_2\|_{\spfb^{K_a}_T}
		\|P_{K_b}u_2\|_{\spfb^{K_b}_T}.
	\end{align*}
	
	Again, we start by estimating $S_1$. Here, we have $|\frac{\sigma_{I\!I}}{\Omega_3}(\xi_1,\xi_2,\xi_{ab})(i\xi_{ab})|\lesssim 1$. Thus, Lemma \ref{l_stqe} yields
	\begin{align*}
		|S_1|
		\lesssim_\delta
		T^\delta
		\sum_{K_j}\Pi(K_1,\!K_2,\!K_a,\!K_b)
		\times
		\begin{cases}
			K_2^\frac{1}{2}K_1^{2s+\frac{1}{2}}
			&\text{if }
			K_a\!\sim\!K_b\!\gg\!K_2,
			\\
			K_1^{2s+\frac{1}{2}}K_b^\frac{1}{2}
			&\text{if }
			K_a\!\sim\!K_2\!\gg\!K_b,
			\\
			K_1^{2s+\frac{1}{2}}K_a^\frac{1}{2}
			&\text{if }
			K_b\!\sim\!K_2\!\gg\!K_a.
		\end{cases}
	\end{align*}
	
	For $S_2$ we have $|\frac{\sigma_{I\!I}}{\Omega_3}(\xi_1,\xi_2,\xi_{ab})(i\xi_{ab})|\lesssim K_1K_2^{-1}$ and thus
	\begin{align*}
		|S_2|
		\lesssim_\delta
		T^\delta
		\sum_{K_j}\Pi(K_1,\!K_2,\!K_a,\!K_b)
		\times
		\begin{cases}
			K_1^{2s+\frac{3}{2}}K_2^{-\frac{1}{2}}
			&\text{if }
			K_a\!\sim\!K_b\!\gg\!K_3,
			\\
			K_1^{2s+1}K_2^{-\frac{1}{2}}K_b^\frac{1}{2}
			&\text{if }
			K_a\!\sim\!K_3\!\gg\!K_b,
			\\
			K_1^{2s+1}K_2^{-\frac{1}{2}}K_a^\frac{1}{2}
			&\text{if }
			K_b\!\sim\!K_3\!\gg\!K_a.
		\end{cases}
	\end{align*}
	
	Last, we consider $S_3$. Then, the derivative is applied to a low frequency factor and we get $|\frac{\sigma_{I\!I}}{\Omega_3}(\xi_1,\xi_2,\xi_{ab})(i\xi_{ab})|\lesssim 1$. Lemma \ref{l_stqe} leads to
	\begin{align*}
		|S_3|
		\lesssim_\delta
		T^\delta
		\sum_{K_j}\Pi(K_1,\!K_2,\!K_a,\!K_b)\times
		\begin{cases}
			K_1^{2s+1}
			&\text{if }
			K_a\!\sim\!K_b\!\gg\!K_1,
			\\
			K_1^{2s}K_a^\frac{1}{2}K_b^\frac{1}{2}
			&\text{if }
			K_1\!\gtrsim\!K_a\!\sim\!K_b\!\gtrsim\!K_3,
			\\
			K_1^{2s}K_a^\frac{1}{2}K_b^\frac{1}{2}
			&\text{if }
			K_a\!\sim\!K_3\!\gg\!K_b,
			\\
			K_1^{2s}K_b^\frac{1}{2}K_a^\frac{1}{2}
			&\text{if }
			K_b\!\sim\!K_3\!\gg\!K_a.
		\end{cases}
	\end{align*}
\end{proof}
It remains to bound the terms $E_{I\!V,7}$, $E_{I\!V,8}$, and $E_{I\!V,9}$.
\begin{lem}\label{l_ee_iii_qua_bnd_i}
	Let $s>\frac{1}{2}$. Then, for all sufficiently small $\theta=\theta(s)>1$ and $\delta\in(0,\frac{1}{2})$, we have
	\begin{align*}
		|E_{I\!V,7}|
		\lesssim_\delta
		T^\delta\|v\|_{\spfc^s_T}^2\|u_2\|_{\spfc^s_T}\|b\|_{\spbc^s_T}
		+
		T^\delta\|v\|_{\spfc^z_T}\|v\|_{\spfc^s_T}\|u_2\|_{\spfc^{s+1}_T}\|b\|_{\spbc^s_T}.
	\end{align*}
\end{lem}
\begin{proof}
	We write $E_{I\!V,7}=S_1+S_2+S_3$ with
	\begin{align*}
		S_j
		:=
		\begin{cases}
			E_{I\!V,7}|_{K_2\sim K_3\gg K_1}
			&\text{if }
			j=1,
			\\
			E_{I\!V,7}|_{K_1\sim K_3\gg K_2}
			&\text{if }
			j=2,
			\\
			E_{I\!V,7}|_{K_1\sim K_2\gtrsim K_3}
			&\text{if }
			j=3.
		\end{cases}
	\end{align*}
	Using $\widehat{(vw)}=(\hat{v}\ast\hat{w})$ and localizing these two factors separately, we obtain
	\begin{align*}
		|S_j|
		\lesssim
		\sum_{K_j}K_1^{2s}\sup_t\int_0^t\int_{\Gamma_4}\frac{\sigma_{I\!I}}{\Omega_3}&(\xi_{ab},\xi_2,\xi_3)(i\xi_{ab})\chi_{K_a}(\xi_a)\chi_{K_b}(\xi_b)
		\\
		&\times\hat{v}(\xi_a)\hat{b}(\xi_b)\hat{v}(\xi_2)\hat{u}_2(\xi_3)d\gamma_4(\xi_a,\xi_b,\xi_2,\xi_3)dr
	\end{align*}
	for $j\in[3]$. Define
	\begin{align*}
		\Pi(K_a,\!K_b,\!K_2,\!K_3)
		:=
		\|P_{K_a}v\|_{\spfb^{K_a}_T}
		\|P_{K_b}b\|_{\spfb^{K_b}_T}
		\|P_{K_2}v\|_{\spfb^{K_2}_T}
		\|P_{K_3}u_2\|_{\spfb^{K_3}_T}.
	\end{align*}
	
	We begin with estimating $S_1$. Here, the derivative is on a low-frequency factor and we have $|\frac{\sigma_{I\!I}}{\Omega_3}(\xi_{ab},\xi_2,\xi_3)(i\xi_{ab})|\lesssim K_2^{-1}K_1$. Hence, Lemma \ref{l_stqe_bnd} leads to
	\begin{align*}
		|S_1|
		\lesssim_\delta
		T^\delta
		\sum_{K_j}\Pi(K_a,\!K_b,\!K_2,\!K_3)\times
		\begin{cases}
			K_2^{-\frac{1}{2}}K_1^{2s+1}
			&\text{if }
			K_a\!\sim\!K_b\!\gg\!K_2,
			\\
			K_2^{-1}K_a^\frac{1}{2}K_1^{2s+1}
			&\text{if }
			K_2\!\gtrsim\!K_a\!\sim\!K_b\!\gg\!K_1,
			\\
			K_2^{-1}K_a^{2s+\frac{3}{2}}
			&\text{if }
			K_a\!\sim\!K_1\!\gg\!K_b,
			\\
			K_2^{-1}K_a^{2s+\frac{3}{2}}
			&\text{if }
			K_b\!\sim\!K_1\!\gg\!K_a.
		\end{cases}
	\end{align*}
	
	For $S_2$, we have $|\frac{\sigma_{I\!I}}{\Omega_3}(\xi_{ab},\xi_2,\xi_3)(i\xi_{ab})|\lesssim K_2^{-1}K_1$, since the derivative is on a high-frequency factor. Again, Lemma \ref{l_stqe_bnd} yields
	\begin{align*}
		|S_2|
		\lesssim_\delta
		T^\delta
		\sum_{K_j}\Pi(K_a,\!K_b,\!K_2,\!K_3)\times
		\begin{cases}
			K_1^{2s+1}K_2^{-\frac{1}{2}}
			&\text{if }
			K_a\!\sim\!K_b\!\gg\!K_1,
			\\
			K_1^{2s+1}K_2^{-\frac{1}{2}}
			&\text{if }
			K_a\!\sim\!K_1\!\gg\!K_b,
			\\
			K_1^{2s+1}K_2^{-1}\min\{K_2,K_a\}^\frac{1}{2}
			&\text{if }
			K_b\!\sim\!K_1\!\gg\!K_a.
		\end{cases}
	\end{align*}
	
	We continue by estimating $S_3$. The derivative is on a high-frequency factor and additionally, both $v_{K_a}$ and $v_{K_2}$ might be localized to high frequencies. Nonetheless, Lemma \ref{l_stqe_bnd} leads to
	\begin{align*}
		|S_3|
		\lesssim_\delta
		T^\delta
		\sum_{K_j}\Pi(K_a,\!K_b,\!K_2,\!K_3)\times
		\begin{cases}
			K_1^{2s+\frac{1}{2}}
			&\text{if }
			K_a\!\sim\!K_b\!\gtrsim\!K_1,
			\\
			K_1^{2s}K_3^{\frac{1}{2}}
			&\text{if }
			K_a\!\sim\!K_1\!\gg\!K_b,
			\\
			K_1^{2s}\min\{K_a,K_3\}^\frac{1}{2}
			&\text{if }
			K_b\!\sim\!K_1\!\gg\!K_a.
		\end{cases}
	\end{align*}
\end{proof}

\begin{lem}\label{l_ee_iii_qua_bnd_ii}
	Let $s>\frac{1}{2}$. Then, for all sufficiently small $\theta=\theta(s)>1$ and $\delta\in(0,\frac{1}{2})$, we have
	\begin{align*}
		|E_{I\!V,8}|
		\lesssim_\delta
		T^\delta\|v\|_{\spfc^s_T}^2\|u_2\|_{\spfc^s_T}\|b\|_{\spbc^s_T}.
	\end{align*}
\end{lem}
\begin{proof}
	We write $E_{I\!V,8}=S_1+S_2+S_3$ with
	\begin{align*}
		S_j
		:=
		\begin{cases}
			E_{I\!V,8}|_{K_2\sim K_3\gg K_1}
			&\text{if }
			j=1,
			\\
			E_{I\!V,8}|_{K_1\sim K_3\gg K_2}
			&\text{if }
			j=2,
			\\
			E_{I\!V,8}|_{K_1\sim K_2\gtrsim K_3}
			&\text{if }
			j=3.
		\end{cases}
	\end{align*}
	Using $\widehat{(vw)}=(\hat{v}\ast\hat{w})$ and localizing these two factors separately, we obtain
	\begin{align*}
		|S_j|
		\lesssim
		\sum_{K_j}K_1^{2s}\sup_t\int_0^t\int_{\Gamma_4}\frac{\sigma_{I\!I}}{\Omega_3}&(\xi_1,\xi_{ab},\xi_3)(i\xi_{ab})\chi_{K_a}(\xi_a)\chi_{K_b}(\xi_b)
		\\
		&\times\hat{v}(\xi_1)\hat{v}(\xi_a)\hat{b}(\xi_b)\hat{u}_2(\xi_3)d\gamma_4(\xi_1,\xi_a,\xi_b,\xi_3)dr
	\end{align*}
	for $j\in[3]$. Define
	\begin{align*}
		\Pi(K_1,\!K_a,\!K_b,\!K_3)
		:=
		\|P_{K_1}v\|_{\spfb^{K_2}_T}
		\|P_{K_a}v\|_{\spfb^{K_a}_T}
		\|P_{K_b}b\|_{\spfb^{K_b}_T}
		\|P_{K_3}u_2\|_{\spfb^{K_3}_T}.
	\end{align*}
	
	For $S_1$, the derivative is on a low-frequency factor and we have the estimate $|\frac{\sigma_{I\!I}}{\Omega_3}(\xi_1,\xi_{ab},\xi_3)(i\xi_{ab})|\lesssim 1$. Hence, Lemma \ref{l_stqe_bnd} leads to
	\begin{align*}
		|S_1|
		\lesssim_\delta
		T^\delta
		\sum_{K_j}\Pi(K_1,\!K_a,\!K_b,\!K_3)\times
		\begin{cases}
			K_1^{2s+\frac{1}{2}}
			&\text{if }
			K_a\!\sim\!K_b\!\gg\!K_2,
			\\
			K_1^{2s+\frac{1}{2}}
			&\text{if }
			K_2\!\sim\!K_a\!\gtrsim\!K_b,
			\\
			K_1^{2s}K_a^\frac{1}{2}
			&\text{if }
			K_b\!\sim\!K_2\!\gg\!K_a.
		\end{cases}
	\end{align*}
	
	In $S_2$, the derivative is on a high-frequency factor and $|\frac{\sigma_{I\!I}}{\Omega_3}(\xi_1,\xi_{ab},\xi_3)(i\xi_{ab})|\lesssim 1$. Again, Lemma \ref{l_stqe_bnd} yields
	\begin{align*}
		|S_2|
		\lesssim_\delta
		T^\delta
		\sum_{K_j}\Pi(K_1,\!K_a,\!K_b,\!K_3)\times
		\begin{cases}
			K_1^{2s+\frac{1}{2}}
			&\text{if }
			K_a\!\sim\!K_b\!\gg\!K_1,
			\\
			K_1^{2s}K_a^{-\frac{1}{2}}
			&\text{if }
			K_1\!\gtrsim\!K_a\!\sim\!K_b\!\gtrsim\!K_2,
			\\
			K_1^{2s}K_a^{-\frac{1}{2}}
			&\text{if }
			K_a\!\sim\!K_2\!\gg\!K_b,
			\\
			K_1^{2s}K_a^{-\frac{1}{2}}
			&\text{if }
			K_b\!\sim\!K_2\!\gg\!K_a.
		\end{cases}
	\end{align*}
	
	We continue by estimating $S_3$. Again $|\frac{\sigma_{I\!I}}{\Omega_3}(\xi_1,\xi_{ab},\xi_3)(i\xi_{ab})|\lesssim 1$ holds. An application of Lemma \ref{l_stqe_bnd} leads to
	\begin{align*}
		|S_3|
		\lesssim_\delta
		T^\delta
		\sum_{K_j}\Pi(K_a,\!K_b,\!K_2,\!K_3)\times
		\begin{cases}
			K_1^{2s+\frac{1}{2}}
			&\text{if }
			K_a\!\sim\!K_b\!\gtrsim\!K_2,
			\\
			K_1^{2s}K_3^\frac{1}{2}
			&\text{if }
			K_a\!\sim\!K_2\!\gg\!K_b,
			\\
			K_1^{2s}K_a^\frac{1}{2}
			&\text{if }
			K_b\!\sim\!K_2\!\gg\!K_a.
		\end{cases}
	\end{align*}
\end{proof}

\begin{lem}\label{l_ee_iii_qua_bnd_iii}
	Let $s>\frac{1}{2}$. Then, for all sufficiently small $\theta=\theta(s)>1$ and $\delta\in(0,\frac{1}{2})$, we have
	\begin{align*}
		|E_{I\!V,9}|
		\lesssim_\delta
		T^\delta\|v\|_{\spfc^s_T}
		(
		\|v\|_{\spfc^s_T}\|u_2\|_{\spfc^s_T}\|b\|_{\spbc^s_T}
		+
		\|v\|_{\spfc^z_T}\|u_2\|_{\spfc^{s+1}_T}\|b\|_{\spbc^{s+1}_T}
		).
	\end{align*}
\end{lem}
\begin{proof}
	We write $E_{I\!V,9}=S_1+S_2+S_3$ with
	\begin{align*}
		S_j
		:=
		\begin{cases}
			E_{I\!V,9}|_{K_2\sim K_3\gg K_1}
			&\text{if }
			j=1,
			\\
			E_{I\!V,9}|_{K_1\sim K_3\gg K_2}
			&\text{if }
			j=2,
			\\
			E_{I\!V,9}|_{K_1\sim K_2\gtrsim K_3}
			&\text{if }
			j=3.
		\end{cases}
	\end{align*}
	Using $\widehat{(wb)}=(\hat{w}\ast\hat{b})$ and localizing these two factors separately, we obtain
	\begin{align*}
		|S_j|
		\lesssim
		\sum_{K_j}K_1^{2s}\sup_t\int_0^t\int_{\Gamma_4}\frac{\sigma_{I\!I}}{\Omega_3}&(\xi_1,\xi_2,\xi_{ab})(i\xi_{ab})\chi_{K_a}(\xi_a)\chi_{K_b}(\xi_b)
		\\
		&
		\times\hat{v}(\xi_1)\hat{v}(\xi_2)\hat{u}_2(\xi_a)\hat{b}(\xi_b)d\gamma_4(\xi_1,\xi_2,\xi_a,\xi_b)dr
	\end{align*}
	for $j\in[3]$. Define
	\begin{align*}
		\Pi(K_1,\!K_2,\!K_a,\!K_b)
		:=
		\|P_{K_1}v\|_{\spfb^{K_1}_T}
		\|P_{K_2}v\|_{\spfb^{K_2}_T}
		\|P_{2,K_a}u_2\|_{\spfb^{K_a}_T}
		\|P_{K_b}b\|_{\spfb^{K_b}_T}.
	\end{align*}
	
	First, we consider $S_1$. Here, we have $|\frac{\sigma_{I\!I}}{\Omega_3}(\xi_1,\xi_2,\xi_{ab})(i\xi_{ab})|\lesssim 1$. Then, Lemma \ref{l_stqe_bnd} leads to
	\begin{align*}
		|S_1|
		\lesssim_\delta
		T^\delta
		\sum_{K_j}\Pi(K_1,\!K_2,\!K_a,\!K_b)\times
		\begin{cases}
			K_1^{2s+\frac{1}{2}}
			&\text{if }
			K_a\!\sim\!K_b\!\gtrsim\!K_3,
			\\
			K_2^\frac{1}{2}K_1^{2s}
			&\text{if }
			K_3\!\sim\!K_a\!\gg\!K_b,
			\\
			K_1^{2s}\min\{K_2,K_a\}^\frac{1}{2}
			&\text{if }
			K_3\!\sim\!K_b\!\gg\!K_a.
		\end{cases}
	\end{align*}
	
	For $S_2$, we get the bound $|\frac{\sigma_{I\!I}}{\Omega_3}(\xi_1,\xi_2,\xi_{ab})(i\xi_{ab})|\lesssim K_1K_2^{-1}$. Together with Lemma \ref{l_stqe_bnd}, we get
	\begin{align*}
		|S_2|
		\lesssim_\delta
		T^\delta
		\sum_{K_j}\Pi(K_1,\!K_2,\!K_a,\!K_b)\times
		\begin{cases}
			K_1^{2s+1}K_2^{-\frac{1}{2}}
			&\text{if }
			K_a\!\sim\!K_b\!\gtrsim\!K_3,
			\\
			K_1^{2s+1}K_2^{-\frac{1}{2}}
			&\text{if }
			K_3\!\sim\!K_a\!\gg\!K_b,
			\\
			K_1^{2s+1}K_2^{-1}\min\{K_2,K_a\}^\frac{1}{2}
			&\text{if }
			K_3\!\sim\!K_b\!\gg\!K_a.
		\end{cases}
	\end{align*}
	
	Lastly, we consider $S_3$. Here, we have $|\frac{\sigma_{I\!I}}{\Omega_3}(\xi_1,\xi_2,\xi_{ab})(i\xi_{ab})|\lesssim K_3K_1^{-1}$. Lemma \ref{l_stqe_bnd} leads to
	\begin{align*}
		|S_3|
		\lesssim_\delta
		T^\delta
		\sum_{K_j}\Pi(K_1,\!K_2,\!K_a,\!K_b)\times
		\begin{cases}
			K_1^{2s-\frac{1}{2}}K_3
			&\text{if }
			K_a\!\sim\!K_b\!\gg\!K_1,
			\\
			K_1^{2s-1}K_a^\frac{1}{2}K_3
			&\text{if }
			K_1\!\gtrsim\!K_a\!\sim\!K_b\!\gtrsim\!K_3,
			\\
			K_1^{2s-1}K_a^\frac{1}{2}K_3
			&\text{if }
			K_a\!\sim\!K_3\!\gg\!K_b,
			\\
			K_1^{2s-1}K_a^\frac{1}{2}K_3
			&\text{if }
			K_b\!\sim\!K_3\!\gg\!K_a.
		\end{cases}
	\end{align*}
\end{proof}

Finally, we can prove Lemma \ref{l_ee_iii}. As shown in the beginning of this section, we have
\begin{align*}
	\|v\|_{\spec^s_T}^2
	-
	\|v_0\|_{H^s}^2
	\lesssim
	I\!I\!I
	+
	I\!V
	+
	V.
\end{align*}
The term $V$ is bounded in Lemma \ref{l_ee_iii_trv}. The estimate for $I\!I\!I$ follows from Lemmata \ref{l_ee_iii_re} and \ref{l_ee_iii_ne}, whereas the estimate for $I\!V$ follows from Lemmata \ref{l_ee_iii_tri}, \ref{l_ee_iii_qua_i}, \ref{l_ee_iii_qua_ii}, \ref{l_ee_iii_qua_iii}, \ref{l_ee_iii_qua_bnd_i}, \ref{l_ee_iii_qua_bnd_ii}, and \ref{l_ee_iii_qua_bnd_iii}.
\section{Proof of Theorem \ref{t_main}}\label{s_p}
	
	In this section we prove Theorem \ref{t_main} in three steps. First, we recall the existence of classical solutions to equation \eqref{eq_bo_split} in the case that the initial datum $(u_0,b,f)$ is sufficiently regular.
	Second, we bootstrap the estimates derived in Sections \ref{s_stbe} and \ref{s_ee} in order to derive three à-priori estimates for these solutions respectively differences of them. Third, we perform the classical Bona-Smith argument to prove that the data-to-solution map can be extended continuously, thus, completing the proof of Theorem \ref{t_main}.
\subsection{Local well-posedness for regular initial data}\label{ss_lwp_highreg}
	
	In \cite{Kat1975}, Kato obtained a local well-posedness result for a large class of quasilinear evolution equations. Specified to \eqref{eq_bo_split}, his result implies the following:
	\begin{pro}\label{p_highreg}
		Let $\varepsilon>0$ and define
		\begin{align*}
			\begin{cases}
				\mathcal{X}
				&:=
				L^\infty([0,1];B^{3+\varepsilon}_{\infty,\infty}(\mathbb{R}))
				\cap
				C^0([0,1];C^2(\mathbb{R})),
				\\
				\mathcal{Y}'
				&:=
				L^\infty([0,1]; H^{2+\varepsilon}(\mathbb{R}))
				\cap
				C^0([0,1];H^0(\mathbb{R})).
			\end{cases}
		\end{align*}
		Then, the initial value problem \eqref{eq_bo_split} is locally well-posed in $H^2(\mathbb{R})\times\mathcal{X}\times\mathcal{Y}'$. By this, we mean that there exists a non-increasing positive function $T=T(R)$,
		denoting the minimal time of existence, such that for each $(u_0,b,f)\in B_{R}(0)\subset H^2(\mathbb{R})\times\mathcal{X}\times\mathcal{Y}'$ there exists a unique strong solution $u$ of \eqref{eq_bo_split} satisfying
		\begin{align*}
			u\in C^0([0,T];H^2(\mathbb{R}))\cap C^1([0,T];H^0(\mathbb{R})).
		\end{align*}
		Moreover, the map $\mathcal{L}^2_{T}:B_R(0)\rightarrow C([0,T];H^2(\mathbb{R})), u_0\mapsto u$ is continuous.
	\end{pro}
	Note that the statement remains valid, if we replace $\mathcal{Y}'$ by $\mathcal{Y}$ (as defined in \ref{as_b_f}). In fact, we need this stronger assumption on $f$ only once in Section \ref{ss_lwp_lowreg}.
	\begin{rem}\label{r_as_bf}
		The results from \cite{Kat1975} prove local well-posedness of \eqref{eq_bo} for initial data $u_0$ in $H^s(\mathbb{R})$, $s>\frac{3}{2}$. Here, the lower bound on the regularity threshold is directly linked to the commutator estimate required to bound the nonlinear term $(uu)_x$. Moreover, we can also obtain local well-posedness of \eqref{eq_bo_split} for initial data $u_0$ in $H^s(\mathbb{R})$, $s>\frac{3}{2}$. However, in this case the commutator estimate for the most problematic term $bu_x$ takes the form
		\begin{align*}
			\|[J^{-s},b\partial_x]J^s\|_{L^2\rightarrow L^2}
			\lesssim
			\|b_{x}\|_{H^{s-1}}+\|b_x\|_{L^\infty},
		\end{align*}
		where $J^{s}=(I-\Delta)^{-s/2}$. Thus, we cannot consider periodic background functions $b$, but we can consider Zhidkov functions. To overcome this restriction on $b$, instead, we use the trivial estimate
		\begin{align*}
			\|[J^{-2},b\partial_x]J^2\|_{L^2\rightarrow L^2}
			\lesssim
			\|b_{xx}\|_{L^\infty}+\|b_x\|_{L^\infty},
		\end{align*}
		which leads to local well-posedness as stated in the Proposition above.
	\end{rem}
	\begin{rem}
		It is likely that Proposition \ref{p_highreg} (i.e.\@ the assumptions on $u_0$, $b$, and $f$) can be improved by using a better commutator estimate or by restricting to a smaller/ different class of background functions $b$. In that case, however, we can only deduce an improvement of Theorem \ref{t_main}, if we also improve/ adapt the estimates in Section \ref{s_ee}, in particular Lemma \ref{l_ee_trv_uub}.
	\end{rem}
\subsection{Low regularity estimates for regular solutions}\label{ss_bta}
 	
 	The goal of this section consists of proving three estimates at different regularities for the strong solutions to \eqref{eq_bo_split} given by Proposition \ref{p_highreg}. Most of the preparatory work has been carried out in Sections \ref{s_stbe} and \ref{s_ee}. Here, it only remains to perform the bootstrap arguments.
	\begin{lem}\label{l_ap}
		Let $2> s\geq s_0>\frac{1}{4}$ and $\varepsilon>0$. Choose $b\in\mathcal{X}$ and $f\in\mathcal{Y}$.
		Then, there exists a non-increasing positive function
		\begin{align*}
			T
			=
			T(R;\|b\|_{L^\infty B^{s+1+\varepsilon}_{\infty,\infty}};\|f\|_{L^\infty H^{s+\varepsilon}})
		\end{align*}
		such that for each $u_0\in B_R(0)\subset H^{s_0}(\mathbb{R})$ with $u_0\in H^2(\mathbb{R})$ the strong solution $u=\mathcal{L}^2_{T}(u_0,b,f)$ to \eqref{eq_bo_split} satisfies the à-priori estimate
		\begin{equation}\label{e_ap}
			\|u\|_{\spec^s_T}
			\!\lesssim\!
			\|u_0\|_{H^s}
			\!+\!
			\|f\|_{L^\infty_T H^{s+\varepsilon}}.
		\end{equation}
	\end{lem}
	\begin{rem}\label{r_highreg_t}
		As a consequence of Lemma \ref{l_ap}, we obtain that the data-to-solution map $u_0\mapsto\mathcal{L}^2_{T}(u_0,b,f)$ can be extended to a time interval $[0,T]$, where the minimal time of existence $T$ depends on the $H^{s_0}(\mathbb{R})$-norm of $u_0$, but no longer on the $H^2(\mathbb{R})$-norm.
	\end{rem}
	\begin{lem}\label{l_de}
		Let $2>s>\frac{1}{2}$, $z=-\frac{1}{2}$, and $\varepsilon>0$. Choose $b\in\mathcal{X}$ and $f\in\mathcal{Y}$.
		Then, there exists a non-increasing positive function
		\begin{align*}
			T
			=
			T(R;\|b\|_{L^\infty B^{3+\varepsilon}_{\infty,\infty}};\|f\|_{L^\infty H^{2+\varepsilon}})
		\end{align*}
		such that for all $u_{1,0}$, $u_{0,2}\in B_R(0)\subset H^s(\mathbb{R})$ with $u_{0,1}, u_{2,0}\in H^2(\mathbb{R})$ the strong solutions $u_1=\mathcal{L}^2_T(u_{1,0},b,f)$ and $u_2=\mathcal{L}^2_T(u_{2,0},b,f)$ of \eqref{eq_bo_split} satisfy the weak Lipschitz estimate
		\begin{equation}\label{e_de}
			\|u_1\!-\!u_2\|_{\spec^z_T}
			\!\lesssim\!
			\|u_{1,0}\!-\!u_{2,0}\|_{H^z}.
		\end{equation}
	\end{lem}
	\begin{lem}\label{l_de1}
		Let $2>s>\frac{1}{2}$, $z=-\frac{1}{2}$, and $\varepsilon>0$. Choose $b\in\mathcal{X}$ and $f\in\mathcal{Y}$.
		Then, there exists a non-increasing positive function
		\begin{align*}
			T
			=
			T(R;\|b\|_{L^\infty B^{3+\varepsilon}_{\infty,\infty}};\|f\|_{L^\infty H^{2+\varepsilon}})
		\end{align*}
		such that for all $u_{1,0}$, $u_{2,0}\in B_R(0)\subset H^s(\mathbb{R})$ with $u_{1,0}, u_{2,0}\in H^2(\mathbb{R})$ the strong solutions $u_1=\mathcal{L}^2_T(u_{1,0},b,f)$ and $u_2=\mathcal{L}^2_T(u_{2,0},b,f)$ of \eqref{eq_bo_split} satisfy the estimate
		\begin{equation}\label{e_de1}
			\|u_1\!-\!u_2\|_{\spec^s_T}
			\!\lesssim\!
			\|u_{1,0}\!-\!u_{2,0}\|_{H^s}
			\!+\!
			(\|u_1\!-\!u_2\|_{\spec^z_T}
			\|u_1\!-\!u_2\|_{\spec^s_T}
			\|u_2\|_{\spec^{s+1}_T}
			)^\frac{1}{2}.
		\end{equation}
	\end{lem}
	\begin{proof}[Proof of Lemma \ref{l_ap}]
		According to \eqref{es_lin}, we have for $T\in(0,1)$
		\begin{align*}
			\|u\|_{\spfc^s_T}
			\lesssim
			\|u\|_{\spec^s_T}
			+
			\|(uu)_x\|_{\spnc^s_T}
			+
			\|(ub)_x\|_{\spnc^s_T}
			+
			T^\frac{1}{2}\|f\|_{L^\infty H^{s}}.
		\end{align*}
		Define $X_s:=X_s(T):=\max\{\|u\|_{\spec^s_T},\|(uu)_x\|_{\spnc^s_T},\|(ub)_x\|_{\spnc^s_T}\}$.
		Using Proposition \ref{p_cnv}, we get $\lim_{T\rightarrow 0} X_s(T)\lesssim \|u_0\|_{H^s}$, but we do not have any control on the speed of convergence.
		
		In this proof, we use the abbreviations
		\begin{align*}
			\|b\|
			=
			\|b\|_{\spbc^{s+1+\varepsilon}_T}
			\qquad\text{and}\qquad
			\|f\|
			=
			\|f\|_{L^\infty H^{s+\varepsilon}}.
		\end{align*}
		Now, we bound the first three terms on the right-hand side in the inequality above. By applying the estimates in \eqref{e_nl_uu}, \eqref{e_nl_ub}, and \eqref{e_ee_i}, we arrive at
		\begin{align*}
			X_s^2
			\lesssim
			\|u_0\|_{H^s}^2
			&+
			X_s^2
			(
			(T+TN+N^{-\zeta})X_{s_0}
			+
			T\|b\|
			+
			T\|f\|
			+
			T
			)
			\\&+
			T^\delta\|u\|_{\spfc^s_T}^2
			(
			\|u\|_{\spfc^{s_0}_T}^2
			+
			\|u\|_{\spfc^{s_0}_T}\|b\|
			+
			\|b\|^2
			)
			+
			T\|f\|^2
		\end{align*}
		for all $N\in\mathbb{D}$ and all sufficiently small $\delta,\zeta>0$.
		Above, the first two summands are already sufficient for the bootstrap argument. To bound the last two summand, we use the estimates $X_{s_0}\leq X_s$ and $\|u\|_{\spfc^s_T}\lesssim X_s+T^\frac{1}{2}\|f\|$ repetitively leading to
		\begin{align*}
			&T^\delta\|u\|_{\spfc^s_T}^2
			(
			\|u\|_{\spfc^{s_0}_T}^2
			+
			\|u\|_{\spfc^{s_0}_T}\|b\|
			+
			\|b\|^2
			)
			\\&\qquad\lesssim
			X_s^2X_{s_0}^2T^\delta
			\\&\qquad+
			X_s^2X_{s_0}
			(
			T^{\frac{1}{2}+\delta}\|f\|
			+
			T^\delta\|b\|
			)
			\\&\qquad+
			X_s^2
			(
			T^{1+\delta}\|f\|^2
			+
			T^{\frac{1}{2}+\delta}\|f\|\|b\|
			+
			T^{\delta}\|b\|^2
			)
			\\&\qquad+
			X_s^1
			(
			T^{\frac{3}{2}+\delta}\|f\|^3
			+
			T^{1+\delta}\|f\|^2\|b\|
			+
			T^{\frac{1}{2}+\delta}\|f\|\|b\|^2
			)
			\\&\qquad+
			(
			T^{2+\delta}\|f\|^4
			+
			T^{\frac{3}{2}+\delta}\|f\|^3\|b\|
			+
			T^{1+\delta}\|f\|^2\|b\|^2
			).
		\end{align*}
		Next, we apply Cauchy-Schwarz to the product in the second to last summand. Combining the previous two estimates, we conclude
		\begin{align*}
			X_s^2
			\lesssim
			\|u_0\|_{H^s}^2
			&+
			X_s^2X_{s_0}^2 T^\delta
			\\&+
			X_s^2X_{s_0}
			(
			T
			+
			TN
			+
			N^{-\zeta}
			+
			T^{\frac{1}{2}+\delta}\|f\|
			+
			T^\delta\|b\|
			)
			\\&+
			X_s^2
			(
			T\|b\|
			+
			T^{\delta}\|b\|^2
			+
			T^{\frac{1}{2}+\delta}\|f\|\|b\|
			+
			T^{1+\delta}\|f\|^2
			+
			T\|f\|
			+
			T
			)
			\\&+
			\|f\|^2
			(
			T^{2+2\delta}\|f\|^4
			+
			T^{1+2\delta}\|f\|^2\|b\|^2
			+
			T^{2\delta}\|b\|^4
			)
			\\&+
			\|f\|^2
			(
			T^{2+\delta}\|f\|^2
			+
			T^{\frac{3}{2}+\delta}\|f\|\|b\|
			+
			T^{1+\delta}\|b\|^2
			+
			T
			).
		\end{align*}
		
		Now, we consider the special case $s=s_0$. For $i\in[4]$, we denote by $c_i$ the implicit constant in front of $X_{s_0}^i$ and, similarly, $c_0$ denotes the implicit constant in front of $\|f\|^2$. We obtain
		\begin{align*}
			X_{s_0}^2
			\leq
			\|u_0\|_{H^{s_0}}^2
			+
			c_4X_{s_0}^4
			+
			c_3X_{s_0}^3
			+
			c_2X_{s_0}^2
			+
			c_0\|f\|_{L^\infty H^{s_0+\varepsilon}}^2.
		\end{align*}
		Since the mapping $T\mapsto X_{s_0}$ is continuous and due to the fact that each $c_i$ can be made arbitrarily small by choosing $N$ large and $T$ small, a bootstrap argument implies
		\begin{align*}
			X_{s_0}
			\lesssim
			\|u_0\|_{H^{s_0}}
			+
			\|f\|_{L^\infty H^{s_0+\varepsilon}},
		\end{align*}
		where $T=T(\|u_0\|_{H^{s_0}};\|b\|_{\spbc^{s_0+1+\varepsilon}};\|f\|_{L^\infty H^{s_0+\varepsilon}})$ is a non-increasing positive function.
		
		Next, we consider the case $s\in(s_0,2)$. Then, applying the previously proved estimate for $X_{s_0}$ and denoting the implicit constants by $\tilde{c}_i$, we get
		\begin{align*}
			X_s^2
			\leq
			\|u_0\|_{H^s}^2
			+
			X_s^2\left(
			\tilde{c}_4\|u_0\|_{H^{s_0}}^2
			+
			\tilde{c}_32\|u_0\|_{H^{s_0}}
			+
			\tilde{c}_2
			\right)
			+
			\tilde{c}_0\|f\|_{L^\infty H^{s+\varepsilon}}^2.
		\end{align*}
		Again, a continuity argument implies
		\begin{align*}
			X_s(T)
			\lesssim
			\|u_0\|_{H^s}
			+
			\|f\|_{L^\infty H^{s+\varepsilon}}
		\end{align*}
		for a non-increasing positive function $T=T(\|u_0\|_{H^{s_0}};\|b\|_{\spbc^{s+1+\varepsilon}_T};\|f\|_{L^\infty H^{s+\varepsilon}})$.
	\end{proof}
		
	\begin{proof}[Proof of Lemma \ref{l_de}]
		We use the notation $v=u_1-u_2$ and $w=u_1+u_2$. Then, the linear estimate \eqref{es_lin} takes the form
		\begin{align*}
			\|v\|_{\spfc^z_T}
			\lesssim
			\|v\|_{\spec^z_T}
			+
			\|(vw)_x\|_{\spnc^z_T}
			+
			\|(vb)_x\|_{\spnc^z_T}.
		\end{align*}
		Define $Y:=Y_z(T):=\max\{\|v\|_{\spec^z_T},\|(vw)_x\|_{\spnc^z_T},\|(vb)_x\|_{\spnc^z_T}\}$. Again, we abbreviate $\|b\|=\|b\|_{\spbc^{s+1+\varepsilon}_T}$.
		The estimates in \eqref{e_nl_vw}, \eqref{e_nl_vb}, and \eqref{e_ee_ii} yield
		\begin{align*}
			Y^2
			\lesssim
			\|v_0\|_{H^z}^2
			&+
			Y^2
			(
			(T+TN^2+N^{-\zeta})
			(\|u_1\|_{\spec^s_T}+\|u_2\|_{\spec^s_T})
			)
			\\&+
			Y^2
			(
			T^\delta\|b\|^2
			+
			T\|b\|
			+
			T\|f\|
			)
			\\&+
			Y^2
			(
			T^\delta
			(
			\|u_1\|_{\spfc^s_T}
			+
			\|u_2\|_{\spfc^s_T}
			)
			(
			\|u_1\|_{\spfc^s_T}
			+
			\|u_2\|_{\spfc^s_T}
			+
			\|b\|
			)
			).
		\end{align*}
		The à-priori estimate \eqref{e_ap} combined with a bootstrap argument leads to
		\begin{align*}
			\|v\|_{\spec^z_T}
			\lesssim
			Y(T)
			\lesssim
			\|v_0\|_{H^z}
		\end{align*}
		for a non-increasing positive function $T=T(\|u_0\|_{H^s},\|b\|_{L^\infty B^{3+\varepsilon}_{\infty,\infty}},\|f\|_{L^\infty H^{2+\varepsilon}})$.
	\end{proof}
	
	\begin{proof}[Proof of Lemma \ref{l_de1}]
		We use the notation $v=u_1-u_2$ and $w=u_1+u_2$. Then, the linear estimate \eqref{es_lin} takes the form
		\begin{align*}
			\|v\|_{\spfc^s_T}
			\lesssim
			\|v\|_{\spec^s_T}
			+
			\|(vw)_x\|_{\spnc^s_T}
			+
			\|(vb)_x\|_{\spnc^s_T}.
		\end{align*}
		Define $Z:=Z_s(T):=\max\{\|v\|_{\spec^s_T},\|(vw)_x\|_{\spnc^s_T},\|(vb)_x\|_{\spnc^s_T}\}$. Again, we abbreviate $\|b\|=\|b\|_{\spbc^{s+1+\varepsilon}_T}$.
		The estimates in \eqref{e_nl_vw}, \eqref{e_nl_vb}, and \eqref{e_ee_ii} yield
		\begin{align*}
			Z^2
			\lesssim
			\|v_0\|_{H^s}^2
			&+
			Z^2
			(T+TN^2+N^{-\zeta})
			(\|u_1\|_{\spec^s_T}+\|u_2\|_{\spec^s_T})
			\\&+
			Z^2
			(
			T^\delta\|b\|^2
			+
			T\|b\|
			+
			T\|f\|
			)
			\\&+
			Z^2
			T^\delta
			(
			\|u_1\|_{\spfc^{s}_T}
			+
			\|u_2\|_{\spfc^{s}_T}
			)
			(
			\|u_1\|_{\spfc^{s}_T}
			+
			\|u_2\|_{\spfc^{s}_T}
			+
			\|b\|
			)
			\\&+
			\|v\|_{\spec^z_T}\|v\|_{\spec^s_T}\|u_2\|_{\spec^{s+1}_T}
			(
			T
			+
			T^\delta
			(
			\|u_1\|_{\spfc^s_T}
			+
			\|u_2\|_{\spfc^s_T}
			+
			\|b\|
			)
			).
		\end{align*}
		Again, we apply the à-priori estimate \eqref{e_ap} and perform a bootstrap argument leading to
		\begin{align*}
			\|v\|_{\spec^s_T}
			\lesssim
			Z(T)
			\lesssim
			\|v_0\|_{H^s}
			+
			(\|v\|_{\spec^z_T}\|v\|_{\spec^s_T}\|u_2\|_{\spec^{s+1}_T})^\frac{1}{2}
		\end{align*}
		for a non-increasing positive function $T=T(\|u_0\|_{H^s},\|b\|_{L^\infty B^{3+\varepsilon}_{\infty,\infty}},\|f\|_{L^\infty H^{2+\varepsilon}})$.
	\end{proof}
\subsection{Local well-posedness for rough initial data}\label{ss_lwp_lowreg}
	
	In this section we prove the extendability of the map $\mathcal{L}^2_T$ to $\mathcal{L}^s_T$ for each $s\in(\frac{1}{2},2)$ and thus complete the proof of Theorem \ref{t_main}.
	
	Fix $R>0$ and let $\frac{1}{2}<s<2$.  Proposition \ref{p_highreg} together with Remark \ref{r_highreg_t} yield the existence of a continuous map $\mathcal{L}^2_T:\tilde{B}_R(0) \subset H^2(\mathbb{R})\times\mathcal{X}\times\mathcal{Y}$, where the domain is given by
	\begin{align*}
		\tilde{B}_R(0)
		=
		\{
		(u_0,b,f)\in H^2(\mathbb{R})\times \mathcal{X}\times\mathcal{Y}
		:
		\|u_0\|_{H^s}
		+
		\|b\|_{\spbc_T^{3+\varepsilon}}
		+
		\|f\|_{L^\infty H^{2+\varepsilon}}
		<
		R
		\}.
	\end{align*}
	Now, we show that $\mathcal{L}^2_T$ extends continuously to $B_R(0)\subset H^s(\mathbb{R})\times\mathcal{X}\times\mathcal{Y}$ (with $B_R(0)$ being the usual open ball with radius $R$).
	
	We begin by extending $\mathcal{L}^2_T$. Fix $(u_0,b,f)\in B_R(0)$. Then, for each dyadic $N$ we have $(P_{\leq N}u_0,b,f)\in \tilde{B}_R(0)$. Thus, for dyadic $K\geq N$, the functions $u_1=\mathcal{L}^2_T(P_{\leq K}u_0,b,f)$ and $u_2=\mathcal{L}^2_T(P_{\leq N}u_0,b,f)$ are well-defined. Now, \eqref{e_ap} yields
	\begin{align*}
		\|u_2\|_{\spec_T^{s+1}}
		\lesssim
		\|P_{\leq N}u_0\|_{H^{s+1}}
		+
		\|f\|_{L^\infty H^{s+1}}
		\lesssim
		N\|P_{\leq N}u_0\|_{H^s}
		+
		\|f\|_{L^\infty H^{s+1}}
		\leq
		NR
	\end{align*}
	and \eqref{e_de} implies
	\begin{align*}
		\|u_1-u_2\|_{\spec_T^z}
		\lesssim
		\|P_{>N}u_0\|_{L^\infty H^z}
		\lesssim
		N^{z-s}\|P_{>N}u_0\|_{L^\infty H^s}
	\end{align*}
	for $z=-\frac{1}{2}$. (Here, the first estimate forces us to assume that $\|f\|_{L^\infty H^{s+1}}$ is finite and thus, we need to work with $f\in\mathcal{Y}$ and not just $f\in\mathcal{Y}'$.)
	Using \eqref{e_de1} and Cauchy-Schwarz, we get
	\begin{align*}
		\|u_1-u_2\|_{\spec_T^s}
		&\lesssim
		\|P_{>N}u_0\|_{L^\infty H^s}
		+
		(
		\|u_1-u_2\|_{\spec^z_T}
		\|u_1-u_2\|_{\spec^s_T}
		\|u_2\|_{\spec^{s+1}_T}
		)^\frac{1}{2}
		\\&\lesssim
		\|P_{>N}u_0\|_{L^\infty H^s}
		+
		R\|P_{>N}u_0\|_{L^\infty H^s}
		+
		N^{1+z-s}
		\|u_1-u_2\|_{\spec_T^s}.
	\end{align*}
	Next, we reorder the last estimate and arrive at
	\begin{align*}
		\|u_1-u_2\|_{L^\infty H^s}
		\lesssim
		\|u_1-u_2\|_{\spec^s_T}
		\lesssim
		(1+R)(1-N^{1+z-s})^{-1}\|P_{>N}u_0\|_{L^\infty H^s}.
	\end{align*}
	Due to $1+z<s$, we obtain for all $K\geq N=N(s,z)$ sufficiently large that
	\begin{align}\label{e_de2}
		\|u_1-u_2\|_{L^\infty H^s}
		\lesssim_R
		\|P_{>N}u_0\|_{L^\infty H^s}
	\end{align}
	holds. Hence, $(\mathcal{L}^2_{T}(P_{\leq N}u_0),b,f)_{N\in\mathbb{D}}$ is a Cauchy sequence in $C([0,T];H^s(\mathbb{R}))$. In particular, this sequence is convergent and by setting
	\begin{align*}
		\mathcal{L}^s_T(u_0,b,f)
		:=
		\lim_N \mathcal{L}^2_{T}(P_{\leq N}u_0,b,f)
	\end{align*}
	we obtain a well-defined function $\mathcal{L}^s_T$ on $B_R(0)$. Moreover, if $(u_0,b,f)\in \tilde{B}_R(0)$, then the continuity of $\mathcal{L}^2_T$ implies $\mathcal{L}^s_T(u_0,b,f)=\lim_N \mathcal{L}^2_{T}(P_{\leq N}u_0,b,f)=\mathcal{L}^2_T(u_0,b,f)$ and, hence, $\mathcal{L}^s_T$ is an extension of $\mathcal{L}^2_T$.
	
	Now, we prove that $\mathcal{L}^s_T$ is continuous. Let $((u_{n,0},b,f))_{n\in\mathbb{N}}\subset B_R(0)$ be a convergent sequence with limit $(u_{\infty,0},b,f)\in B_R(0)$.
	For all $l,m\in\mathbb{N}$ and all $N\in\mathbb{D}$, we have
	\begin{align*}
		\|\mathcal{L}^s_{T}(u_{l,0})-\mathcal{L}^s_{T}(u_{m,0})\|_{L^\infty H^s}
		\leq
		&\|\mathcal{L}^s_{T}(u_{l,0})-\mathcal{L}^2_{T}(P_{\leq N}u_{l,0})\|_{L^\infty H^s}
		\\+
		&\|\mathcal{L}^s_{T}(u_{m,0})-\mathcal{L}^2_{T}(P_{\leq N}u_{m,0})\|_{L^\infty H^s}
		\\+
		&\|\mathcal{L}^2_{T}(P_{\leq N}u_{l,0})-\mathcal{L}^2_{T}(P_{\leq N}u_{m,0})\|_{L^\infty H^s}.
	\end{align*}
	The first two terms on the right-hand side are bounded by using \eqref{e_de2} (for $K=\infty$) and another triangle inequality; we get
	\begin{align*}
		\|\mathcal{L}^s_{T}(u_{l,0})-\mathcal{L}^s_{T}(u_{m,0})\|_{L^\infty H^s}
		&\lesssim
		\|P_{> N}u_{\infty,0}\|_{H^s}
		\\&+
		\|P_{> N}(u_{l,0}-u_{\infty,0})\|_{H^s}
		+
		\|P_{> N}(u_{m,0}-u_{\infty,0})\|_{H^s}
		\\&+
		\|\mathcal{L}^2_{T}(P_{\leq N}u_{l,0})-\mathcal{L}^2_{T}(P_{\leq N}u_{m,0})\|_{L^\infty H^s}.
	\end{align*}
	Let $\varepsilon>0$. Then, there exists $N=N(u_{\infty,0})$ with $\|P_{> N}u_{\infty,0}\|_{L^\infty H^s}\leq\frac{\varepsilon}{4}$. The sequence $(P_{>N}u_{k,0})_{k\in\mathbb{N}}$ converges to $P_{>N} u_{\infty,0}$ in $H^s(\mathbb{R})$, hence, there exists $L_1>0$ such that for all $k,m\geq L_1$ the terms in the second line are bounded by $\frac{\varepsilon}{2}$. Lastly, since $(P_{\leq N}u_{k,0})_{k\in\mathbb{N}}$ converges to $P_{\leq N} u_{\infty,0}$ in $H^2(\mathbb{R})$ and since the map $\mathcal{L}^2_T$ is continuous, there exists $L_2>0$ such that for all $k,m\geq L_2$ the term in the third line is bounded by $\frac{\varepsilon}{4}$. Thus, $(\mathcal{L}^s_{T}(u_{k,0}))_{k\in\mathbb{N}}$ is a Cauchy sequence in $C([0,T];H^s(\mathbb{R}))$ and in particular $\mathcal{L}^s_T$ is continuous.
\section*{Acknowledgments}
	The author thanks Sebastian Herr for helpful discussions. Funded by the Deutsche Forschungsgemeinschaft (DFG, German Research
	Foundation) - IRTG 2235 - Project number 282638148.

\end{document}